\documentclass[11pt]{amsart}
\usepackage{amsmath}
\usepackage{amssymb}
\usepackage{amsfonts}
\usepackage{amsthm}
\usepackage{verbatim}
\usepackage{amscd}
\usepackage{cite}
\usepackage{leftidx}
\usepackage{enumerate}
\usepackage{txfonts}
\usepackage{manfnt}
\usepackage{amscd}
\usepackage{mathpazo}
\usepackage[mathscr]{eucal}
\usepackage{xr}
\usepackage{hyperref}
 \DeclareMathOperator{\Zeros}{Zeros}
  \DeclareMathOperator{\zeros}{Zeros}

\DeclareMathOperator{\tor}{Tor}
 
 \DeclareMathOperator{\Res}{Res}
 \DeclareMathOperator{\gm}{\mathbb G_m}
\def\norm{{\mathrm{norm}}}

\DeclareMathOperator{\Bl}{\mathcal {B}\l}
 \DeclareMathOperator{\bl}{b\l}
\def\inv{^{-1}}

\def\sepp{{\mathrm{sepp}}}
\def\ssep{^{\mathrm{sep}}}

\DeclareMathOperator{\mgbar}{\overline\M_g}
\DeclareMathOperator{\del}{\partial}
 \DeclareMathOperator{\Hom}{Hom}
\DeclareMathOperator{\coker}{coker}

\DeclareMathOperator{\bD}{\mathbb D}

\def\refp #1.{(\ref{#1})}

\newcommand{\A}{\mathcal{A}}

\newcommand{\M}{\mathcal{M}}

\newcommand{\bE}{\mathbb{E}}

\newcommand{\sE}{{\mathrm s\mathbb E}}
\newcommand{\aE}{{\leftidx{_\mathfrak a}{\mathbb E}{}}}

\newcommand{\lbl}[1]{\label{#1}}

\newcommand{\Cal}[1]{\mathcal #1}

\newcommand{\ul}[1]{\underline {#1}}

\def\sbr #1.{^{[#1]}}
\def\sfl #1.{^{\lfloor #1\rfloor}}

\newcommand{\myit}[1]{\emph{\ {#1}\ }}
\newcommand{\mymax}{{\max}}
\newcommand{\mydiamond}[4]
{\begin{matrix}&& #1&&\\ &\swarrow&&\searrow&\\ #2&&&& #3\\
&\searrow&&\swarrow&     \\&& #4&&
\end{matrix}
}
\newcommand{\mycd}[4]{
\begin{matrix}
#1&\to&#2\\\downarrow&\circlearrowleft&\downarrow\\
#3&\to&#4
\end{matrix}
}

\newcommand\twostack [2]
{\begin{matrix}#1 \\ #2\end{matrix}}
\newcommand{\sectionend}
{\[\circ\gtrdot\ggg\mathbb{\times}\lll\lessdot\circ\]}

\newcommand\mymatrix[2] {\begin{matrix} #1\\ #2\end{matrix}}

\newcommand\red{{\mathrm red}}
\newcommand\vir{{\mathrm{vir}}}
\newcommand\sep{{\mathrm{sep}}}
\newcommand\phel{{\rm{PHEL }}}
\newcommand\vtheta{{\underline{\theta}}}
\DeclareMathOperator{\Mod}{Mod}

\def\inv{^{-1}}
\def\?{{\bf{??}}}
\def\iso{\ isomorphism\ }

\newcommand{\bn}{{\ Brill-Noether\ }}

\def\rx{\leftidx{_\mathrm R}{X}{}}
\def\lx{{\leftidx{_\mathrm L}{X}{}}}
\def\a{{\frak a}}

\def\M{\Cal M}
\def\A{\Bbb A}

\def\C{\mathbb C}
\def\P{\mathbb P}

\def\R{\mathbb R}
\def\Z{\mathbb Z}

\def\ord{\text{\rm ord}}

\def\sym{\text{\rm Sym} }

\def\Q{\mathbb Q}

\def\L{\mathcal L}
\def\rmL{\mathrm L}

\def\O{\mathcal O}

\def\Sym{\textrm{Sym}}

\def\g{\mathfrak g}

\def\m{\mathfrak m}

\def\1/2{\frac{1}{2}}

\def\I{\mathcal{ I}}

\def\im{\text{im}}
\def\simto{\stackrel{\sim}{\rightarrow}}
\def\2{{[2]}}
\def\l{\ell}
\def\nl{\newline}

\def\he{\mathcal{HE}}
\def\hebar{\bar{\mathcal{HE}}}

\def\<{\langle}
\def\>{\rangle}

\def\im{\text{im}}
\def\2{{[2]}}
\def\l{\ell}

\def\scl #1.{^{\lceil#1\rceil}}
\def\spr #1.{^{(#1)}}
\def\sbc #1.{^{\{#1\}}}

\def\subpr#1.{_{(#1)}}

\newcommand{\aphi}{\leftidx{_\mathfrak a}\phi}

\def\beq{\begin{equation*}}
\def\eeq{\end{equation*}}

\newcommand{\sing}{{\mathrm{sing}}}
\newcommand{\lf}[1]{\leftidx{_{\ \mathrm L}}{#1}{}}
\newcommand{\rt}[1]{\leftidx{_{\ \mathrm R}}{#1}{}}
\newcommand{\lfrt}[1]{\leftidx{_{\ \mathrm LR}}{#1}{}}

\def\g3{{\Gamma\spr 3.}}

\def\ggg{{\Gamma\spr 3.}}

\def\hep{hyperelliptic\ }
\def\nhep{non-hyperelliptic\ }
\def\eva{essentially very ample\ }
\def\fkb {\frak b}

\newcommand{\md}[1]{\leftidx{_{\ \mathrm M}}{#1}{}}
\def\cJ{\mathcal J}
\def\rR{\mathrm R }
\def\rL{\mathrm L }
\def\ae{\leftidx{_{\mathfrak a}}{\mathbb E}}
\def\aE{\leftidx{_{\mathfrak a}}{\mathbb E}}

\newcommand{\eqspl}[2]{
\begin{equation}\label{#1}
\begin{split}
#2\end{split}\end{equation}}
\newcommand{\eqsp}[1]{\begin{equation*}
\begin{split}#1\end{split}\end{equation*}}

\newcommand{\exseq}[3]{
0\to #1\to #2\to #3\to 0
}
\newcommand{\beginalphaenum}{
\begin{enumerate}\renewcommand{\labelenumi}{ }
\item \begin{enumerate}
}

\def\eex{\end{rm}\end{example}}
\newcommand\newsection[1]{\section{#1}\setcounter{equation}{0}
}
\newcommand\newsubsection[1]{\subsection{#1}\setcounter{equation}{0}}

\newcommand{\be}{\mathbb E}

\newcommand\dual{\ \check{}\ }

\newtheorem{thm}{Theorem}  [section]

\newtheorem*{mainthm}{Main Theorem}
\newtheorem*{thm*}{Theorem}
\newtheorem*{prop*}{Proposition}
\newtheorem{cor}[thm]{Corollary}
\newtheorem*{cor*}{Corollary}

\newtheorem{lem}[thm]{Lemma}
\newtheorem{lem*}{Lemma}
\newtheorem{claim}[thm]{Claim}
\newtheorem*{claim*}{Claim}
\newtheorem{prop}[thm]{Proposition}
\newtheorem{propdef}[thm]{Proposition-definition}
\newtheorem{defn}[thm]{Definition}
\theoremstyle{remark}

\newtheorem{rem}[thm]{Remark}
\newtheorem{crit-rem}[thm]{Critical remark}
\newtheorem{remarks}[thm]{Remarks}

\newtheorem{example}[thm]{Example}
\newtheorem*{example*}{Example}

\newtheorem*{defn*}{Definition}

\begin{document}

\title {\large Modifications of Hodge bundles and\\ enumerative
geometry :\\ \bigskip
 the stable hyperelliptic locus}
\normalsize
\author
{Ziv Ran}
\date {\today}
\address{\tiny  {\newline Ziv Ran \newline University of California\newline
Mathematics Department,  Big Springs Rd. Surge Facility
\newline Riverside CA 92521 US
\newline ziv.ran @ ucr.edu}}

\begin{abstract}
We study the stable hyperelliptic locus, i.e. the closure, in the Deligne- Mumford Moduli
 of stable curves, of the locus of smooth hyperelliptic curves. Working on a suitable blowup of the relative Hilbert
scheme (of degree 2) associated to a family of stable curves,
we construct a bundle map ('degree-2 Brill-Noether') from
a modification of the Hodge bundle to a tautological bundle,
whose degeneracy locus is the natural lift of the stable hyperelliptic locus
plus a simple residual scheme. Using intersection theory on Hilbert schemes and Fulton-MacPherson residual intersection theory, the class of the structure sheaf and various
 other sheaves supported on the stable hyperelliptic locus  can be computed by
the Porteous formula and similar tools.

\end{abstract}

\maketitle
\tableofcontents
\section*{Introduction}
Our main aim is to prove Theorem \ref{fund-class-thm} below, 
which is a precise version of
the following
\begin{mainthm}[First approximation]
Given a family $X/B$ of stable curves, there is a bundle map over an explicit
birational modification of its second symmetric product , whose degeneracy
locus consists of the closure of the hyperelliptic locus, which is reduced of the expected dimension, 
plus an explicit
and computable excess locus. In this way the fundamental class of the 
closure of the hyperelliptic locus
can be computed as an element of the Mumford tautological ring.
\end{mainthm}
This paper is a continuation of our study of finite subschemes of
families of nodal-or-smooth curves (see e.g. \cite{geonodal}, \cite{canonodal}, \cite{internodal}). Technically, our aim is to introduce a new tool in the  global and enumerative
geometry of the moduli space
$\mgbar$ of stable curves: modified Hodge bundles.
The rationale for these is a pervasive problem which has long stood in the way of
applying 'classical' methods to $\mgbar$:
that equations (e.g. degeneracy conditions for $g^r_d$'s)  describing
geometry on smooth curves become excessively
degenerate on singular,
especially reducible nodal
curves, and accordingly fail to define
a limit, in any good sense, of the appropriate locus
(e.g. $g^r_d$ locus) on a smooth curve.
Accordingly, most recent work on $\mgbar$ has focused on
extrinsic, Gromov-Witten type methods, studying maps of curves
to other varieties; see \cite{vak} for a survey of some of this work and references. \par
Nevertheless,  this paper represents
the beginning of an attempted attack on the aforementioned
excess degeneracy problem, based
\emph{grosso modo} on resolving the excess through  boundary modifications of Hodge bundles.
This approach has its roots in the work of Harris and Mumford
\cite{har-mum} on the Kodaira dimension of $\mgbar$, especially their computation of the fundamental class of the closure of
the divisor of curves carrying a $g^1_{\frac{g+1}{2}}$.
The basic older insight is that the appropriate boundary object
corresponding to a linear system is a \emph{collection}
of linear systems on components, or rather certain
subcurves, of the boundary curve. The new 'twist' (double-entendre) is that those systems can be accessed via
a suitable map of vector bundles.
\par
 Specifically, we are concerned
here with $g^1_2$'s, i.e. the locus of smooth hyperelliptic curves and its closure  in 
$\mgbar, g\geq 3$, viewed via degeneracy
(non-very ampleness) of the canonical system.
The usual description in the smooth case is
in terms of the 'degree-2 \bn map' (evaluation map)
\[\phi: \bE\to \Lambda_2(\omega)\]
where $\be$ is the Hodge bundle and $\Lambda_2(\omega)$ denotes the tautological bundle of rank 2 associated to the relative
canonical bundle, defined over the degree-2 relative
Hilbert scheme of the universal family
(which for smooth curves coincides with the relative symmetric product). Precisely, the degeneracy locus of $\phi$
consists of 'hyperelliptic pairs' $(C,\a)$ where $C$ is a hyperelliptic
curve and $\a$ is a divisor in its unique $g^1_2$.
\par Now all of the above data, including $\phi$,
 extend over $\mgbar$
and its associated relative Hilbert scheme. But
the degeneracy locus of the extended map is \emph{not}
the closure of the locus of hyperelliptic pairs.
It contains, e.g. pairs $(C_1\cup C_2, \a)$
where $C_1\cup C_2$ is a reducible stable curve and
$\a$ is a hyperelliptic divisor on $C_1$,
which cannot be the limits of a smooth
hyperelliptic curve.
In essence, this issue
is what this paper is about.\footnote{The presence of extraneous, often excessive,
boundary components is a difficulty in Gromov-Witten theory as well.}
The basic idea is an obvious one: enlarge $\be$ at the boundary by allowing (carefully controlled) poles, so that $\phi$ remains defined but has smaller degeneracy locus, because it effectively accesses
a larger linear system. More precisely, given a family
$X/B$ of stable curves of genus $g\geq 3$, the enlargement is accomplished via suitable \emph{echelon modifications}
(a generalizations of the familiar elementary modifications,
see \cite{echelon}) along certain boundary divisors.
These divisors are associated to the separating nodes
and separating pairs of nodes (binodes); the latter case requires blowing up the Hilbert scheme. The ultimate result
is the following (see Theorems \ref{main-thm}  and \ref{fund-class-thm} for the precise statements)
\begin{mainthm}[Second approximation]
There is a bundle $\ae$ called the
 \emph{azimuthal Hodge bundle}, which is an echelon modification
 of the Hodge bundle  defined over a blowup $X\sbc 2._B$ of
the relative Hilbert scheme $X\sbr 2._B$, together
with a mapping, called the \emph{azimuthal \bn map}
\[\aphi: \ae\to\Lambda_2(\omega)\]
whose degeneracy scheme is the union of
 a lift of the hyperelliptic locus and the locus of schemes supported on some separating node.
 Via the excess Porteous formula, the fundamental class of the lifted hyperelliptic locus can be computed
 as an element of the Mumford tautological ring.
\end{mainthm}
The  contribution to the degeneracy locus from the
 locus of schemes supported at nodes can be easily computed using residual intersection theory;
 moreover, its image on $B$ vanishes, essentially because the dualizing sheaf restricts to
 the structure sheaf on a node (see Cor. \ref{fund-class-base-cor}). Therefore,
the class of the hyperelliptic locus can be computed
 by the Theorem using  Porteous's formula.\par
 To be fair, the phrase 'can be computed' as used above should be understood,
 for $g\geq 4$, 
 in the sense of 'reduces to routine, if tedious, calculations'; these calculations can be handled in
 principle by
 Gwoho Liu's Macnodal program (see \cite {internodal}), though some of the details have yet to be worked out.
 \par
 It should be mentioned that in the case of genus 3, the Brill-Noether
 map as is has some 'extraneous' (non-hyperelliptic) degeneracy loci, but these
 are not 'excessive' in dimension, and their contribution
 can be easily computed, leading to a computation of the hyperelliptic
 class in $\overline\M_3$, confirming a formula of Harris-Mumford. See
 \cite{bleier-genus3}, \cite{esteves-genus3} or \cite{internodal}, \S 4.5.
  The genus-4 case is the first involving excessive
  degeneration of Brill Noether. It is discussed in detail in \S \ref{genus-4-example},
  where we recover a formula of Faber-Pandharipande (\cite{faber-pandh}, Prop. 5).\par
The azimuthal \bn map is related to a new geometric structure
on the boundary, encoded in the \emph{sepcanonical system}
on boundary curves. This is a collection of (usually incomplete) linear systems on certain subcurves ('2-separation
components'); on a given subcurve $Y$, the sepcanonical system
is a twist of the canonical system of $Y$ which
reflects the geometry of $X$ as a whole. Then
on subschemes of $Y$, the azimuthal
\bn map of $X$ is 'essentially' the evaluation map associated to the sepcanonical system.\par
This paper is divided in two parts. Part 1
culminates in the proof of the Main Theorem for curves of
\emph{semicompact type}, i.e. those whose dual graph contains
no circuits of size $>2$. Part 2 extends the result to the general case.
See the introductions to each part for further organizational details.\par
\subsection*{Further developments}
The methods of this paper seem to extend to the case of pencils
(1-dimensional systems) of degree $m>2$. This
corresponds to studying the submaximal-rank locus or maximal minors ideal of the
Brill-Noether map, and involves two main new steps (the details will be
pursued elsewhere):\par
(i) modifying the tautological bundle as well as the Hodge bundle;\par
(ii) performing bundle modifications on the intermediate boundary components,
e.g. those birational to $ X_1\spr i.\times X_2\spr m-i.$, for a
compact-type boundary curve
$X_0=X_1\cup X_2, i=0,...,m$, rather than just the extremal components (where $i=0$ or $m$).
These modifications eliminate extraneous boundary components.\par
Being able to perform these bundle modifications, especially for curves with complicated dual graph,
 would require blowing up the Hilbert scheme, extending the
azimuthal modifications considered here.\par

In the case of 2-dimensional systems, corresponding to the sub-submaximal
rank (corank 2 ) locus or submaximal minors ideal of Brill-Noether, extraneous but nonexcessive components arise, and
it appears the present methods may apply. Going beyond 2-dimensional systems however
gives rise to extraneous and excessive boundary components, and it's not clear how
to account for their contributions.\par
 See \cite{faber-pandh} for another approach, based on Gromov-Witten
theory, to computing the fundamental class of hyperelliptic
and similar loci related to maps of curves to $\P^1$.\par
\subsection*{Acknowledgment}
I thank the referee for his very careful reading of the paper and highly detailed comments
and corrections, including  his insistence that we reproduce the Faber- Pandharipande formula
in genus 4; these have resulted in a much improved paper. I thank Ann Kostant for her patient and determined
assistance.
\Large
\begin{center}{\part{{Semicompact type}}}
\normalsize \end{center}
\normalsize
This part is mainly devoted to proving the Main Theorem in
the case of curves of semi-compact type (defined below);
however, some topics are developed in greater generality
for use in the general case, to be completed in Part 2 
(as well as in potential further applications).
In \S \ref{smooth-hypell-sec} we review some standard facts about smooth hyperelliptic curves and derive normal form for
some objects associated to them, such as the \bn map.
\par
In \S\ref{hilb-review} we review some constructions and properties of Hilbert schemes of curves in the very special
case of degree 2.\par
The core of the paper begins in \S \ref{bn-sep}, which
constructs and studies the modified Hodge bundle in the case of a separating node, using an appropriate echelon modification.
The object that appears on the boundary turns out to be
closely related to sepcanonical systems.\par
\S \ref{bn-bisep} extends the modified \bn map, first to
the case of a single separating binode, then more generally
to a disjoint collection of separating nodes and
binodes. The binode case, because it occurs in codimension
2, requires a blowup of the Hilbert scheme that we call
an \emph{azimuthal} modification. This amounts to adding
some tangential data, called an azimuth, at the binode.
On the modified Hilbert scheme, a modified \bn map can be constructed largely as in the separating node case, again
leading to an object closely related to the sepcanonical
system. We
then prove a provisional form of our main theorem, stating
that for curves of 'semi-compact type', i.e. whose dual graph has no cycles of size $>2$, the degeneracy locus of the modified \bn map consists of an appropriate lift
of the hyperelliptic locus, plus the locus of
all schemes supported on some separating nodes. \par
In Part 2, we will extend the latter result to general
stable curves, derive a formula for the fundamental
class of the stable hyperelliptic locus, and study
intersection theory on the azimuthal modification of
the Hilbert scheme.\bigskip

\subsection{Riemann-Roch without denominators for anti-self-dual}\label{asd}
The purpose of this brief section
 is to point out that Fulton's
Riemann-Roch without [integer] denominators (see \cite{ful}, Ch. 15)
can be simplified in the case of Anti-Self-Dual bundles (defined below),
so as to eliminate (characteristic class) denominators other
than those of the form \mbox{$1+D, D=\ $divisor,} which are easy to invert.
These results will be used in \S \ref{chern-of-azi}.
\par
 A vector bundle $E$ on a
scheme $Y$ is said to be \emph{anti-self-dual} or ASD if
\[c(E)c(E\dual)=1,\] where $E\dual $ denotes the dual bundle.
\begin{example}
As importantly observed by Mumford \cite{Mu},  the Hodge bundle $\be_g$ on $\mgbar$
 has the ASD property. Consequently, the pullback  of $\be_g$ by any map $Y\to\mgbar$
 also has the ASD property.\par
\end{example}
Let $D$ be a Cartier divisor on a variety $Y $,
let $i_D:D\to Y$ be the inclusion map, and let $E$ be a vector bundle on $D$.
The ASD property for a bundle $E$ allows us to avoid the
computationally unwieldy process of dividing by $c(E)$ . Therefore,
computations involving the Riemann-Roch without denominators
(\cite{ful}, \S15.3), even in the divisor case, are simplified when
the bundle is ASD because rather than divide by $c(E(-D))$ we need
only divide by powers of $c(\O_D(-D))=1-[D]_D$. To elaborate, define
a polynomial $Q$ in Chern classes of a rank-$e$ bundle $E$ and a
line bundle $L$ by \eqspl{}{ c(E\otimes L)=c(E)+[L]Q(L, E) } or
explicitly, \eqspl{}{
Q(L,E)&=\sum\limits_{p=0}^{e-1}\sum\limits_{i=0}^{p}\binom{e-i}{p+1-i}
c_i(E)c_1(L)^{p-i}\\&=e+\binom{e}{2}c_1(L)+(e-1)c_1(E)+\binom{e}{3}c_1^2(L)+\binom{e-1}{2}c_1(L)c_1(E)
+(e-2)c_2(E)+... }
Note that
\eqspl{Q-dual-eq}{
Q(L, E)=Q(L\dual, E\otimes L).
}
The Riemann-Roch without denominators (\cite{ful}, Example
15.3.4) states that,
for any locally free $\O_D$- sheaf $E$, we have
\eqspl{rrwod}{c(i_{D*}(E))=1+i_{D*}(\frac{Q(-D, E)}{c(E\otimes\O_D(-D))}).}
Then from the definition of ASD, we conclude directly:
\begin{prop} Notations as above,
for $E$  ASD of rank $e$ on the divisor $D$, and any 
Cartier divisor $G$ on $D$, we have \eqspl{}{
c(i_{D*}(E(G)))=1+i_{D*}((1+G-D)^{-2e}Q(-D, E(G))c(E\dual(G-D))). \qed}\end{prop}
\begin{example}
When $G=D$, the above simplifies nicely, using \eqref{Q-dual-eq},  to
\eqspl{}{
c(i_{D*}(E(D)))&=1+i_{D*}(Q(D, E)c(E\dual))\\
&=1+i_{D*}(\sum\limits_{j=0}^e\sum\limits_{p=0}^{e-1}\sum\limits_{i=0}^p(-1)^{j}
\binom{e-i}{p+1-i}c_i(E)c_j(E)D^{p-i}).
}
\end{example}
\begin{example}\label{hodge-boundary-example}
Let $Y=\mgbar$ and let $D$ be the divisor $\overline \M_{h,1}\times\overline \M_{g-h, 1}$, 
with conormal bundle
$\psi=\psi_h\otimes\psi_{g-h}$, the product of the
respective cotangent classes. Let $\be_h$ be the pullback
of the rank-$h$ Hodge bundle to $D$, which is ASD as 
vector bundle over $D$ because
the Hodge bundle itself is ASD as vector bundle over $\overline\M_h$.
Then for a divisor class $G$ on $D$,
\[c(i_{D*}(\be_h\otimes G))=1+i_{D*}(\frac{Q(\psi, \be_h\otimes G)c(\be_h\dual\otimes G\otimes\psi)}{(1+G+\psi)^{2h}}).\]
\end{example}
\begin{cor}\lbl{RR-trivial-normal-cor}
Assumptions as above, let $Z\subset D$ be a subvariety such that $c_1(\O(D))\cap [Z]=0$ and let $F$ be any vector bundle on $Y$. Then: (i)
\[c(i_{D*}(E\otimes F))._Y Z=1;\]
(ii) for any line bundle $M$ on $D$,
\[c(i_{D*}(M))._Y Z=1.\]
\end{cor}
\begin{proof}(i) Because $i_{D*}(E\otimes_D F)=i_{D*}(E)\otimes_Y F$, we may assume $F=\O_Y$. Then the assertion is obvious from the Proposition
   above, plus the standard fact that for any class $a$ on $D$,
\eqspl{star}{\ \  i_{Z,Y}^*i_{D*}(a)=i_{Z,D}^*(a.[D]).}
(ii) In this case, Fulton's example reads
\[c(i_{D*}(M))=1+i_{D*}\frac{1}{1+[M]+[L]}, \ L=\O_D(-D).\]
Then our assertion follows easily from \eqref{star}
above.
\end{proof}
\newsection{Locus of smooth hyperelliptics: a review}\lbl{smooth-hypell-sec}
The purpose of this section is to review some elementary facts about
smooth hyperelliptic curves and the locus they make up in a family of
smooth curves, especially the normal bundle to this locus.\par In general,
given a  versal family $\pi:X\to B$ of stable, generically smooth, curves, we let
$\he_B\subset B$ denote the
closure of the locus of smooth hyperelliptic curves, which is of virtual
codimension  equal to $g-2$. We
also denote by $\he^2_B\subset X\sbr{2}._B$ the closure of the locus of
schemes invariant by the hyperelliptic involution, i.e. fibres of the canonical
mapping. This has virtual codimension $g-1$. The fibre of $\he^2_B$ over
an interior point $b\in B$ is either empty, if $X_b$ is non-hyperelliptic, or equal to
the target $\P^1$ of the hyperelliptic pencil, otherwise. We will denote by
$\he_B^1\subset\he_B^2$ the sublocus consisting of length-2, 1-point schemes,
i.e. the (schematic) intersection of $\he^2$ with the diagonal divisor
$\Gamma=\Gamma\sbr 2.\simeq X$ or equivalently,
the locus of ramification points of the hyperelliptic map,
i.e. Weierstrass points. Because the ramification
  is simple, the intersection is transverse and
$\he^1_B\to \he_B$ is \'etale.
We let $\bE$ denote the Hodge bundle on $B$,
$\bE=\pi_*\omega, \omega=\omega_{X/B}$.
\newsubsection{Normal bundle}
Let $X$ be a smooth hyperelliptic curve with hyperelliptic map
\[f:X\to \P^1\]
and $\eta=f^*O(1)$ the hyperelliptic bundle and $\beta\subset X$ the ramification divisor, which is reduced of degree
$2g+2$. Then deformations of the pair $(X,f)$ are unobstructed
and parametrized by $H^0(N_f)$, which fits in an exact sequence
\[ 0\to H^0(f^*(\theta_{\P^1}))\to H^0(N_f)\to
H^1(\theta_X)\to H^1(f^*(\theta_{\P^1}))\to 0\]
and $f^*(\theta_{\P^1})=\eta^2$. Then the first group
coincides with $\sym^2(H^0(\eta))$, which corresponds to
reparametrizations of the target $\P^1$, while
the last group is dual to $H^0(\omega\otimes\eta^{-2})
=H^0(\omega^2(-\beta))$ which is $g-2$-dimensional. This implies that in any versal family $X/B$, the hyperelliptic locus $\he_B\subset B$ is smooth of codimension $g-2$. Moreover, because
the Weierstrass point locus $\he^1_B\subset X$ is  \'etale over $\he_B$, it too is smooth.
\subsection{Normal forms}\lbl{normal forms} We study the locus of hyperelliptics
(and related loci)  in
a family of smooth (non-pointed) curves.
Consider a  family $X/B$ of smooth curves. We note that because of the
existence of a tautological family of curves over $B$, the degree of the natural map
$B\to\M_g$ is at least $2$ near any hyperelliptic curve. We assume
the family  is locally versal, hence  this map is ramified only over
the hyperelliptic locus and other loci of curves
with automorphisms. Therefore $\he_B$ is smooth of codimension
$g-2$ in $B$ (of course $\M_g$ is singular along its image).
Then $\he^2_B\subset X\spr 2._B$ is
just the degeneracy (rank-1) locus of the natural evaluation map
that we will call the (degree-2) Brill-Noether map \eqspl{}{\phi:
\pi^{(2)*}\be \to \Lambda_2(\omega). } Here $\Lambda_2(\omega)$ is
the 'secant bundle' which in the general case of stable curves
 is defined
on the relative Hilbert scheme $X\sbr m._B$ (see \S \ref{hilb-review}).\par
Our purpose here is to give a normal form for the Brill-Noether map $\phi$
and an 'augmented' analogue, especially in a neighborhood of a
hyperelliptic curve. Let $(X_0,\theta_0)$ be a hyperelliptic pair or 'pointed
hyperelliptic curve', i.e. $X_0$ is hyperelliptic and $\theta_0$ is a Weierstrass
point on it. We work locally on the degree-2 Hilbert scheme $X_0\sbr
2.=X_0\spr 2.$, with the tautological rank-2 bundle $\Lambda_2(\O)$, at the
scheme $2\theta_0$ (see \S \ref{hilb-review} for a review). Let $s_0,...,s_{g-1}$ be a basis for $H^0(\omega_{X_0})$
such that \[\ord_{\theta_0}(s_i)=2i,\] i.e. locally at $\theta_0$ with local coordinate
$x$, $s_i\sim x^{2i}$. Local coordinates for $X_0\spr 2.$ near $2\theta_0$ are
$\sigma_1=x_1+x_2, \sigma_2=x_1x_2$ where $x_i=p_i^*x$. Set
\eqspl{ei-def-eq}{e_i=\Lambda_2(x^i).
} Then $e_0,e_1$ is a local frame for
$\Lambda_2(\O)$ and we have
\eqspl{e-recursion}{e_{i+1}=\sigma_1e_i-\sigma_2e_{i-1}, i\geq 1. } It is
easy to check from \eqref{e-recursion} that if we write
$e_i=a_{0i}e_0+a_{1i}e_1, i\geq 2$, then \eqsp{a_{0i}\equiv
0\mod\sigma_2,\\a_{1i}\equiv 0\mod \sigma_1,\  i {\mathrm{\ even}}.} In terms
of these data, the Brill-Noether map is represented by the $2\times g$
matrix \eqspl{bn-single-he}{\Lambda=\Lambda_2(s_0,...,s_{g-1})= \left [
\begin{matrix}
1&-\sigma_2&\sigma_2(\sigma_2-\sigma_1^2)&...\\
0&\sigma_1&\sigma_1(\sigma_1^2-2\sigma_2)&...
\end{matrix}\right ]
}
It is easy to see that
all  entries of the 1st (resp. 2nd) row beyond the 1st
column are divisible by $\sigma_2$
(resp. $\sigma_1$). Consequently the ideal of $2\times2$ minors of $\Lambda$ is $\sigma_1$,
i.e the equation of the graph of hyperelliptic involution.
\par
To deal with equations for hyperelliptic pairs we consider analogously the
\emph{augmented} Brill-Noether map, which is the analogous map for the
line bundle $\omega(2\theta_0)$. Note that for any smooth pair $(X,\theta)$,
$\omega_{X_0}(2\theta_0)$ is base-point free, $(g+1)$-dimensional and ramified
at $\theta_0$, and $(X_0,\theta_0)$ is hyperelliptic if and only if
the mapping associated to $|\omega_{X_0}(2\theta_0)|$ fails
to be an isomorphism off $\theta_0$, in which case it is actually composed of
the hyperelliptic involution. If $(X_0,\theta_0)$ is hyperelliptic, we may choose a
basis $(s.)$ for $H^0(\omega_{X_0}(2\theta_0))$ so that $\ord_\theta(s_i)=2i,
0\leq i\leq g$ so the associated augmented $2\times (g+1)$ Brill-Noether
matrix has the form similar to \eqref{bn-single-he} \eqspl{}{
\Lambda^+=\Lambda_2(s_0,...,s_g)=\left [
\begin{matrix}
1&-\sigma_2&\sigma_2(\sigma_2-\sigma_1^2)&...\\
0&\sigma_1&\sigma_1(\sigma_1^2-2\sigma_2)&...
\end{matrix}\right ]
}
By contrast, when $(X,\theta)$ is non-hyperelliptic, the sequence of vanishing orders at $\theta$ for $\omega(2\theta)$ starts $(0,2,3,...)$, so the augmented
Brill-Noether starts
\[\left[ \begin{matrix} 1&*&*&...\\
0&\sigma_1&\sigma_1^2-\sigma_2&...
\end{matrix}\right ]\]
So in that case the ideal of 2-minors of this is clearly $(\sigma_1, \sigma_2)$, the ideal of the point
$2\theta_0\in X_0\spr 2.$.
\par
Now suppose $X_0$ varies in a versal  family $(X_0\subset X)/(0\in B)$ and admits
a section  $\theta$  (not considered part of the data and subject to change), whose value over $0\in B$ is a
Weierstrass point $\theta_0$ on $X_0$, and extend $(s.)$ to a basis $\tilde s_i$ of
the Hodge bundle $\bE$ over $B$. In terms of a local fibre coordinate, we
can write, after a suitable change of basis 
\[\tilde s_i=x^{2i}+z_ix^{2i-1}+\sum\limits_{i\leq j\leq 2i-2} w_{ij}x^j, i=1,...,g-1\]
where the $z_i, w_{ij}\in\m_{0,B}$. 
By suitably
changing the section $\theta$,
or equivalently, changing the fibre coordinate $x$ based at $\theta$ to $x-z_1/2$, we may assume
$z_1=0$, i.e. $s_1=x^2$.  This ensures that $\theta$ is in the  ramification
locus of the map to $\P^1$ determined by $s_0, s_1$.
 As long as $X$ varies in the hyperelliptic locus $\he_B$,
the latter ramification locus coincides locally with the locus of
Weierstrass points, which itself is \'etale over $\he_B$.
Thus, $\theta$ remains a
Weierstrass point in any deformation of $X$ \emph{as hyperelliptic curve.}
Then for $i>1$, we may inductively subtract off a suitable $\O_B$-linear combination of the $s_h, h<i$
to arrange that $w_{ij}=0$ for $j$ even.
Consequently, note that the schematic condition defining the
locus $\he_B\subset B$ is now precisely
\[ z_2=...=z_{g-1}=0.\]
Indeed this vanishing condition is equivalent to
$h^1(\O(2\theta))=h^0(\omega(-2\theta)\geq g-1$, i.e. by Riemann-Roch,
$h^0(\O(2\theta))\geq 2.$ The fibre of $\he^2_B\to\he_B$ is the graph of the
hyperelliptic involution on a given curve, which for suitable coordinates is
given by the 'antidiagonal' $\sigma_1=0$ (which is transverse to the
diagonal, due to the fact that the involution has simple fixed
points). Because $\he_B\subset B$ is a smooth subvariety of
codimension $g-2$, $z_2,...,z_{g-2}$ are regular parameters near $0$.
\par Now because the $w_{ij}$ vanish on the hyperelliptic locus, we can
write \eqspl{}{ w_{ij}=\sum\limits_{k=2}^{g-1} u_{ijk}z_k } Plugging this into
the Brill-Noether matrix for the family, we obtain

\eqspl{bn-rel-he}{
\left [
\begin{matrix}
1&-\sigma_2&\sigma_2(\sigma_2-\sigma_1^2)&*&...\\
0&\sigma_1&\sigma_1(\sigma_1^2-2\sigma_2)+z_2+\sum\limits_{k>2}
z_k(*)&\sigma_1(*)+z_3+\sum\limits_{k>3} z_k(*)&...
\end{matrix}
\right ]
}
 which is column-equivalent to
\eqspl{bn-rel-he2}{
\left [
\begin{matrix}
1&*&*&*&...\\
0&\sigma_1&z_2&z_3&...
\end{matrix}
\right ]
}

The degeneracy locus of this is precisely
\eqspl{}{
\he^2_B=\{\sigma_1=z_2=...=z_{g-1}=0\}\subset X\spr 2._B.
}\newsubsection{Pointed case}
The case of the pointed hyperelliptic locus, i.e. the locus of hyperelliptic pairs $(X,\theta)$ is similar, using the
augmented Brill-Noether map. Thus let $(X_{B_1},\theta_{B_1})$ be a versal family of pointed
smooth curves parametrized by $B_1$, where we may assume
$B_1$ itself is the total space of a family $X_B$ over $B$, with the extra parameter specifying
the value of the section $\theta$. Then, working analogously with the system $\omega(2\theta)$,
we may assume
\eqsp{
s_0=1, s_1=x^2,\\
s_i=x^{2i}+z_ix^{2i-1}+\sum\limits_{j=i+1}^{2i-2} w_{ij}x^j,
i=2,...,g
} where $w_{ij}=0$ for $j$ even.
Here $z_1$ is a vertical coordinate of $B_1/B$, i.e. a fibre coordinate.
 As before, the ideal of the
locus of hyperelliptic pairs $\he^1_B\subset B_1$, i.e. the locus of pairs $(X,\theta)$ where
$h^0(\omega_X(-2\theta))\geq g-1$, is generated by the $z$'s,
so we may assume the $w$'s are in the ideal generated by the $z$'s. Now we have an augmented Brill-Noether map over
$X\spr 2._{B_1}$:
\[(\pi\spr 2.)^*\pi_*(\omega(2\theta))\to\Lambda_2(\omega(2\theta)).\]
The matrix relative to the $s.$
basis above and the usual $e_0, e_1$ basis on the target,
known as the \emph{augmented Brill-Noether matrix},  takes the form
\eqsp{
\Lambda^+=\left [\begin{matrix}
1&*&*&...&* \\
0&\sigma_1&z_2\sigma_2&...&z_g\sigma_2
\end{matrix}\right ]
} The degeneracy locus of this in $X\spr 2._B$ is the schematic union of
the section $2\theta$, with ideal $(\sigma_1, \sigma_2)$, and the locus of
the hyperelliptic involution, lying over the locus of hyperelliptic pairs in $B$,
with ideal $(\sigma_1,z_2,...,z_g)$. Again by our assumptions, these are
regular parameters at $(0,2\theta)$.\par

\newsection{Good families and their Hilbert scheme}\lbl{hilb-review}
For more information on Hilbert schemes of families of
nodal curves, see \cite{structure} or \cite{internodal}.
Consider a proper family $\pi:X\to B$ of
connected nodal curves of genus $g$.
Fix a fibre node $\theta$ of $X/B$, with corresponding boundary divisor $\del=\del_\theta\subset B$. At least locally in $X$ near $\theta$, the boundary family $X_{\del}$ splits:
\[X_{\del}= (\lf{X}, \lf{\theta})\bigcup\limits_{\lf{\theta}\to\theta\leftarrow\rt{\theta}}
(\rt{X}, \rt{\theta}).\]
In the case $\theta$ is separating- the case of principal interest here- the splitting is defined locally in $B$, and
even globally in $B$ if $g(\lf{X})\neq g(\rt{X})$, e.g. if $g$ is odd.
We may call $\lf{X}, \rt{X}$ the left and right sides of $\theta$ respectively. The choice of left and right sides of $\theta$ is called an orientation.
We call the boundary component $\del$  decomposable if
\[\del=\lf{\del}\times\rt{\del}\]
where $(_*X, \ _*\theta)/_*\delta$ is a family of stable pointed curves of
genus $_*g, \ *=L, R, \ \lf{g}+\rt{g}=g$. 
This is certainly true for the universal family over $\mgbar$  and suitably chosen base-changes over it.
The decomposability assumption is not essential for our basic constructions, however.
\par Now consider the length-2 Hilbert scheme $X\sbr 2._B$
near the boundary component $\del$.
Its boundary part has the form
\[(\lf{X})\sbr 2._{\del}\cup \lfrt{X}\cup (\rt{X})\sbr 2._{\del}\]
with $\lfrt{X}\to \lf{X}\times_{\del} \rt{X}$ the blowing-up of $(\lf{\theta}, \rt{\theta})$,
with exceptional divisor
\[R_\theta=\P(\lf{\psi}\oplus\rt{\psi}),\  _*\psi=T^*_{_*\theta, _*X/{\del}}.\]
The locus $R_\theta$ parametrizes relative length-2 subschemes of $X$ supported on $\theta$. It is a $\P^1$-bundle over $\del$ admitting a pair of disjoint sections  
$_*Q:=\P(_*\psi)=R_\theta\cap (_*X)\sbr 2._B, \ *=L,R$.
Moreover, if $*\dag$ denotes the mirror of $*=L,R$,
\[(_*X)\sbr 2._{\del}\cap \lfrt{X}=_*X\times_{*\dag}\theta, *=L,R;\ 
(\lf{X})\sbr 2._{\del}\cap (\rt{X})\sbr 2._{\del}=\emptyset.\]
Now working locally, assume the family is given by the standard 
form $xy=t$ with $x,y$ regular parameters on $X$ and $t$ a local defining equation for the Cartier divisor $\delta$ on $B$.
Set $x_i=p_i^*x, \ i=1,2$ (a function on the Cartesian square) and similarly $y_i$.  The elementary symmetric functions  $\sigma_1^x, \ \sigma_2^x$ (resp. $\sigma^y_1\sigma^y_2$) are functions on the symmetric product. Then
near the finite part of $R_\theta$, i.e. off $\rt{Q}$, we get regular parameters on $X\sbr 2._B$
\[\sigma_1=\sigma_1^x, \ \sigma_2=\sigma_2^x, \ u=y_1/x_2=y_2/x_1=(\sigma_1^y)/(\sigma^x_1),\]
with $\sigma_2, u$ being defining equations, respectively,  for $\lfrt{X}, \ (\lf{X})\sbr 2._{\del}$, 
while $\sigma_1, \sigma_2$ together define $R_\theta$.  Similarly, off $\lf{Q}$, 
we have $\sigma_1^y, \sigma_2^y, \ v:=1/u$ as regular parameters.
Note that the diagonal is defined here by $\sigma_1^2-4\sigma_2$ (in either $x$ or $y$ variables in
the appropriate open set). This diagonal is isomorphic to
the blowup of $X$ in $\theta$, with coordinates $\sigma_1, u$.
\par
We want to study the Brill-Noether map
\eqspl{}{
\phi:(\pi\sbr 2.)^*(\be)\to\Lambda_2(\omega).
} locally over $2\theta\in X\spr 2._B$ and off $\lf{Q}\cup\rt{Q}$.
 We recall (see \eqref{ei-def-eq}) that $e_i=\Lambda_2(x^i)$ and $e_0, e_1$ is a local basis for $\Lambda_2(\O_X)$,
 identified with $\Lambda_2(\omega)$. Then
note that
\eqspl{sigma-table}{
\Lambda_2(y)=u(\sigma_1e_0-e_1), \Lambda_2(y^2)=u^2((\sigma_1^2-\sigma_2)e_0
-\sigma_2e_1),\\
 \Lambda_2(y^3)=u^3((\sigma_1^3-2\sigma_1\sigma_2)e_0+(\sigma_2-
\sigma_1^2)e_1),\\
\Lambda_2(y^i)=u\sigma_1\Lambda_2(y^{i-1})-u^2\sigma_2\Lambda_2(y^{i-2}).
}
This comes from the fact that, on the Cartesian product, the lift $\Lambda_2^\lceil(y)$ of $\Lambda_2(y)$
is given by
\[\Lambda_2^\lceil(y)=(y_1,y_2)=u(x_2,x_1)=u(\sigma_1e_0-e_1).\]
Now assuming $(X_1, \theta_1), (X_2, \theta_2)$ are both non-hyperelliptic, $\omega_{X_i}(\theta_i)$ have vanishing sequences
$(1,2,...)$, therefore the matrix of the Brill-Noether mapping with respect to the $e_0,e_1$ basis on $\Lambda_2$ and a suitable adapted basis of the Hodge bundle, where the part corresponding to $X_2$ is located right of the vertical bar and indexed negatively,
has the form
\eqspl{}{\left [
\begin{matrix}
...&u^2(\sigma_1^2-\sigma_2)&
u\sigma_1&|&0&-\sigma_2&...\\
...&-u^2\sigma_2&-u&|&1&\sigma_1&...
\end{matrix}\right ]
} whose degeneracy scheme coincides locally with the schematic union
\[\Zeros(u,\sigma_2)\cup\Zeros(\sigma_1, \sigma_2)=(\lf{X}\times\rt{\theta})\cup R_\theta.\] Taking into account the 'opposite' open set containing $Q_2$, we see that the degeneracy locus of $\phi$ is
\[R_\theta\cup \lf{X}\times\rt{\theta}\cup \rt{X}\times\lf{\theta},\]
i.e. the locus of schemes whose support contains the node $\theta$.
In the next section we will describe a modification of $\be$ and $\phi$ that will have a smaller degeneracy locus; e.g. over the general curve, we just get $R_\theta$.\par
\begin{defn}\label{good-family}
A family of nodal curves $X/B$ is said to be \emph{good}
if
\begin{itemize}\item
every fibre is semi-stable;
\item every boundary component with reducible general fibre is decomposable;
\item the family is everywhere locally versal: i.e. for every fibre $X_0$ with nodes $\theta_1,...,\theta_k$ the map from the  germ of $B$ at $0$ to the product of local deformation spaces of $X_0$ at $\theta_1,...,\theta_k$ is smooth.
\end{itemize}
\end{defn}
\subsection{Disconnected, pointed version}\label{disconnected-pointed-sec}
Let $Y$ be a possibly disconnected nodal curve (over some base) endowed with a collection of
disjoint (\'etale) multisections $p_1,...,p_n$ with each $p_i$ contained in a unique connected component.
A \emph{separating node} of  $Y$ is a node $\theta$ whose blowup disconnects some connected component
of $Y$ in two components, say $ \lf{Y}(\theta),  \rt{Y}(\theta)$,   and for each $i$, $p_i$ is 
either contained in or disjoint from $\lf{Y}(\theta)$ (resp. $\rt{Y}(\theta)$). 
\newsection{Modifying Brill-Noether, Case (i): separating nodes}\lbl
{bn-sep}
The traditional difficulty with extending the degeneracy-locus description of the hyperelliptic locus (and its analogues) across the boundary stems largely from the presence of
 'part dead' sections of the relative canonical, i.e. sections vanishing on some component.
 In this paper, our approach to resolving this difficulty is to revive
the sections that vanish on the part of the curve
 equal to a side of a node or binode by enlarging the Hodge bundle,
i.e. the source of the Brill-Noether map; or rather, we enlarge
the pullback of the Hodge bundle over the degree-2 Hilbert scheme (or in the next
section, a modification thereof).
This enlargement is accomplished by appropriate echelon modifications (see \cite{echelon}) of the Hodge bundle which mirror,
 and will later be linked (see Theorem \ref{main-thm})
 to the modifications yielding the sepcanonical
system, developed in \cite{canonodal}. We begin in this section with the case of a separating node. The more challenging
case of a separating binode is taken
up in the next section.
\newsubsection{Construction}\lbl{mod-bn-sep-sec}
For now, we work one node at a time. Thus, fix a separating node or 'sep'
$\theta$ of the family $\pi:X\to B$ with corresponding boundary divisor
$\del=\del_\theta$. Set
\[ \rt{X}=\rx(\theta), \lf{X}=\lx(\theta).\]
These are families over $\del$ and embedded as Cartier
divisors on $X$. For now,
assume for convenience that these are individually defined, i.e. that
 $\theta$ is oriented as a sep. The ultimate
construction will not depend on the choice of orientation. We will also
assume for now that $\lf{X}, \rt{X}$ have respective genera $\lf{g}, \rt{g}\geq 2$.
\par Define \emph{twisted Hodge bundles} on $B$: \eqspl{}{ \be^{i,j}=\pi_*(\omega(i\lf{X}+j\rt{X})). }
For any $i\leq i', j\leq j'$, there is a full-rank inclusion $\be^{i,j}\to \be^{i',j'}$.
In particular, we consider two chains of full-rank subsheaves \eqspl{}{
\be^{-2, 0}\to\be^{-1,0}\to\be,\\
\be^{0,-2}\to\be^{0,-1}\to\be,
}
induced respectively by
\eqsp{
\omega(-2\lx)\to\omega(-\lx)\to\omega,\\
\omega(-2\rx)\to\omega(-\rx)\to\omega.
}
On the boundary, the latter sheaves take the form
\eqsp{
\lf{\omega}(3\lf{\theta})\cup\rt{\omega}(-\rt{\theta})\to
\lf{\omega}(2\lf{\theta})\cup\rt{\omega}\to\lf{\omega}(\lf{\theta})\cup\rt{\omega}(\rt{\theta}),\\
\lf{\omega}(-\lf{\theta})\cup\rt{\omega}(3\rt{\theta})\to
\lf{\omega}\cup\rt{\omega}(2\rt{\theta})
\to\lf{\omega}(\lf{\theta})\cup\rt{\omega}(\rt{\theta}),}
where $_*\omega=\omega_{_*X/_*\delta}, *=L,R. $ The maps
in the two rows  are given, locally near
$\theta$, by multiplication by $y,x$ respectively, where $x,y$ are
respective local equations for $\rt{X}, \lf{X}$. So these maps are injective on
$\lx$ and zero on $\rx$ or vice versa. On fibres, the sheaves involved,
other than $\omega=\lf{\omega}(\lf{\theta})\cup\rt{\omega}(\rt{\theta})$ itself, are base-point free on at least one of $\lx, \rx$, and it is elementary to check that the direct images are locally free
and compatible with base-change.\par
\textdbend{ {\bf{NB:}}
 the latter assertion
 would be false for any other
twists of $\omega$, e.g. $\lf{\omega}(4\lf{\theta})\cup\rt{\omega}(-2\rt{\theta}) $: this is because $h^0(\rt{\omega}(-2\rt{\theta}))$ can jump}.\qed\par 

\par
Now we shift attention to the Hilbert scheme $X\sbr 2._B$.
Set $_*D=_*D(\theta)=(_*X)\sbr 2._\delta$, a Cartier divisor on $X\sbr 2._B$.
In the local coordinates above, it has
 equation $u$ for $*=L$ and $v=1/u$ for $*=R$. Consider the
following echelon data (\cite{echelon}, \S 1)
on $X\sbr 2._B$:
\eqspl{echelon-data-sep-node}{\chi(\theta)=(\lf{\chi}(\theta), \rt{\chi}(\theta))\\
\lf{\chi}=
((\be^{-2,0},2\lf{D}),(\be^{-1, 0}, \lf{D}),\be)\\
\rt{\chi}=
((\be^{0,-2},2\rt{D}),(\be^{ 0,-1}, \rt{D}),\be)
.
} These are certainly transverse as $\lf{D}\cap \rt{D}=\emptyset$. Let $\be_\theta$ be the associated echelon modification (\cite{echelon}, \S 2). Clearly,
the Brill-Noether mapping $\phi$ vanishes on $\be^{-1, 0}$ over $\lf{D}$, over $\be^{-2.0}$
to order 2 over $\lf{D}$ etc. Therefore by the universal property of echelon modifications
(\cite{echelon}, Theorem 2.3), $\phi$
factors through $\be_\theta$:
\eqspl{}{
(\pi\sbr 2.)^*(\be)\to\be_\theta\stackrel{\phi_\theta}{\to} \Lambda_2(\omega).
} We call $\phi_\theta$ the \myit{modified Brill Noether map} (with
respect to the sep $\theta$).\par
\begin{rem}\label{genus-1-side-rem}In the event that $\lf{g}=1$, the datum $\lf{\chi}$ above is to be replaced by
\[\lf{\chi}=((\be^{-1,0}, \lf{D}), \be)\]
and similarly if $\rt{g}=1$.
\end{rem}
\newsubsection{Interpretation}\lbl{mod-bn-sep-interpret}
We will analyze $\phi_\theta$ near $\lf{D}$. Recall from \cite{echelon}, (8),
the injective map
\[ \be^{-2,0}(2\lf{D})\to \be_\theta\]
that is an isomorphism over the 'interior' $\lf{D}^o$ of $\lf{D}$, i.e.
\[\lf{D}^o=\lf{D}\setminus \lx\times\lf{\theta}=\{{\mathrm{
schemes\ not\ containing\  }} \lf{\theta}\}.\]
On the other hand, the inclusion $\omega(-2\lx)\to\omega$ induces
over $X\sbr 2._B$ a map
\[\Lambda_2(\omega(-2\lx))\to\Lambda_2(\omega).\]
This map vanishes twice on $\lf{D}$, hence induces
\[ \Lambda_2(\omega(-2\lx))(2\lf{D})\to\Lambda_2(\omega),\]
which again is an isomorphism over over the interior, $\lf{D}^o$ (indeed over this interior
a local equation for the boundary divisor $\del$ pulls back to a local equation
for $\lf{X}$ in $X$ and a local equation for $\lf{D}$ in $X\sbr 2._B$) .
Then we get a (left)  \emph{comparison diagram}, which
is a commutative square
\eqspl{interp-1node}{
\begin{matrix}
\be^{-2,0}(2\lf{D})&\to& \be_\theta\\
\downarrow&&\downarrow\\
 \Lambda_2(\omega(-2\lx))(2\lf{D})&\to&\Lambda_2(\omega)
\end{matrix}
}
in which the left column is just the ordinary
Brill-Noether map associated to $\omega(-2\lf{X})$; of
course, there are analogous diagrams, one with left column
replaced by
\[\be^{-1,0}(\lf{D})\to \Lambda_2(\omega(-\lx))(\lf{D}), \]
and two more with
left replaced by right.
So at least over $\lf{D}^o$,
$\phi_\theta$ can be identified with the Brill-Noether map associated
to $\omega(-\lx)$ and its restriction on $\lx$. That restriction coincides with the complete linear system $\omega_{\lx}(3\theta)$ if
$(\rx,\rt{\theta})$ is non-hyperelliptic (so that $\rt{\theta}$ is not
a base point of $\omega_{\rx}(-\rt{\theta})$);
otherwise, the restriction
of $\omega(-2\lx)$ on $\lx$ is $|\omega_{\lx}(2\lf{\theta})|+\lf{\theta}$, i.e. has an imposed base point. Note that the twist involved coincides with
that coming from the sepcanonical enlargement of the
 canonical system on $X$. On the other hand, it is elementary
 that a pair of distinct smooth points, each on
 one of $\lf{X}, \rt{X}$ are separated  by the canonical
 system.
Thus we conclude
\begin{lem}\lbl{sepcan-1node}Notations as above, let $z\in X\sbr 2._B$ be disjoint from $\theta$. Then
 \par (i) if $z$ is contained in one side of $\theta$, the image of the modified Brill-
Noether map $\phi_\theta$ coincides over $z$
with that of the
Brill-Noether map associated to the sepcanonical system;
\par (ii) if $z$  meets both sides of $\theta$, $\phi_\theta$
has maximal rank over $z$.

\end{lem}

 In this sense, the modified Brill-Noether
  realizes the goal of reviving dead sections. The case of schemes not disjoint from $\theta$ will follow from the calculations of the next subsection.\par
  It is worth noting at this point that the normal bundle of
  $\lf{D}$, which features in the intersection theory of the
 modifications $\be_\theta$, is easy to calculate.
First recall that by Faber's result in \cite{fab-alg}, the
conormal bundle to $\del_\theta\to B$ is 
$\psi_\theta=_L\psi_\theta\otimes_R\psi_\theta$, i.e. the tensor 
product of the branch cotangents, and it follows easily that
\eqspl{normal-1-var-eq}{
\check N_{\lx(\theta)/X}=\O(\lf{\theta})\otimes\psi_\theta
} (compare the proof of Lemma \ref{conormal-sep-lem}). This can easily be extended to the Hilbert scheme.
The formula uses the operation $[2]_*$ discussed in \S \ref{normal-bundles}.

 \begin{lem}\label{conormal-sep-lem}
 The conormal bundle of $\lf{D}=\lf{X}(\theta)\sbr 2._{\del_\theta}$ in
 $X\sbr 2._B$ is $[2]_*(\lf{\theta})\otimes\lf{\psi}\otimes\rt{\psi}$,
 where $\lf{\psi}=(\lf{\theta})^*(\omega_{\lf{X}/\del_\theta})$,
  considered as a line bundle on $\del_\theta$
  via viewting $\lf{\theta}$ as
  a cross-section $\lf{\theta}:\del_\theta\to \lf{X}/\del_\theta$, and likewise $\rt{\psi}$.
 \end{lem}
 \begin{proof}
 By  Faber's result \cite{fab-alg}, the conormal bundle of $\del_\theta$ in $B$ is $\psi_\theta$.
Hence there is an injection $\psi_\theta\to\O_{\lf{D}}(-\lf{D})$ which,
because $X\sbr 2._{\del_\theta}$ has normal crossings, vanishes simply on $[2]_*(\lf{\theta})$, which is the intersection of
$\lf{D}$ with the other components. This proves our assertion.
 \end{proof}
 For convenience, we will denote $[2]_*(\lf{\theta})$ by $(\lf{\theta})\sbr 2.$. We will also need to compute
 self-intersections of $(\lf{\theta})\sbr 2.$ within $\lf{D}$. This follows from the following
 \begin{lem}
 Let $X/B$ be a family of curves with a section  $\theta$ contained in the smooth part.
 Let $\theta\sbr 2.\subset X\sbr 2._B$ denote the locus of schemes meeting $\theta$
 and $S_\theta$ the locus of schemes equal to  $2\theta$ and $\psi_\theta=\omega_{X/B}.\theta$. Then
 \eqspl{}{
 (\theta\sbr 2.)^i\sim S_\theta.(-2\psi_\theta)^{i-1}, i\geq 1.
 }
 \end{lem}
 \begin{proof}
 The case $i=1$ is clear. For $i=2$, we need to compute the normal bundle of
 $\theta\sbr 2.$ restricted on $S_\theta$. In coordinates, a point in $S_\theta$
 is represented by a scheme $x^2$ and its neighborhood by a deformation
 of the form $x^2+ax+b$. The equation of $\theta\sbr 2.$ is $b$, which may be identified
 with $\omega_{X/B}.\theta$ via $(x+\eta)^2=x+2\eta x+\eta^2$. This shows the case $i=2$,
 and the case of higher $i$ follows.
 \end{proof}
\newsubsection{Local forms}\lbl{sep-local-forms-sec}
We analyze $\phi_\theta$ locally.
We work near a fibre
\[X_0=(\lx(0), \lf{\theta}(0))\cup (\rx(0), \lf{\theta}(0))\] comprised of 2
hyperelliptic pairs, as other cases are simpler. To begin with,  using the
results of \S\ref{normal forms} on normal forms, we have adapted bases
for the Hodge bundles $\bE_{_*X}$ so that the evaluation maps
$\bE_{_*X}\to_*\omega$ take the form \eqspl{}{
(1, x^2+z_1x, x^4+z_2x^3,...,x^{2g_1-2}+z_{g_1-1}x^{2g_1-3}),\\
(1, y^2+w_1y, y^4+w_2y^3,...,y^{2g_2-2}+w_{g_2-1}y^{2g_2-3})
}
such that the locus of hyperelliptic pairs $(_*X, _*\theta)$
is defined by $z_1=...=z_{g_1-1}=0$ or $w_1=...=w_{g_2-1}=0$.
Then the evaluation map $\bE\to \omega$ takes the form
(where we have written the $Y$ basis negatively, left of the bar)
\eqspl{eval-map-eq}{
(...,y^5+w_2y^2, y^3+w_1y^2, y| x, x^3+z_1x^2, x^5+z_2x^2, ...)
}

Next, we work on the Hilbert scheme,
locally near the locus of schemes supported on $\theta$; the case of the rest of the boundary of $\lf{D}$ (i.e. $\lx\times \rt{\theta}$) is similar and simpler.
In terms of the coordinates $u=y_2/x_1$ etc. above (\S\ref{hilb-review}), we will analyze $\be_\theta$ over the open set where $u$ is regular, the other case, where $v=1/u$ is regular, being similar. Set $\sigma_i=\sigma_i^x$. We will use the rules 
\eqref{sigma-table}. We factor the Brill-Noether  $\phi$
through $E_\theta$:
\eqspl{}{
(\pi\sbr 2.)^*(\be)\to\be_\theta\stackrel{\phi_\theta}{\to} \Lambda_2(\omega).
}
In the local coordinates above, factoring through $\be_\theta$
means the matrix for the factored mapping $\phi_\theta$
has $u$ factored out of the $-1$ column and $u^2$ factored out of the farther
negative columns, i.e.
\eqspl{normal-form-bn-modif-eq}{\left [
\begin{matrix}
...&u(\sigma_1^3-2\sigma_1\sigma_2)+w_1(\sigma_1^2-\sigma_2)&
\sigma_1&|&0&-\sigma_2\sigma_1-z_1\sigma_2&...\\
...&u(\sigma_2-\sigma_1^2)-w_1\sigma_1&-1&|&1&\sigma_1^2-\sigma_2+z_1\sigma_1&...
\end{matrix}\right ].
}
Here the matrix for $(\pi\sbr 2.)^*(\be)\to\Lambda_2(\omega)$ itself has the same right side as the above,
and left side equal to the $y$-mirror of the right side.
By applying suitable column operations, i.e. composing with
an automorphism of $E_\theta$, which does not affect degeneracy,
we can kill off all multiples of $\sigma_1$ in the top row, except
in column $-1$, then kill off the entire second row except in column
$+1$, ending up with
\eqspl{Lambda-theta}{\Lambda_\theta=\left [
\begin{matrix}
...&-w_1\sigma_2&
\sigma_1&|&0&-z_1\sigma_2&...\\
...&0&0&|&1&0&...
\end{matrix}\right ].
}

Now the ideal of $2\times 2$ minors of $\Lambda_\theta$ is
\begin{equation}\lbl{I_2-lambda-theta}{ I_2(\Lambda_\theta)= (\sigma_1, z_1\sigma_2,...,z_{g_1-1}\sigma_2,
w_1\sigma_2,...,w_{g_2-1}\sigma_2) }
\end{equation}
and note that
$u\sigma_2=ux_1x_2=y_2x_2=t$ is the equation of the boundary, which reflects the fact that in the open set we are considering (the one where $u$ is regular), the boundary is the union of $\lf{D}=(\lx)_\delta\sbr 2.$, with equation $u$, and $\lx\times \rx$, with equation $\sigma_2$. Also, $w_1,...,w_{g_2-1}$ define the locus
in $\delta$ where $(\rx, \rt{\theta})$ is a hyperelliptic pair.\par For future reference, it is important to note that
the above normal form \eqref{Lambda-theta} depends only on
the normal form \eqref{eval-map-eq} for the
(1-point) evaluation
map.\par
Now set \[ R_\theta=\zeros(\sigma_1, \sigma_2).\]
This is the inverse image of the cycle $2\theta$,
   a cross-section of $X_\delta\spr 2.$ . In fact, $R_\theta$ is
a $\P^1$-bundle over $\delta$, called the \emph{node scroll} associated to $\theta$.
Also let \[ T=\zeros(\sigma_1, z_1,...,z_{g_1-1}, w_1,...,w_{g_2-1})\]
This is a regular subscheme of codimension $g-1$,
 which coincides with $\he^2_g$ over the interior $B^o\subset B$
i.e. the set of smooth curves, and whose boundary portion
is the locus $\he_{g_1, g_2}^2$ of
pairs of pointed hyperelliptics, which sits in
a Cartesian diagram
\[\begin{matrix}
\he_{g_1, g_2}^2&\to&\delta_\theta\\
\downarrow&&\downarrow\\
\he^1_{g_1}\times\he^1_{g_2}&\subset&\delta_1\times\delta_2
\end{matrix}
\]
Now by \eqref{I_2-lambda-theta}, the degeneracy scheme $\bD_2(\phi_\theta)$ splits schematically
\eqspl{}{
\bD_2(\phi_\theta)=T\cup R_\theta.
}
Now because  $\he_{g_1, g_2}^2$ is of codimension $g-1$ in
the boundary, while $T$ is a priori purely of codimension at most $g-1$
overall, i.e. in $X\sbr 2._B$, it follows that $T$ must be contained in, hence
coincide with, the closure of the degeneracy locus
of $\phi=\phi_\theta$ in  $B^o$, i.e.
\[ T=\overline{\he^2_{B^o}}.\]
Furthermore, $T$ is transverse to the diagonal , which in the current local coordinates has equation $\sigma_1^2-4\sigma_2$.
Therefore, $T$ meets the diagonal, which is isomorphic
to the blowup of $X$ in $\theta$, with
exceptional divisor $R_\theta$, in a regular subscheme of codimension
$g-1$ of the diagonal.\par
Note we have shown all this near the 'finite' part of $R_\theta$, where $u$ is regular. By symmetry, it holds near the part where $v=1/u$ is regular, hence near  all of $R_\theta$.
Also, we note that all the above assertions are true and much easier to verify off $R_\theta$.
 We conclude
\begin{prop}\lbl{single-sep-prop} Given a separating node $\theta$ of $X/B$ and
the associated modification $\phi_\theta$ of the Brill-Noether map,
we have\begin{enumerate}\item The rank of $\phi_\theta$ is at least
1 everywhere.\item The degeneracy locus $\bD_2(\phi_\theta)$ splits schematically as
\eqspl{}{
R_\theta\cup \overline{\he^2_{B^o}}
} with $\overline{\he^2_{B^o}}$ regular of codimension
$g-1$.
\item The boundary is a transverse intersection:
\[ \overline{\he^2_{B^o}}\cap \partial_\theta=\he_{g_1,g_2}^{2}\]
\item The 1-point hyperelliptic locus $\overline{\he^1_{B^o}}$ is regular of codimension $g-1$ in the diagonal
    of $X\sbr 2._B$.
\item The intersection
\[R_\theta\cap \overline{\he^2_{B^o}}=R_\theta|_{\he^1_{g_1}\times\he^1_{g_2}}
\simeq \P_{\he^1_{g_1}\times\he^1_{g_2}}(\psi_1\oplus\psi_2).\]
\end{enumerate}
\end{prop}
\begin{cor}
A length-2 subscheme $\tau$ of $X_1\cup_\theta X_2$ is a limit
of a pair in the hyperelliptic involution of a smooth curve if and only if
$X_1, X_2$ are hyperelliptic and either
$\tau$ is in the  hyperelliptic involution on $X_1$ or $X_2$, or
$\tau$ is supported on $\theta$.
\end{cor}
As promised, the above analysis also allows us to strengthen slightly the result of Lemma \ref{sepcan-1node}:
\begin{cor}\lbl{sepcan-1node+}Suppose $z\in X\sbr 2._B$  is
contained in one side of $\theta$ but is not supported on $\theta$. Then the image of the modified Brill-
Noether map $\phi_\theta$ coincides over $z$
with that of the
Brill-Noether map associated to the sepcanonical system.
\end{cor}
 For application in inductive arguments, we mention another consequence of the above computations, to the behavior of
$\phi_\theta$ away from $\theta$:
\begin{cor}\lbl{mod-bn-sep-other-pt-cor}For $X_0$ as above,
if $p\in \lf{X}(\theta)$ is a Weierstrass point distinct from
$\lf{\theta}$, then a local matrix for the modified \bn
map $\phi_\theta$ near $2p$ has the form  \eqref{bn-rel-he2}
where $z_2, ...,z_g$ are local equations for the closure of the hyperelliptic locus near $0\in B$.
\end{cor}
\subsection{Case of multiples seps}\lbl{multi-sep-sec}
We now extend the previous results to the case of multiple separating nodes. Thus let $\Theta=(\theta_1,...,\theta_n)$ be a collection of separating nodes of the versal family $X/B$. Then we get
a collection of mutually transverse echelon data as in
\eqref{echelon-data-sep-node}:
\[\chi(\Theta)=(\lf{\chi}(\theta_i), \rt{\chi}(\theta_i), i=1,...,n).\]
Let $\be_\Theta$ be the associated multi-echelon
modification and
\[\phi_\Theta:\bE_\Theta\to\Lambda_2(\omega)\]
the resulting modification of the \bn map.
\begin{prop}\lbl{all-seps-prop}
Notations as above.
Then for any length-2 subscheme $z$ contained in a $\Theta$-separation
component of a fibre $X_0$ but not supported on $\Theta$, the
image of $\phi_\Theta$ over $z$ coincides with that of the Brill-Noether map of the $\Theta$-sepcanonical system of $X_0$.
Further, the degeneracy locus of $\phi_\Theta$ coincides
near $(X_0)\sbr 2.$ with
\[\overline{\he^2}\cup\bigcup\limits_{i=1}^nR_{\theta_i}.\]
\end{prop}
\begin{proof}
$Y$ be a $\Theta$-separation component of $X_0$ and assume as usual that  each $\theta\in\Theta$ is oriented with $Y$ on
its left ($Y$ may or may not meet $\theta$). Set
\[\be^{\Theta, Y}=\pi_*(\omega_{X/B}(-2\sum\limits_{i=1}^n\lf{X}(\theta_i))).\]
Then we get a \emph{multi-sep comparison diagram}:
\eqspl{multisep-comp-diag-eq}{
\begin{matrix}
\be^{\Theta,Y}(\sum\lf{D}(\theta_i))&\to&\be_\Theta\\
\downarrow&&\downarrow\phi_\Theta\\
\Lambda_2(\omega(-2\sum\lf{X}(\theta_i)))
(2\sum\lf{D}(\theta_i))&\to&\Lambda_2(\omega).
\end{matrix}
}
By \cite{canonodal}, Proposition 6.10, the left vertical map
coincides with the ordinary \bn map associated to
the $\Theta$-sepcanonical of $X$ restricted on $Y$. Moreover, for any  scheme $z$ on $Y$ disjoint from $\Theta$, the lower
horizontal map is an \iso. This proves the Proposition for all schemes $z$ on $Y$ disjoint from $\Theta$. The analysis
of the diagram at a scheme meeting $\Theta$ is
analogous to the case of a single sep in \S \ref{sep-local-forms-sec}. This concludes the proof.
\end{proof}
A nodal curve $X_0$ is said to be of  \emph{pseudo-
compact type} if it has no proper biseps;
or equivalently, if every connected component of
the separation of all seps of $X_0$
is 2-inseparable.
\begin{cor}\label{compact-type-cor}
If $X_0$ is a fibre of pseudo- compact type and $\Theta$
is the set of all seps occurring on $X_0$, , then near
$X_0\sbr 2.$, the degeneracy locus of $\phi_\Theta$ coincides schematically with
\[\overline{\he^2}\cup\bigcup\limits_{i=1}^nR_{\theta_i}.\]
\end{cor}
\begin{proof}
We will use Theorem 7.2 of \cite{canonodal} and argue as in
\S \ref{mod-bn-sep-interpret}.
To begin with, if $X_0$ is not hyperelliptic, then the fact that the sepcanonical system is essentially very ample,
together with a local analysis similar to \S \ref{sep-local-forms-sec} shows that the degeneracy locus
of $\phi_\Theta$ locally equals $\bigcup R_{\theta_i}$ .
Therefore we may assume $X_0$ is hyperelliptic. Pick
$\theta\in\Theta$, which is necessarily (bilaterally) hyperelliptic.
Then the sepcanonical system
 of $\lf{X_0}(\theta)$ is (simply) ramified on $\lf{\theta}$.
 By induction, as in Corollary \ref{mod-bn-sep-other-pt-cor},
  we may assume $\he^1_{\lf{g}}\times\rt{\del}$ is regular of
 codimension $\lf{g}-1=g(\lf{X_0}(\theta))-1$ in an appropriate versal
 family $\lf{\del}\times\rt{\del}$
 mapped to by $\del_\theta$. Likewise for the right side.
 Therefore, the local analysis as in \S \ref{mod-bn-sep-interpret} applies, yielding the conclusion.
\end{proof}

\newsection{Modifying Brill-Noether, case (ii):
disjoint separating binodes}
\lbl{bn-bisep}
Our next focus is on biseps, i.e. separating binodes.
That this requires a modification of the Brill-Noether map is already clear from the fact that, for any smooth connected curve $C$ with points
$p \neq q$, the system $|\omega_C(p+q)|$ fails to separate $p$ and $q$. Additionally, there are curves with a binode that is
hyperelliptic on one
 side only, for which \bn drops rank even though they are not limits of hyperelliptics. Because this situation occurs in codimension 2, some birational modification of Hilb will be required as well before an echelon
modification can be applied to the Hodge bundle.
We begin in \S \ref{azimuthal-curves-sec}  by analyzing a 'naive' approach to this issue,
based on modifying
the family of curves. The approach we will actually
use, based on modifying the
Hilbert scheme directly, will be considered subsequently, starting in \S \ref{mod-hilb-sec}.
\subsection{Azimuthal curves : Single
binode}\lbl{azimuthal-curves-sec}
 Fix a bisep, i.e. a properly separating binode $\vtheta$ on our family $X/B$, with corresponding locus \[\del=\del_\vtheta\subset B\] and sides $\lx, \rx$ defined over $\del$.
We view $\vtheta$ as \emph{oriented}, i.e. with fixed
choice of sides.
 Typically, $\vtheta$ will be defined only over a suitable analytic
or \'etale neighborhood of a boundary curve.
 Our purpose in this subsection is to construct a modification of $X/B$, a new family called an azimuthal modification,
  in which $\lx$, at least, becomes a Cartier divisor.
 One could then look at the Hilbert scheme of the new family. In the sequel, we will actually do something a little different, viz. modify the Hilbert scheme directly.
 Nonetheless, comparing the two constructions will be important in analyzing and interpreting the latter one. \par
 The idea of the construction is an obvious one:
 first blow up $\del_\vtheta$, then blow up the inverse image
 of $\lf{X}$, which is a Weil divisor. We proceed with the
 details.\par
First, let $B(\vtheta)$ be the blowup
of $B$ in $\del_\vtheta$, with exceptional divisor
\[P(\vtheta)=\P(\psi_1\oplus\psi_2),\  \psi_i=_L\psi_i\otimes_R\psi_i,\]
and let $X_{B(\vtheta)}=X\times_B B(\vtheta)$ be the base-changed family. Let $S_i\subset P(\vtheta)$ denote the section
$\P(\psi_i)$. We recall from Definition 6.1 of \cite{canonodal} that elements of $P(\vtheta)$ are called middle {azimuths}
at $\vtheta$,
irregular or regular according as they are in $S_1\cup S_2$ or not.\par
 In local coordinates $x_iy_i=t_i$ near $\theta_i, i=1,2$,
a local coordinate on a suitable open in $B(\vtheta)$ containing
$P(\vtheta)\setminus S_1$ is
$w=[t_2/t_1]$ and there, the base-changed family is given by
\eqspl{base-changed-eq}{ x_1y_1=t_1, \ x_2y_2=t_1w.}
The inverse image of $\lx$ here is defined near $\theta_1$ resp. $\theta_2$
by $y_1$ resp. $ (y_2, t_1)$, so it is a Weil divisor,
non $\Q$-Cartier near $\theta_2$;
similarly for $\rt{X}$.
\par Now let \[b:X_{ B(\vtheta),L}\to X_{B(\vtheta)}\] denote the blowup of $X_{B(\vtheta)}$
 along the Weil divisor $\lf{X}$. The family $X_{B(\vtheta), L}/B(\vtheta)$ is called
the \myit{left azimuthal modification} of $X/B$ corresponding
to the binode $\vtheta$. The construction may be diagrammed as follows
\eqspl{}{
\begin{matrix}
X_{B(\vtheta), L}&\to&X_{B(\vtheta)}&\to&X\\
&\searrow&\downarrow&\square&\downarrow\\
&&B(\vtheta)&\to&B
\end{matrix}
} with the upper left horizontal arrow a small blowup.
Using local coordinates as above, especially
\eqref{base-changed-eq}, it follows that
$X_{B(\vtheta), L}$ is \emph{virtually smooth} (smooth if
$X/B$ is locally versal at $\vtheta$). \par
Locally near the inverse image of $\theta_2$,
where  the function $w$ is regular,  $X_{B(\vtheta), L}$ is covered by opens $U_1, U_2$ so that
$[t_1/y_2]$ (resp. $ [y_2/t_1]$) is regular in $U_1$
(resp. $U_2$). There is
a natural
  'left
side' divisor $\lx(\vtheta)$,   defined locally by $y_2$ in
$U_1$,  or $t_1$ in $U_2$,
which coincides with the inverse image in the blowup of \[\lx=(t_1=y_2=0).\] Recall that we are working over a
neighborhood of $P(\vtheta)\setminus S_1$ in $B(\vtheta)$.
Over this neighborhood, the exceptional locus of the
blowup is a $\P^1$ bundle $P_{\theta_2}$ over $\vtheta_2|_{S_2}$, which is of codimension 2.
Note $P_{\theta_2}$ has equations \[t_1= [t_2/t_1]=
 x_2= y_2=0.\]
In $U_1$, $P_{\theta_2}$ is defined by $y_2, [t_2/t_1]$ (since
\[t_1=[t_1/y_2]y_2, \ x_2=[t_1/y_2][t_2/t_1]).\]
Similarly in $U_2$, $P_{\theta_2}$ is defined by $t_1, x_2$.
All in all, the exceptional locus of the blowup
consists of $P_{\theta_1}\cup P_{\theta_2}$.
\par This construction is decidedly not left-right symmetric however.
In $U_1$,
the inverse image of
$\rx$ is $\Zeros (t_1=[t_1/y_2]y_2, x_2=[t_1/y_2][t_2/t_1])$.
Therefore it consists of the divisor of $[t_1/y_2]$, which is
 the proper transform $\rt{\tilde X}$ of $\rt{X}$, plus
 $P_{\theta_2}$.
 In $U_2$, the inverse image of $\rt{X}$ coincides with $P_{\theta_2}$ and $\rt{\tilde X}$ is disjoint from it.
 Thus, denoting the blowup map by $b$, we have all in all
 \[b\inv(\rt{X})=\rt{\tilde X}\cup P_{\theta_1}\cup P_{\theta_2}\]
 where $\rt{\tilde X}$ is a Cartier divisor mapping isomorphically to $\rt{X}$. On the other hand,
 \[P_{\theta_1}\cup P_{\theta_2}\subset \lf{X}(\vtheta)\]
 therefore $\lf{X}(\vtheta)$, the inverse image of $\lf{X}$,
 is a divisor.
 In fact, $\lf{X}(\vtheta)$ is just the blowup of $\lf{X}_{P(\vtheta)}$ in its regular codimension-2
 subvariety $\theta_1|_{S_1}\coprod \theta_2|{S_2}$
 (this has codimension 3 in $X_{B(\vtheta)}$). In the above
 coordinates, $\theta_2|{S_2}$ is defined by $[t_2/t_1]=x_2=0$.

The full boundary of the azimuthal modification
 is the union:
\eqspl{full-boundary}{ X_{L, B(\vtheta)}\times_{B(\vtheta)}P(\vtheta)
=\lf{X}(\vtheta)\cup\rt{\tilde X}}
and its fibres (with $\vtheta$ given)
are called \emph{left azimuthal curves}, regular or irregular, or curves with a \emph{left azimuthal binode} in $\vtheta$.
\begin{remarks} We list some elementary properties of this construction.\par\begin{enumerate}[1)]
\item
A (left or right) regular azimuthal fibre
is isomorphic to the corresponding fibre over $B$.
In this case, the left and right azimuthal structures on $\vtheta$ are the same, constant over the fibre, and simply given by
a regular element (\emph{azimuth})
\eqspl{regular-azimuth}{\a\in\P(\psi_{\theta_1}\oplus \psi_{\theta_2})
=\P(T^*_{\lx,\theta_1}\otimes T^*_{\rx,\theta_1}\oplus
T^*_{\lx,\theta_2}\otimes T^*_{\rx,\theta_2}).}

\item As for left irregular azimuthal fibres,
living over $S_1\coprod S_2$,
$S_1$ is the (transverse) intersection of $P(\vtheta)$ with the proper transform of $\del_{\theta_1}$, the boundary divisor corresponding to $\theta_1$. It parametrizes
semistable triangular curves of the form { $Y=\nolinebreak\lx\cup\rx\cup F$}
where $F$ is a bridge $\P^1$ meeting $\lx, \rx$ in  $\lf{\theta}_1, \rt{\theta}_1$,
respectively and $\lx$ and $\rx$ meet additionally in $\theta_2$. Thus $Y$ is stably equivalent to $\lf{X}\cup\rt{X}$, the corresponding fibre of $X/B$.
Note that by construction, $F\subset \lx(\vtheta)$ ('$F$ tilts left')
(because we are in the left azimuthal modification).
Going out of the point $[Y]\in S_1$ in the direction of $P(\vtheta)$ corresponds to smoothing the node $\lf{\theta}_1$, keeping the other two as nodes.
Going out in the direction of $\del_{\theta_1}$ does the opposite,
smoothing $\rt{\theta}_1$ and $\theta_2$ and keeping $\lf{\theta}_1$ as a node.\par
Obviously, analogous comments apply with 1 replaced by 2 and L
replaced by R. Thus, a  bridge at either $\theta_i$
can tilt left or right.
\item It follows from \eqref{full-boundary} that the normal bundle
to $\lf{X}(\vtheta)$ is
\eqspl{normal-bundle-azimuth}{
\check N_{\lf{X}(\vtheta)/X_{\rmL, B(\vtheta)}}=\pi^*(\O_{P(\vtheta)}(1))
\otimes\O(\lf{\vtheta}')
} where $\lf{\vtheta}'$ is the pullback of $\lf{\vtheta}$, which coincides with $\vtheta$ over the regular azimuthal fibres; 
over the irregular ones, the node $\theta_i$ is replaced by the
other node on $F$ (i.e. where $F$ meets $\rt{X})$. Here
the first factor corresponds to a choice of middle azimuth
(see Definition 6.1 of \cite{canonodal}). We have
\eqspl{}{\check N_{\lf{X}(\vtheta)/X_{\rmL, B(\vtheta)}}|_{\lf{\theta_i}}=\rt{\psi}_i}
In particular, for a regular azimuthal fibre,
$\check N_{\lf{X}(\vtheta)/X_{{\rm L}, B(\vtheta)}}|_{\vtheta}$
parametrizes \emph{right} azimuths at $\vtheta$.

\item We refer to \cite{canonodal}, \S 6 (especially Definition 6.1).
A \emph{hyperelliptic} binode $\vtheta$
on a fibre $X_0$ admits a uniquely determined regular azimuthal
structure $\zeta=(\lf{\zeta}, \md{\zeta}, \rt{\zeta})$ called a
\emph{hyperelliptic azimuth}. This results from the hyperelliptic
identification between left (resp. right) cotangents at $\theta_1$ and $\theta_2$. Then
$X_0$ corresponds to a unique regular azimuthal fibre on either the
right or left azimuthal modifications. \par On the other hand if a
regular azimuthal fibre (left or right) $X_{\md{\zeta}}$ corresponds
to a fibre $X_0$ of $X$ over $\del_\vtheta$
 where $\vtheta$ is a left- hyperelliptic binode,
 then $X_{\md{\zeta}}$
carries a canonical  compatible azimuthal structure
$(\lf{\zeta}, \md{\zeta}, \rt{\zeta})$ where
$\rt{\zeta}$ is induced by $\lf{\zeta}, \md{\zeta}$ as in
Definition 6.1 of \cite{canonodal}.
\item The constructions and results of this subsection extend
directly to a collection $\Theta=\{\vtheta_1,...,\vtheta_n\}$ of
disjoint biseps . This is because the boundary loci
$\del_{\vtheta_i}$ are in general position so their common blowup
$B(\Theta)$ is well behaved (smooth if $B$ is smooth and $X/B$ is
versal). Also the 
Weil divisors $\lf{X}(\vtheta_i)\times_{\del_{\vtheta_i}}B(\Theta)$
are in general position.
\par\end{enumerate}
\end{remarks}

\newsubsection{Modifying Hilb}\lbl{mod-hilb-sec}
In this section we construct azimuthal modifications
applied directly on the Hilbert scheme.
We continue with the notation and assumptions of the previous
subsection. Consider
the
regular, codimension-2 subvarieties
\[\lf{X}\sbr 2._{\del}, \rt{X}\sbr 2._{\del}\subset X\sbr 2._B.\]
These meet transversely in the codimension-4 subvariety
$\vtheta_{\del}$ which is a cross-section over $\del$.
The transversality (even just properness) of the
intersection easily implies that the ideal of the schematic union
 is a product
 \[\cJ_\vtheta:=\I_{(\lf{X})\sbr 2._{\del}\cup (\rt{X})\sbr 2._{\del}}=\I_{(\lf{X})\sbr 2._{\del}}
  \I_{(\rt{X})\sbr 2._{\del}},\] and moreover that the successive blowup of $X\sbr 2._B$ in
  $(\lf{X})\sbr 2._{\del}$ and $ (\rt{X})\sbr 2._{\del}$ in either order is isomorphic
  to the blowup of $X\sbr 2._B$ in $(\lf{X})\sbr 2._{\del}\cup  (\rt{X})\sbr 2._{\del}$
  (see also the local calculations below).\par
  We also
 consider the base-changed family $X_{B(\vtheta)}$. Let $P=P(\vtheta)$
and  consider the Weil divisors
\[(\lf{X})\sbr 2._P, (\rt{X})\sbr 2._P\subset X\sbr 2._{B(\vtheta)}=X\sbr 2._B\times_B B(\vtheta).\]
These  intersect  in $\vtheta_P$ itself,
considered as relative length-2 subscheme of $X_P/P$, thus forming a section of $X\sbr 2._P$ and having codimension 3(!)
in $X\sbr 2._{B(\vtheta)}$. We denote by $\cJ'_\vtheta$ the product ideal
\[\I_{(\lf{X})\sbr 2._P}.\I_{ (\rt{X})\sbr 2._P},\]
i.e. the pullback ideal of $\cJ_\vtheta$.

\begin{defn} The \myit{azimuthal Hilbert scheme}
\mbox{ (resp. \myit{extended azimuthal Hilbert scheme})}
 of $X/B$ associated to
$\vtheta$, denoted $X\sbc 2._B(\vtheta)$
(resp. $X\sbc 2._{B(\vtheta)}$) is
 the blowup of the ideal $\cJ_\vtheta$ on $X\sbr 2._{B}$
  (resp.
$ \cJ'_\vtheta$ on $X\sbr 2._{B(\vtheta)}$).
\end{defn}
A point of $X\sbc 2._B(\vtheta)$ is called an azimuthal length-2 scheme of $X/B$ with respect to $\vtheta$. We will see in Corollary \ref{interior-azimuthal-cor} that in the case of
an interior scheme (strictly to one side of $\vtheta$), an azimuthal scheme in fact consists of a scheme together with
a (left or right) azimuth at $\vtheta$.
We denote by $\lf{D}(\vtheta), \rt{D}(\vtheta)$ the respective
inverse images of $(\lf{X}(\vtheta))\sbr 2._P, (\rt{X}(\vtheta))\sbr 2._P$ in $X\sbc 2._B(\vtheta)$.\par
Note the natural map \[X\sbc 2._{B(\vtheta)}\to X\sbc 2._B(\vtheta)\times _B B(\vtheta);\]
in fact $X\sbc 2._{B(\vtheta)}$ can be identified with the unique component of $X\sbc 2._B(\vtheta)\times_B B(\vtheta)$ dominating $X\sbc 2._B(\vtheta)$.
Clearly, $X\sbc 2._{B(\vtheta)}$ and $X\sbc 2._B$ are
right-left symmetric, i.e. independent of the orientation of the binode $\vtheta$.
Though $X\sbc 2._{B(\vtheta)}$ is in some sense more natural
(e.g. it is flat over $B(\vtheta)$), the
objects we need to work with are defined already over
its smaller relative $X\sbc 2._B(\vtheta)$, so the latter
is ultimately preferable.
\par

 We describe $X\sbc 2._{B(\vtheta)}$ and $X\sbc 2._B(\vtheta)$ in local coordinates over the point $\vtheta$; other local descriptions are
simpler. We may assume the family is given in local analytic coordinates
near $\theta_1, \theta_2$ respectively by
\[ x_1y_1=t_1, \ x_2y_2=t_2\]
where $x_1, x_2$ and $y_1, y_2$ are respective local coordinates
on open sets $\lf{X}_1, \lf{X}_2\subset \lx$, $\rt{X}_1, \rt{X}_2\subset \rx$.
Locally analytically near $\vtheta$, $X\sbr 2._B$ is a
Cartesian product $(\lf{X}_1\cup \rt{X}_1)\times (\lf{X}_2\times \rt{X}_2)$.
Let $u=u_2/u_1$
be an affine coordinate on the finite part of
the projectivized normal bundle
$\P(N_{(\lf{\theta}_1,\lf{\theta}_2)}{\lx_1\times \lx_2})$,
where $u_1, u_2$ are linear coordinates on the normal bundle
(these will not be needed in the sequel).
For later reference, the latter bundle can be identified
as $_L\psi_1\oplus_L\psi_2$ with the obvious notations.
On  the blowup $\Bl_{\vtheta}X\sbr 2._B$, we can write
$u=x_2/x_1$ so $u$ measures the 'speed ratio' as 2 points approach the $y$-axis.
Similarly for $v=v_2/v_1=y_2/y_1$ on the right side and $w=w_2/w_1=t_2/t_1$
on the base (i.e. $w$ is an affine coordinate on $P(\vtheta)/\del_\vtheta$).
So we can cover the exceptional locus
in $X\sbc 2._{B(\vtheta)}$
by 8 opens where ($u$ or $u'=u\inv$) and ($v$ or $v'=v\inv$) and ($w$ or $w'$) are regular; in the case of $X\sbc 2._B(\vtheta)$, the $w$'s don't exist and 4 opens suffice. On the open where $u,v, w$ are regular, the equations for
 $X\sbc 2._{B(\vtheta)}$ are given by
\eqspl{}{
x_1y_1=t_1, x_2=ux_1, y_2=vy_1, t_2=wt_1\\
w=uv }
In the case of $X\sbc 2._B(\vtheta)$, the equations are
\eqspl{}{
x_1y_1=t_1, x_2=ux_1, y_2=vy_1, \\
t_2=x_1y_1uv }
In particular both are virtually smooth there,
  with local
coordinates $x_1, y_1, u,v$ relatively over
 the base ($B$ or $B(\vtheta)$). Here $x_1, y_1$
are equations for $\lf{D}, \rt{D}$ respectively and
 in the case of $X\sbc 2._{B(\vtheta)}$ these together
make up the entire boundary component corresponding to $\vtheta$,
that is, $P=P(\vtheta)$ with equation
$t_1=x_1y_1$. Analogous equations on the other covering opens
show that $X\sbc 2._B(\vtheta)$ is smooth. For example,
on the open where $u,v',w$ are regular, we have
local coordinates $x_1, y_2, v',w$ and the boundary, with equation $t_1=x_1v'y_2$,
 consists of $\lf{D}$ with equation
$y_2$, $\rt{D}$ with equation $x_1$, and $\lf{X}\times\rt{X}$, with equation $v'$.

\newsubsection{Normal bundles}\label{normal-bundles}
Recall that in Lemma \ref{conormal-sep-lem} we computed the conormal bundle to the divisor $(\lf{X}(\theta))\sbr 2._B$ in $X\sbr 2._B$, for a sep $\theta$.
We now extend this to the case of a bisep and the associated
azimuthal modification of the Hilbert scheme.
For a line bundle $L$ on $X$, we denote by $[2]_*L$ or sometimes
 $L\sbr 2.$ its 'norm', a line
bundle on $X\sbr 2._B$ (cf. \cite{internodal}).
Succinctly, \[[2]_*L=L\sbr 2.=\det(\Lambda_2(L))\otimes\O(\Gamma\spr 2.)\]
where $\Gamma\spr 2.$ is the discriminant divisor.
Similarly we have $[2]_*\frak b$ for a divisor $\frak b$ on $X/B$ (for $\frak b$ prime, $[2]_*\frak b$ is geometrically the locus of schemes meeting $\frak b$).
Let $\vtheta$ be a bisep, i.e. a
 properly separating oriented binode on a nodal curve $X$, with sides $\lx(\vtheta), \rx(\vtheta)$, defined over a codimension-2 locus $\del_\vtheta$.
Define the (linear)
\myit{left azimuth bundle} on $(\lx(\vtheta))\sbr 2.$ associated to this data as

\eqspl{}{
\lf{\tilde\a}_{X,\vtheta}=[2]_*(\O(\lf{\theta}_1)\otimes\psi_{\theta_1})\oplus
[2]_*(\O(\lf{\theta}_2)\otimes\psi_{\theta_2})
}
\begin{lem}\lbl{azim-bundle-lem} There is a canonical isomorphism
\eqspl{azimuth bundle formula}{\lf{\tilde\a}_{X,\vtheta}\simeq\check N_{(\lx(\vtheta))\sbr 2._{\del_\vtheta}/X\sbr 2._B}.}\end{lem}
\begin{proof}Indeed there is a canonical map
\[\psi_{\theta_1}\oplus\psi_{\theta_2}=\check N_{\del_\vtheta/B}\to \check N_{(\lx(\vtheta))\sbr 2._{\del_\vtheta}/X\sbr 2._B}=:
\check N.\]
The component map $\psi_{\theta_i}\to\check N$ drops rank precisely on the divisor $[2]_*(\lf{\theta}_i)$
so the saturation of the image of the component map
 is exactly
$[2]_*(\O(\lf{\theta}_i)\otimes\psi_{\theta_i})$. Putting the
two components together, we get a canonical map
$\lf{\tilde\a}_{X,\theta}\to\check N$. This map is obviously an
 isomorphism in codimension 1, hence we obtain \eqref{azimuth bundle formula}.
\end{proof}

An element of the projectivization   $\lf{\a}_{X,\theta}:=\P(\lf{\tilde\a}_{X,\theta})$ is
called a \myit{left azimuthal scheme} of length 2 (or left azimuthal structure
on the underlying scheme). This projectivization coincides
 with the exceptional divisor of the blowup of $(\lf{X})\sbr 2._B$ on $X\sbr 2._B$.
 Note that it contains 2 distinguished sections, viz.
$\P([2]_*(\O(\lf{\theta}_i)\otimes\psi_{\theta_i})), i=1,2$;
elements of these are called \emph{singular} azimuthal schemes,
while elements of the complement of their union are said to
be \emph{regular}.

 \par Note that the fibre of $\lf{\tilde\a}_{X, \vtheta}$ at
$\lf{\theta}_1+\lf{\theta}_2$ is naturally isomorphic (by residue)
to $\rt{\psi}_{\theta_1}\oplus\rt{\psi}_{\theta_2}$. On the other hand the
fibre at any 'interior' scheme, i.e. one disjoint from $\lf{\theta}_1+\lf{\theta}_2$, is naturally
isomorphic to $\psi_{\theta_1}\oplus\psi_{\theta_2}$. Thus
\begin{cor}\lbl{interior-azimuthal-cor}
Notations as above, an interior left [resp. left regular ] azimuthal scheme of $X$ with respect to $\vtheta$
consists of an interior subscheme of $\lf{X}(\vtheta)$
together with a middle
[resp. middle regular] azimuth of $X$ at $\vtheta$.
\end{cor}
We need to work out the relevant normal bundles on the
azimuthal Hilbert scheme $X\sbc 2._B(\vtheta)$.
We denote by $\lf{D}(\vtheta), \rt{D}(\vtheta)$, respectively, the
projectivized normal bundles of $(\lx)\sbr 2., (\rx)\sbr 2.$ in $X\sbr 2._B$,
which coincide with the exceptional divisor of the respective blow up.
By Lemma \ref{azim-bundle-lem}, we can  identify
\eqspl{}{
\lf{D}(\vtheta)\simeq\P_{(\lf{X})\sbr 2._D}(\lf{\theta}_1\sbr 2.\otimes\psi_{\theta_1}
\oplus \lf{\theta}_2\sbr 2.\otimes\psi_{\theta_2})
} where we use the abbreviation $\theta\sbr 2.$ for $[2]_*\theta$.
So this is a split $\P^1$ bundle with a simple intersection theory.
Ditto for the other side. \par
We denote by $\lf{D}^\dag(\vtheta), \rt{D}^\dag(\vtheta)$ the respective inverse images of
$(\lf{X})\sbr 2._D, (\rt{X})\sbr 2._D$ on $X\sbc 2._B$, both Cartier divisors and admitting a natural map, respectively,  to  $_*D(\vtheta), *=L,R$.
Moreover,  $\lf{D}^\dag(\vtheta)$ is the projectivized normal bundle of the inverse
image of $(\lx)\sbr 2._D$ on  the blowup of $(\rx)\sbr 2._D$. This inverse image
is just the blowup of $(\lx)\sbr 2._D$ in the section $\lf{\vtheta}=\lf{\theta_1}+\lf{\theta_2}$. By transversality, the normal bundle of the inverse image is just the
pullback of the normal bundle. Therefore,
\eqspl{D-vtheta-eq}{
\lf{D}^\dag(\vtheta)\simeq\P_{B_\vtheta(\lf{X})\sbr 2._D}(\lf{\theta}_1\sbr 2.\otimes\psi_{\theta_1}
\oplus \lf{\theta}_2\sbr 2.\otimes\psi_{\theta_2})
} and therefore, with respect to this identification,
\eqspl{conormal-D-vtheta-eq}{
\O_{\lf{D}^\dag(\vtheta)}(-\lf{D}^\dag(\vtheta))=
\O(1).
}
We call \emph{interior} points of ${X}\sbr 2._B$,  the points
corresponding to schemes
disjoint from $\vtheta$. Similarly for spaces equipped with
a map to $X\sbr 2._B$, such as its subschemes, e.g. $\rt{X}\sbr 2.$.
\par
\newsubsection{Modifying Brill-Noether}\lbl{bisep-modifying-bn-sec}
We will construct a modification of the Hodge bundle on the
azimuthal Hilbert scheme $X\sbc 2._B(\vtheta)$.
Essentially, there are two independent parts to
this modification,
for the left and right
sides. \par
Let $\bE$ be the pullback of the Hodge bundle on
the azimuthal Hilbert scheme $X\sbc 2._B(\vtheta)$.
Over $\del:=\del_\vtheta$, we have exact
 \[\exseq{\pi_*(\omega_{\rt{X}})}{\bE_{\del_\vtheta}}
 {\pi_*(\omega_{\lf{X}}
(\lf{\theta_1}+\lf{\theta_2}))}.\]
Let $\lf{\bE}, \lf{\bE}^0$ denote the respective pullbacks of ${\pi_*(\omega_{\lf{X}}
(\lf{\theta_1}+\lf{\theta_2}))}, {\pi_*(\omega_{\lf{X}}}$ on
$\lf{D^\dag}=\lf{D^\dag(\vtheta)}\subset X\sbc 2._B(\vtheta)$, i.e. the pullback of $\lf{X}\sbr 2._{\del}$.
Similarly for $\rt{\bE}$.
Then over $X\sbc 2._B(\vtheta)$, we get exact
 \[\exseq{\bE^{-1,0}}{\bE}{ \lf{\bE}}\]
  \[\exseq{\bE^{0,-1}}{\bE}{ \rt{\bE}}\]
  (these define $\bE^{.,.}$), and also
  \eqspl{lfE-bisep}{
  \exseq{\rt{\bE}^0}{\bE_{\del_\vtheta}}{\lf{\bE}},\\
    \exseq{\lf{\bE}^0}{\bE_{\del_\vtheta}}{\rt{\bE}}
  .} Hence we get a (clearly transverse) pair of echelon data (with $D=\delta$)
\eqspl{echelon-data-sep-binode}
{\lf{\chi}(\vtheta)=(\bE^{-1,0}(\vtheta), \lf{D^\dag}),
\rt{\chi}(\vtheta)=(\bE^{0,-1}(\vtheta), \rt{D^\dag}).
}
Note that the pullback of $\be^{-1,0}$ on $X\sbc 2._{B(\vtheta)}$ coincides with the pullback
from $B(\vtheta)$ of 
$\pi_*(\omega(-\lf{X}))$. In this way
this $\be^{-1,0}$ is similar to the analogously-denoted
bundle in the separating node case, see \S \ref{bn-sep}.
These give rise to an echelon modification $\be_\vtheta$,
through which
 the Brill-Noether map factors  yielding a
map that we call the \emph{modified} Brill-Noether map with respect
to $\vtheta$ \eqspl{}{ \phi_{\vtheta}:\bE_{\vtheta} \to
\Lambda_2(\omega).
 }\par
\begin{rem}
Rather than take as our starting point the Hodge bundle
itself, we could take its modification $\be_\Theta$ with
respect to any collection of seps (see \S \ref{multi-sep-sec}). More general cases will be considered below (e.g. \S \ref{multi-bisep-sec}).
\end{rem}
 Next we develop a convenient local normal matrix form for
 the \bn map and
 its modification $\phi_\vtheta$ and derive an important dimension
 count for the degeneracy locus of $\phi_\vtheta$.
 As $\phi_\vtheta$ is a map of vector bundles, its
 degeneracy locus can be studied fibrewise, i.e. the
 fibre of the degeneracy locus over $0\in B$ coincides with
 the degeneracy locus of the restriction of $\phi_\vtheta$
 over $X_0\sbc 2.$  We work
 on a fixed fibre over $\del_\vtheta$, $X_0=\lx\cup\rx$. 
We will
work near $\vtheta_0:=\vtheta\cap X_0$
and use the obvious local basis for $\Lambda_2$
with values $(1,0), (0,1)$ at $(\theta_1, \theta_2)$.
 \par \ul{Case 1:} $\vtheta_0$ is (bilaterally) hyperelliptic on
 $X_0$.\par
 Then for each
$i>0$ there is a section $s_{-i}$ vanishing on $\rx$ and having
local form $(x_1^i, x_2^i)$ on $\lx$ near
$(\lf{\theta}_{01},\lf{\theta}_{02})$ (recall that $(\lx,
\lf{\theta}_{01}+\lf{\theta}_{02})$ is hyperelliptic). Similarly on the
right, with local coordinates $y_i$. Then the \bn matrix takes the
form
\[\left [
\begin{matrix}
...&x_1^2&x_1&1&y_1&y_1^2&...\\
...&x_2^2&x_2&1&y_2&y_2^2&...
\end{matrix}
 \right ]\]
 where $X_0$ is locally defined by $x_1y_1=x_2y_2=0$,
 Now the azimuthal Hilbert scheme of $X_0$ is a $\P^1\times\P^1$
 -bundle over the Hilbert scheme,
 locally covered by 4 open affines \mbox{$\A^1\times\A^1$}
  which may be labelled
 (finite, finite) etc.
 On the (finite, finite) part of the azimuthal Hilbert scheme
 $X_0\sbc 2.(\vtheta)$, where $u,v$ are regular, the matrix
 $\Lambda_{\vtheta}$ for $\phi_\vtheta$  takes the form
 \eqspl{modif-bn-matrix-hypell}{\left [
\begin{matrix}
...&x_1&1&1&1&y_1&...\\
...&x_1u^2&u&1&v&y_1v^2&...
\end{matrix}
 \right ]}
 which drops rank
on the codimension-2 locus $u=v=1$. Note that this open set only meets $\lf{D^\dag}$, defined by $y_1$,
and $\rt{D^\dag}$, defined by $x_1$; it misses the
'mixed' part of the (azimuthal) Hilbert scheme,
and misses as well the locus of singular azimuths $uv=0$.
In $\lf{D^\dag}$,  $u=1$ defines
the graph of the hyperelliptic involution while $v=1$ is the
hyperelliptic right azimuth, i.e.  the section of the $\P^1$-bundle
parametrizing azimuths, corresponding to the
hyperelliptic azimuth on $\rt{\vtheta}$.
In $\lf{D}^\dag$, $v=1$ corresponds to the condition
on sections of $\omega_{\lf{X}(\vtheta)}(2\lf{\vtheta})$
that they have the value at $\lf{\theta}$ dictated by
the (left) hyperelliptic azimuthal condition.
 Similarly for
 $\rt{D^\dag}$. The other open sets covering the azimuthal Hilbert scheme are handled similarly. We conclude that
overall in this case, the fibre over 0 of the degeneracy locus of
$\Lambda_{\vtheta}$  has
 the form $x_1y_1=0$, i.e. is itself a nodal curve.
\par
\ul{Case 2:} $\vtheta_0$ not bilaterally hyperelliptic on
$X_0$.\par
We may assume that $(\rx, \rt{\theta}_1+\rt{\theta}_2)$ is non-
 hyperelliptic as a pair. Then in the above calculation the basis on the right side can be taken
 of the form $(y_1, y_2^2), (y_1^2, y_2),...$ and the corresponding
 modified \bn matrix will take the form
  \eqspl{modif-bn-matrix-rt-non-he}{\left [
\begin{matrix}
...&x_1&1&1&1&y_1&...\\
...&x_1u^2&u&1&v^2y_1&v&...
\end{matrix}
 \right ]}
 which has maximal rank near $y_1=0$. Off the locus $y_1=0$,
 the right half of the matrix is equivalent to the Brill-Noether
 matrix of $\rx$ itself, hence has maximal rank. Therefore,
 $\phi_\vtheta$ has maximal rank everywhere.
 \par
 We can summarize the foregoing discussion as follows.
 \begin{prop}
 Let $\vtheta$ be a properly separating binode of $X/B$ with
 associated boundary locus $\del=\del_\vtheta$ and modified
 \bn map $\phi_\vtheta$. Then for any fibre $X_0$ in the interior $\del^0$ of
 $\del$, i.e. having no seps or biseps off $\vtheta$, the intersection
 $\bD_2(\phi_\vtheta)\cap X\sbc 2._0$ is 1-dimensional if
 $X_0$ is
 hyperelliptic and empty otherwise. When $X_0$ is hyperellptic, the intersection consists of the loci of hyperelliptic
 subschemes on each side of $\vtheta$, endowed with
 the hyperelliptic azimuth at $\vtheta$.
 \end{prop}
It follows from the Proposition that, assuming $X/B$ is versal, the codimension of
$\bD_2(\phi_{\vtheta})\cap  \pi\inv(\del)$
 in $B_{\vtheta}X\sbr 2.$ is at least
 \[(_Lg-1)+(_Rg-1)+2+1=g\]
 (codimensions of hyperelliptic 2-pointed curve loci
  on left and right plus
 codimension of binodal locus plus fibre codimension). This number
 is greater than the expected codimension of the degeneracy locus
  $\bD_2(\phi_\vtheta)$ of $\phi_\vtheta$, which
 equals $g-1$.  Therefore this boundary locus
 cannot contribute a component
 to $\bD_2(\phi)$.  Thus
 \begin{cor} We have, schematically,
 \eqspl{}{\bD_2(\phi_{\vtheta})\cap
  \pi\inv(B^0\cup\del^0)=
  \overline{\bD_2(\phi_\vtheta)\cap\pi\inv( B^0)}\cap
   \pi\inv(B^0\cup\del^0)
  =\overline{\bD_2(\phi)\cap \pi\inv(B^0)}\cap  \pi\inv(B^0\cup\del^0)
  } where $B^0\subset B$ is the interior, parametrizing smooth curves.
In particular, every point of\nl \mbox{$\bD_2(\phi_{\vtheta})\cap
  \pi\inv(\del^0)$} is a specialization of a hyperelliptic divisor on a smooth hyperelliptic curve and conversely.
\end{cor}
\begin{rem}\lbl{bisep+all-seps-rem}
It is worth noting provisionally
at this point that while the above construction took as its starting point the
'original' Hodge bundle and \bn map, we could just as well
have  started with the modifications $\be_\Theta, \phi_\Theta$, where $\Theta$ is the set of all seps of $X_0$, as in
Proposition \ref{all-seps-prop}. More general results will be
given below.
\end{rem}

\newsubsection{Interpretation}\lbl{single-binode-interp-sec}
As in the case of a separating node (see \ref{mod-bn-sep-interpret}),
 the azimuthal modification of Brill-Noether can be interpreted
in terms of reviving sections of the canonical that vanish
identically to one side of the binode. To do so, we must have an
appropriate \myit{twisting divisor} on the total space $X$, as the
sides in question are defined in codimension 2. This is accomplished
by an azimuthal modification of the original family, as in
\S\ref{azimuthal-curves-sec}.\par Thus, fixing notations as above,
let $X_{ B(\vtheta), \rL}/B(\vtheta)$ be the left azimuthal
modification of $X/B$, which is a nodal family over
$B(\vtheta)=\Bl_{\del_\vtheta}B$, and is endowed with  left and
right boundary families $\lf{\tilde X}_P, \rt{\tilde{X}}_P$ over the exceptional divisor
$P=P(\vtheta)\subset B(\vtheta)$, which as we recall parametrizes middle
azimuths $\md{\zeta}$ at $\vtheta$. Here the right boundary family
 projects isomorphically:
\[\rt{\tilde{X}}_P\simeq\rt{X}_P= \rt{X_{{\del}}}\times_{\del}P\]
while the left one $\lf{\tilde{X}}_P$ is the pullback  of
$\lf{X}(\vtheta)\subset X$
 and isomorphic to the blowup of
$\lf{X_{{\del}}}\times_{\del}P$ in $\vtheta_P=\vtheta\times_{\del}
P$. Over the complement of the locus of singular azimuths, i.e. the
2 distinguished sections of $P/\del$, the modified family ${X}_{
B(\vtheta), \rL}$ coincides with $X_{B(\vtheta)}$.\par
 We consider
the  relative Hilbert scheme for this modified family, viz. $(X_{B(\vtheta), \rL})\sbr 2._{B(\vtheta)}$
(note the outer $B(\vtheta)$ subscript indicates that this is the \emph{relative}
Hilbert scheme). This
comes equipped with a map to the symmetric product
$(X_{\rL, B(\vtheta)})\spr 2._{B(\vtheta)}$, hence to $X\spr 2._{B(\vtheta)}$.
Therefore we get a correspondence diagram
\eqspl{extended-azimuthal-eq}{
\mydiamond{Y:=X\sbc 2._{B(\vtheta)}\times_0 (X_{\rL, B(\vtheta)})\sbr 2._{B(\vtheta)}}
{X\sbc 2._{B(\vtheta)}}{(X_{\rL, B(\vtheta)})\sbr 2._{B(\vtheta)}}
{X_{B(\vtheta)}\spr 2.} } where $\times_0$ denotes the unique
dominant component (over either factor) of the fibre product over
$X_{B(\vtheta)}\spr 2.$. Now up on $Y$, we have all the sheaves we
need to work as in \eqref{interp-1node}. Thus, suppressing
various pullbacks, we have a map
\[\Lambda_2(\omega(-\lx_P))\to\Lambda_2(\omega)\] which vanishes on
$\lf{D^\dag}$, the inverse image of $\lf{X}_P\sbr 2.$, and we get, as in
the case of separating nodes (see \eqref {interp-1node}), a
\emph{comparison diagram} (where we recall that $\be^{-1,0}$ may be
identified as the pullback of $\pi_*(\omega(-\lx_P))$ from $X\times
_BB(\vtheta)$, hence also as the pullback of
$\pi_*(\omega(-\lf{\tilde X}_P))$ from $X_{
B(\vtheta), \rL}/B(\vtheta)$): \eqspl{interp-binode}{
\mycd{\be^{-1,0}(\lf{D^\dag})}{\be_\vtheta}
{\Lambda_2(\omega(-\lx_P))(\lf{D^\dag})} {\Lambda_2(\omega).} } Here the
horizontal maps are isomorphisms over the interior of $\lf{D^\dag}$, i.e.
the pullback of $\lf{X}\sbr 2._B\setminus[2]_*(\vtheta)$.
Moreover, the right column, i.e. $\phi_\vtheta$, is already
defined over $X\sbc 2._B(\vtheta)$ and does not explicitly depend on a choice
of middle azimuth. As for the left column, which does involve an explicit
choice of middle azimuth,
 it follows
from \eqref{normal-bundle-azimuth} that, at least over a fibre $X_0$
corresponding to a regular (middle) azimuth $\md{\zeta}$,
$\omega(-\lf{X_0})$ coincides with
$(\omega_{\lf{X_0}}(2\lf{\vtheta})\otimes\md{\zeta})\cup\omega_{\rt{X_0}}$
(we note that, being regular, $\md{\zeta}$ is naturally isomorphic by projection to both $\psi_1$ and $\psi_2$) .
\par
An interior point $z$ of $\lf{D}^\dag$ consists of an underlying subscheme $z_0$ of $\lf{X}(\vtheta)$ plus
a middle azimuth $\md{\zeta}$ at $\vtheta$,
and the latter determines
a gluing of
$\omega_{\lf{X_0}}(2\lf{\vtheta})$
and 
$\omega_{\rt{X_0}}$
at $\theta_1$ and $ \theta_2$, up to a common scalar.
Then given a right azimuth $\rt{\zeta}$ at $\vtheta$, e.g. a hyperelliptic one, it determines a left azimuth
$\lf{\zeta}=\md{\zeta}\rt{\zeta}\inv$ at $\vtheta$
which varies with $\md{\zeta}$.
If sections of $\omega_{\rt{X_0}}$ all satisfy the azimuthal
condition $\rt{\zeta}$ (as in the hyperelliptic case),
then as $z$ varies fixing the
underlying scheme, the left column of \eqref{interp-binode}
corresponds to subsystems of $\omega_{\lf{X_0}(\vtheta)}(2\lf{\vtheta})$ determined by
a varying azimuthal condition at $\lf{\vtheta}$.
Therefore, over regular
azimuthal subschemes $z$ in the interior of $\lf{D^\dag}$, the image of
the right column, i.e.  of the modified \bn $\phi_{\vtheta}$,
coincides with that of the \bn map associated to the sepcanonical
system $|\omega_{X_0}|^\sep$ restricted on $\lf{X_0}$, i.e. sections
satisfying the residue condition (automatic in this case) and,
when $(\rt{X_0}(\vtheta), \rt{\vtheta})$ is hyperelliptic, the
azimuthal condition
$\lf{\zeta}=\md{\zeta}\rt{\zeta}\inv$ . By direct inspection, considering the
matrix $\Lambda_\vtheta$ above, the same also holds at the boundary
of $\lf{D^\dag}$, i.e. the schemes meeting $\lf{\vtheta}$, and for the
singular azimuths, i.e. those corresponding to replacing the twist
$2\lf{\vtheta}$ by $\lf{\vtheta}+\lf{\theta_i}, i=1,2$. Of course,
$\phi_{\vtheta}$ is left-right symmetric as well. Thus:
\begin{lem}\lbl{mod-bn-sepcan-1binode}
For subschemes entirely on one side of $\vtheta$, the image of
the modified Brill-Noether map
coincides with that of the Brill-Noether map
associated to the $\vtheta$-sepcanonical system of $X$.\par

\end{lem}
Now consider a fibre $X_0$ in the 'interior' of
 $\del_\vtheta$,
in the sense that $(\lf{X_0}(\vtheta), \lf{\vtheta})$
and  $(\rt{X_0}(\vtheta), \rt{\vtheta})$ are
2-inseparable (see \S 3 of \cite{canonodal});
equivalently, $X_0$ has no other seps or biseps besides
$\vtheta$. If $(\rt{X_0}, \rt{\vtheta})$ is not hyperelliptic, clearly $|\omega_{\lf{X_0}(\vtheta)}(2\lf{\vtheta})\cup\omega_{\rt{X}_0}|$
induces on the left side the complete linear system
$|\omega_{\lf{X_0}(\vtheta)}(2\lf{\vtheta})|$, necessarily very ample. Consequently, whenever $X_0$ is not hyperelliptic,
$\phi_\vtheta$ is surjective over $(X_0)\sbc 2.$.\par
On the other hand if $X_0$ is hyperelliptic, a similar
argument shows that $\phi_\vtheta$ drops rank precisely
on the hyperelliptic azimuthal schemes, i.e.
those regular azimuthal schemes of the form
$(z_0, \rt{\zeta})$ where $z_0$ is a hyperelliptic divisor on $\lf{X_0}(\vtheta)$ and $\rt{\zeta}$ is the hyperelliptic azimuth on $\rt{\vtheta}$, or the analogous schemes on the
right. Note $\phi_\vtheta$ has maximal rank at the singular azimuthal schemes as well as the mixed schemes (those on
both sides of $\vtheta$). Thus, if we let $\he\sbc 2._B\subset
X\sbc 2._B$ denote the locus of hyperelliptic azimuthal schemes, we conclude:
\begin{cor}
Over the fibers having no sep or bisep besides $\vtheta$,
the degeneracy locus of $\phi_\vtheta$
 is equal to $\he\sbc 2._B$ and
 coincides with the closure of the smooth hyperelliptic
 locus $\he^2$.
\end{cor}
\newsubsection{Case of multiple disjoint biseps}\lbl{multi-bisep-sec}
Now consider a collection of pairwise disjoint,
properly separating binodes
$\Theta=(\vtheta_1,...,\vtheta_n)$ of $X/B$, $\vtheta_i=(\theta_{i,1},\theta_{i,2})$,
 corresponding to a mutually transverse
collection of  codimension-2 boundary loci $\del_i=
\del_{\theta_{i,1}}\cap \del_{\theta_{i,2}}$. Then the corresponding loci
\[\lf{D}(\vtheta_i), \rt{D}(\theta_i)\subset X\sbr 2._B, i=1,...,n\]
are also mutually transverse, so we may blow them up
in any order, thus obtaining the azimuthal Hilbert scheme
of $X/B$
associated to $\Theta$, which we denote
by  $X\sbc 2._B(\Theta)$.
There is also an extended version, which is flat over
$B(\Theta)$, the blowup in any order of $\del_1,...,\del_n$,
 and is denoted $X\sbc 2._{B(\Theta)}$. As above, we obtain
over  $X\sbc 2._B(\vtheta)$ a mutually transverse collection
of $2n$ echelon data
  \[(\bE^{-1,0}(\vtheta_i), \lf{D^\dag}(\vtheta_i)),
(\bE^{0,-1}(\vtheta_i), \rt{D^\dag}(\vtheta_i)), i=1,...,n\]
from which we construct the associated echelon    modification $\be_\Theta$, which comes with a map ('modified \bn'):
\[\phi_\Theta:\be_\Theta\to\Lambda_2(\omega).\]
This of course factors $\phi_{\Theta'}$ for any subset
$\Theta'\subset\Theta$.
A similar construction can be made if $\Theta$ is a collection
containing separating nodes (seps) and properly separating binodes (biseps) of $X/B$,
all pairwise disjoint.
 As in \S\ref{multi-sep-sec},
it is easy to extend Lemma \ref{mod-bn-sepcan-1binode} to this situation and conclude
\begin{prop}\lbl{mod-bn-sepcan-disjoint}
Let $\Theta$ be a collection of pairwise disjoint seps and
biseps of $X/B$, then for any azimuthal subscheme of
$X/B$ contained in one separation component of $\Theta$,
the image of
the modified Brill-Noether map $\phi_\Theta$
coincides that of the Brill-Noether map
associated to the $\Theta$-sepcanonical system of $X$.
\end{prop}
A nodal curve $X_0$ is said to be of \emph{semicompact type}
if all its biseps are maximal; equivalently, if the node set
of $X_0$ is a disjoint union of seps, biseps, and absolutely nonseparating nodes.
\begin{cor}\lbl{semicompact-type-cor}
Let $X_0$ be a semicompact- type fibre of the versal family
$X/B$,  let $\Theta$ be the collection of all seps and biseps
occurring on $X_0$ and
\[\phi_\Theta: \be_\Theta\to \Lambda_2(\omega)\]
be the associated modification, defined over the azimuthal Hilbert scheme $X\sbc 2._B(\Theta)$. Then
\[\mathbb D_2(\phi_\Theta)=\he\sbc 2.\cup\bigcup\limits_{\theta\in\Theta\mathrm
{\ sep}} R_\theta.\]
\end{cor}

\Large
\part{General case}
\normalsize
Here we deal with curves of non-semicompact type, i.e. curves admitting 'separating
cycles' of nodes, herein called \emph{polyseparators}. Constructing an appropriate
modification of the Brill-Noether map to cover this case requires  a more elaborate kind of
blowup of the Hilbert scheme, related but not identical to 
the blowup induced by the natural
number-of-nodes stratification
of the boundary of the parameter space. This kind of blowup called \emph{stratified} or \emph{normal} blowup can be defined generally in the setting of a divisor with normal crossings.
Applying it to the Hilbert scheme leads to the \emph{azimuthal} Hilbert scheme. These constructions are pursued in \S \ref{polysep-sec}. Then in \S \ref{azi-bn-sec} we construct
the appropriate modification of the Brill-Noether map over
the azimuthal Hilbert scheme and establish its basic
properties. In \S \ref{excess-porteous-sec} we combine this
result with an excess-intersection version of the
Porteous formula to derive a formula for
the fundamental class of the closure of the
hyperelliptic locus (Theorem \ref{fund-class-thm}).
\S \ref{azi-int-sec} contains some of the intersection theory
needed to make explicit the formula of Theorem \ref{fund-class-thm}.
\newsection{Modifying Brill-Noether: polyseparators}
\label{polysep-sec}
\newsubsection{$S$-Stratified blowup}\lbl{stratified-blowup}
Let $T$ be a divisor with local normal crossings on a
smooth variety (or orbifold) $Y$
of dimension $n$, and let $S$ be a union of some irreducible components of the (orbifold)
singular locus $\sing(T)$, with the reduced structure, and some components of $T$ itself. Thus, 
each component of $S$ is either $(n-1)$ or $(n-2)$-dimensional.
We assume $S$ has the following 'transitivity' property:\par
(*)\quad if $t_1, t_2, t_3$ define local branches of $T$ so that
$(t_1, t_2)$ and $(t_2, t_3)$ define branches of $S$, then so does $(t_1, t_3)$.\par
 Set
$s_2(S)=S\cap\sing(T)$ and inductively,
\[ s_{i}(S)=\sing (s_{i-1}(S)), i>2\]
(all with the reduced structure; singular locus understood in the orbifold sense).
It is easy to see that $s_i(S)$ has codimension $i$ in $Y$ or is empty (e.g. if $i>n$) and that
$s_i(S)$ is contained in the locus
of points of multiplicity
$\geq i$ on $T$. In fact, $s_i(S)$ coincides with
the locus of points lying on $i$ local branches of $T$, every pair of which intersects in a branch of $S$.
Thus a general point of $s_i(S)$ lies
on $\binom{i}{2}$ local analytic components of $S$. The chain
\[s_2(S)\supset s_3(S)...\supset s_n(S)\] is called the \emph{$S$-stratification} of $Y$.

\begin{example}\lbl{polysep-stratif}
Let $T=\delta_0\subset\mgbar$, the divisor of
curves with a nonseparating node, and let $\mathcal S\subset \sing(\delta_0)$
denote the locus of curves with a separating binode.
The associated $\mathcal S$-stratification, called the \emph{polyseparator
stratification}, has $i$-th stratum corresponding to curves with
a degree-$i$ polyseparator.
\end{example}
Given $S$ as above, the \emph{stratified blowup} of $Y$ associated
to $S$, denoted \[B^\sigma_S(Y)\stackrel{b^\sigma}{\to}Y\] is the (smooth) variety obtained from $Y$ by first blowing up $s_n(S)$, then the
(smooth) proper transform of $s_{n-1}(S)$, etc.
(by definition, the blowup of the empty set is the identity
$Y\to Y$).  We denote by $E^\sigma(i)$
 or $E^\sigma_S(i)$ the proper transform of $s_i(S)$ on $B^\sigma_S(Y)$, so that
\[(b^\sigma)\inv(s_i(S))=\sum\limits_{j=n}^iE^\sigma(j).\]
Clearly, $E^\sigma(i)$ is smooth and forms a locally trivial
fibre bundle over the normalization of $s_i(S)$. The fibre, denoted $W^\sigma_{i-1}$,
is a toric variety which may be identified with the stratified blowup of the union of the coordinate hyperplanes in $\P^{i-1}$.
The structure of such bundles generally is described in \S\ref{w-bundles-sec}.
\newsubsection{Normal blowup}\lbl{normal-blowup-sec}
Here we describe an a-priori different, 'purely codimension-2' blowup construction
that ultimately leads to the stratified blowup.
\par We continue the above notation.
Locally, let $S_1, ...,S_k$ be all the local branches of $S$ through some point $p$.
From the transitivity property, it follows that there
are branches
 $T_1,..., T_r$ of  $T$ through $p$, such that
$k=\binom{r}{2}$ and the $S_i$ are the pairwise intersections
$T_j\cap T_\l, j\neq\l$; in other words, locally at $p$,
there is a normal-crossing subdivisor $T'=\bigcup\limits_{i=1}^r T_i$ such that
\[ S=\sing(T').\] Let
\[\mathcal J=\prod\limits_{i=1}^k I_{S_i}.\]
In fact, the ideal $\mathcal J$ is defined globally.
We define the \emph{$S$-normal blowup} of $Y$ as
\[B^\perp_S(Y):=\Bl_{\mathcal J}(Y)\]
Note that this coincides with the 'main'
component of the fibred product
$\prod (\Bl_{S_i}(Y)/Y)$, i.e. the unique component
dominating $Y$.
\par Our purpose is to prove:
\begin{prop}\lbl{stratified-is-normal-prop}
Notations as above, the $S$-stratified and normal blowups are equivalent
over $Y$.
\end{prop}
The Proposition  will be proven below after some discussion. Its conclusion
amounts to saying that the identity on $Y$ lifts to an isomorphism $B^\perp\simto
B^\sigma$. This assertion is local over $Y$, so replacing $T$
by the union of its branches through a given point
and using the transitivity property, we may as well assume
$S=\sing(T)$. Clearly, each irreducible component $S_i$ of $S$ pulls
back to a Cartier divisor on $B^\sigma$, so by the universal property of blowing up we get a morphism
$B^\sigma\to \Bl_{S_i}(Y)$. Putting these together
yields a morphism
$B^\sigma\to \prod (\Bl_{S_i}(Y)/Y)$, which as $B^\sigma$
 is irreducible lands in the
main component, whence a morphism
\[B^\sigma\to B^\perp .\]
To go the other way, we first study the normal blowup. We work locally analytically, so
we can write \[S=\sing(T), \ T=\bigcup T_i, T_i=(t_i)\]
 where the
 $t_i$ are part of a regular
sequence of parameters. For any index-set  $I$, let
\[T_{I}=\bigcap\limits_{i\in I} T_i\] with ideal $\mathcal J_{I}$ locally generated by $t_i, i\in I$. Then we get the local factorization
\[\mathcal J=\prod\limits_{i<j}\mathcal J_{ij}, \ \mathcal J_{ij}=(t_i, t_j).\]
Also, let $T(i)=\bigcup\limits_{|I|=i}T_I$, 
and let $T(i)^{\norm}$ be
its normalization, which is smooth and locally a disjoint
union of branches:
\[T(i)^\norm=\coprod\limits_{|I|=i}T_I.\]
Let $T_I^o$ be the interior
 of $T_I$, i.e. $T_I\setminus\bigcup\limits_{j\not\in I}T_j$.
For $i<j$, let $B_{ij}$ be the blowup
of $T_{ij}=T_i\cap T_j$ in $Y$. Then we have
\[B^\perp_{T}(Y)/Y={\prod\limits_{i<j}}'(B_{ij}/Y)\]
 (where $\prod '$ means unique dominating component of
 fibre product over $Y$). Let $b:B^\perp_{T}(Y)\to Y$ be the natural map.
For any $I$, let $E_I$ be the closure of $b\inv(T_I^o)$,
and similarly for $E(i)$ and $E(i)^{\norm}$.\par
To describe the structure of the exceptional divisors, we make the following construction. Let $L_0, ..., L_n$ be a collection of  line
bundles on a scheme $Z$, and let
\[P_{i,j}=\P(L_i\oplus L_j), \ G_{i,j}=P_{i,j}\setminus (\P(L_i)
\cup \P(L_j)), 0\leq i<j\leq n.\]
We identify $G_{i,j}$ with the $\C^*$ bundle associated to
$L_i\inv\otimes L_j$.
Consider the map
\eqspl{W[L]-def-eq}{\prod\limits_{i=0}^{n-1}G_{i,i+1}\to \prod\limits_{0\leq i<j\leq n}P_{i,j}\\
(\lambda_1,...,\lambda_n)\mapsto (\lambda_i\cdots\lambda_{j-1}:0\leq i<j\leq n),\  \lambda_0:=1
} and let $W[L_0,...,L_n]$ be the closure of its image, obviously a fibre bundle over $Z$ with fibre
a toric $n$-fold $W_n=W[\C,...,\C]$ ($n+1$ factors, where $\C$ denotes the trivial line bundle over a point).
\par It is easy to see that $W[L_0,...,L_n]$ can also be realized as the closed image of the rational map
\eqsp{
\P[L_0\oplus...\oplus L_n]&\dashrightarrow \prod\limits_{
0\leq i\leq j\leq n} P_{i,j},\\
[x_0,...,x_n]&\dashrightarrow ([x_i,x_j]:0\leq i\leq j\leq n).
} While the former description is better suited for the Lemma that follows, the latter one is more convenient for further study (see \S \ref{w-bundles-sec} below). Note that
$W[L_0,...,L_n]$ is independent  of the order of the line bundles. By abuse of
notation, we will denote $W[L_0,...,L_n]$ by $W[\bigoplus\limits_{i=0}
^n L_i]$ when the splitting of the rank-$(n+1)$ bundle is understood
up to order.
\begin{lem}\label{nbl}
\begin{enumerate}
\item  $B^\perp_{T}(Y)$ is smooth.
\item $T(i)^{\norm}$ is smooth, the
inverse image $\tilde T(i+1)$ of $T(i+1)$ on it is a divisor with
normal crossings. The map $E(i)^{\norm}\to T(i)^{\norm}$ factors
through a map
 \[E(i)^{\norm}\to B^\perp_{\tilde T(i+1)} T(i)^{\norm}\]
which is a
 locally trivial fibration with fibre $W_{i-1}$, of the form
 $W[N_{T(i)^{\norm}\to Y}]$.
 \item The reduced total transform of $T(i)$ on $B^\perp_{T}(Y)$
 is $\bigcup\limits_{i'\geq i} E(i')$, and this is
 a divisor with normal crossings.
\item The exceptional divisor of $B^\perp_T(Y)\to Y$ is
$\bigcup\limits_{i\geq 2} E(i)$.
\end{enumerate}
\end{lem}
\begin{proof} (i) Working locally, we cover the blowup with $2^r$ open sets $U$,
each specified by a choice of ordering $i\to j$ or
$t_i\to t_j$ on each integer
pair
$\{i,j\}\subset [1,r],$ indicating that $t_j/t_i$ is regular there.
Then $\to$ generates a total order $\prec$, possibly with degeneracies or equivalences. Let $i_0$ be a minimum, unique up to equivalence.
Then for any $i\prec j$, $j$ is reachable from $i$ by a chain of immediate-
successor pairs and equivalent pairs. For an equivalent pair $a\sim b$, clearly $t_a/t_b$ is a unit. For any immediate successor $a\prec b$
we may up to equivalence assume $a\to b$.
Thus, we may choose a set $P$ consisting of a
maximal collection $P_1$ of immediate successors plus a suitable collection $P_2$ of equivalent pairs, $k-1$ from each equivalence
class of cardinality $k$, such that for any  $i\prec j$
(and this includes $i_0\prec j, \forall j=1,...,r$), $j$
is reachable from $i$ by a succession
of pairs in $P$. Clearly $P$ has $n-1$ elements.
Then $(t_{i_0}, t_i/t_j: (i,j)\in P_1)$ is a regular system of parameters on the blowup. These together
with the units corresponding to $P_2$ and
complementary coordinates to the $t_i$ yield a coordinate system on $U$.
\par (ii) The proper transform of $T_I$ in the blowup
is locally the zero locus of $t_i/t_j, \forall i\in I, j\not\in I$. This admits a forgetful map, forgetting the ratios $t_i/t_{i'}, i,i'\in I$, which clearly lands in the space of the $t_j$ and $t_j/t_{j'}, j,j'\not\in I$, i.e. the normalized blowup $B^\perp_{(T_j|_{T_I}:j\not\in I)} T_I$, and has all fibres $W_{|I|-1}$.\par
(iii) and (iv): Straightforward, given the analysis above.
Note that the local branches of $T(i)$ define a splitting of the normal bundle with summands defined up
to order, as required in the definition above.
Note that in the schematic total transform of $T(i)$, the $E(i')$ with $i'>i$ will appear with higher multiplicities due to $T(i')$ lying on
multiple branches of $T(i)$.
\end{proof}
\begin{rem}\lbl{min-coordinate} The mimimal index $i_0$ and corresponding minimal coordinate $t_{i_0}$ are of importance in their own right and will be used in the sequel.
\end{rem}
\begin{proof}[Proof of Proposition \ref{stratified-is-normal-prop}]
It remains to show there is a morphism $B^\perp\to B^\sigma$. This
follows easily if we show that each $i$-dimensional stratum of the
$S$-stratification pulls back to a Cartier divisor on $B^\perp$. We
use the notations developed in the proof of the Lemma. Then, on a
suitable open set and after rearranging so that the ordering of the
coordinates is the standard one $1\prec 2...\prec r$, the
$i$-dimensional stratum, defined in $Y$ by $t_{i+1}, ...,t_r$, pulls
back
 to the Cartier divisor in $B^\perp$ defined by $t_{i+1}$,
 because $[t_{i+2}/t_{i+1}],...,[t_r/t_{i+1}]$ are regular.
 This concludes the proof.
\end{proof}

\newsubsection{Azimuthal blowup}\lbl{azi-blowup-sec}
We extend the notion of left azimuthal modification
of \S \ref{azimuthal-curves-sec} to the case of multiple,
not necessarily disjoint, binodes.
This again will be needed primarily for the sake of comparing
a modified \bn map with an ordinary one associated to
a sepcanonical system.
The construction will be local over the
base, so fix a fibre $X_0$ of $X/B$,  and let $\Theta=\Theta_1\cup\Theta_2$ be a set of seps
$\theta_i\in\Theta_1$ and biseps $\vtheta_i\in\Theta_2$
occurring on
$X_0$, and assume $\Theta_2$ has the transitivity property as in \S \ref{stratified-blowup}. Let $Y$ be a component of the separation
$X_0^\Theta$, and assume that all the elements of $\Theta$ 
are oriented
so as to have $Y$ on their left (we call this a \emph{$Y$-compatible} orientation).
Let $B^\sigma(\Theta)$ denote the normal (=stratified) blowup of $B$ corresponding
to $\Theta_2$. This is virtually smooth, i.e.
smooth if $B$ is smooth and $X/B$ is versal, and is
independent of orientations.
Also $B^\sigma(\Theta)$ maps to the blowup $B(\vtheta), \forall \vtheta\in\Theta_2$.
When $\Theta$ coincides with the set of all seps and biseps
of $X_0$ we will omit $\Theta$. Locally, branches
  of the exceptional locus of $B^\sigma(\Theta)/B$ correspond to polyseparators
\mbox{$\Pi\subset |\Theta_2|:=\bigcup\limits_{\vtheta\in\Theta_2}\vtheta$}
and we denote the branch corresponding to $\Pi$ by $\Xi(\Pi)$.
\par
Note that a point $0'\in B^\sigma(\Theta)$ over $0\in B$ corresponds to a collection of middle azimuths $\md{\zeta}(\vtheta),\forall
\vtheta\in\Theta_2$, subject to
relations coming from maximal polyseparators. The point $0'$ is said to be regular if all the
$\md{\zeta}(\vtheta)$ are regular. This holds if and only if
$0'$ sits only on branches
$\Xi(\Pi)$ where $\Pi$ is a \emph{maximal} polyseparator
on $X_0$.

Let $X_{B^\sigma(\Theta)}$ be the base changed family.
This contains Cartier divisors $X_{\Xi(\Pi)}$ for polyseparators $\Pi$. The latter splits as a union of Weil
divisors, where $\Pi=(\theta_1,...,\theta_n)$ (cyclical
arrangement):
\[X_{\Xi(\Pi)}=\bigcup\limits_{i=1}^n \lf{X}(\theta_i, \theta_{i+1})_{\Xi(\Pi)}.\]
To describe this in local coordinates, let $t_i=x_iy_i$
be a local equation for $\del_{\theta_i}, i=1,...,n$, and
recall that at each point
of $\Xi(\Pi)$, there is a 'minimal' index $j$ so that
$B^\sigma(\Theta)$ admits local parameters\[t_1/t_j,...,t_{j-1}/t_j,t_j,
t_{j+1}/t_j, ...,t_n/t_j\] and $t_j$ is an equation for
$\Xi(\Pi)$.
At $\theta_i$,  $\lf{X}(\theta_i, \theta_{i+1})_{\Xi(\Pi)}$
is either Cartier, if $i=j$, or has local equations $y_i, t_j$
if $ j\neq i$.

The $\lf{X}(\theta_i, \theta_{i+1})_{\Xi(\Pi)}$ are mutually transverse for fixed $\Pi$ (in fact, three
distinct ones have empty intersection). Clearly they are also
mutually transverse for distinct $\Pi$ (as the divisors $\Xi(\Pi)$ are already transverse).
Then consider their joint blowup, i.e.  the blowup of the ideal sheaf $\prod\limits_i
\I_{\lf{X}(\theta_i, \theta_{i+1})_{\Xi(\Pi)}}$, which coincides with the unique dominant component of the fibre product
$\prod\limits_i \Bl_{\lf{X}(\theta_i, \theta_{i+1})_{\Xi(\Pi)}}X_{B^\sigma(\Theta)}$. We denote this common blowup by
$_YX_{B^\sigma(\Theta)}$ and call it the \emph{azimuthal modification of
$X/B$ corresponding to $Y$} .                                           It depends on $Y$ only for orientation, and is defined over a neighborhood of
the fibre  $X_0$ containing $Y$ and defines
a family of curves over a neighborhood of $0\in B$.
As in \S \ref{azimuthal-curves-sec}, it is virtually smooth.
It comes equipped with a birational morphism over $B(\Theta)$
\[_YX_{B^\sigma(\Theta)}\to X_{B^\sigma(\Theta)}.\]
Now suppose that $0'\in B^\sigma(\Theta)$ is a regular
point as above, hence lies only  on divisors $\Xi(\Pi)$
where $\Pi$ is a maximal polyseparator on $X_0$. Given any bisep $\vtheta$ (as always, oriented with $Y$ to its left)
on $X_0$, it is contained as
an adjacent pair in a unique
maximal polyseparator $\Pi(\vtheta)$, and we let
\[_YX(\vtheta)=\lf{X}(\vtheta)_{\Xi(\Pi(\vtheta))}\]
be the corresponding (Cartier) divisor. Additionally,
we have for each sep $\theta$ on $X_0$, a divisor
$_YX(\theta)$ on $_YX_{B^\sigma(\Theta)}$ which is the
pullback of $\lf{X}(\theta)$.
Then on $_YX_{B^\sigma(\Theta)}$ we may consider
the twisted canonical bundle
\eqspl{rel-sepcanonical}{_Y\omega^\sep=_Y\omega^\sep(\md{\zeta}_\bullet)
:=\omega_{_YX_{B^\sigma(\Theta)}/B^\sigma(\Theta)}
(-2\sum\limits_{\theta_i\in\Theta_1} {_YX}(\theta_i)-\sum\limits_{\vtheta_i\in\Theta_2} {_YX}(\vtheta_i))}
where $\md{\zeta}_\bullet$ is the collection of
(regular) middle azimuths corresponding to $0'$
(and will be suppressed when understood).
We call this the \emph{relative sepcanonical system adapted to} $Y$.
 Proposition 6.10 of \cite{canonodal} can now be generalized
 as follows
 \begin{prop}\lbl{stratified-sepcanonical-prop}For the above line bundle $_Y\omega^\sep$, the image
of the natural restriction map
\[\pi_*(_Y\omega^\sep)\to H^0(Y, _Y\omega^\sep\otimes\O_Y)\]
coincides with the restriction of the sepcanonical system $|\omega_X|^\sep$ on $Y$.\end{prop}
\begin{proof}
Begin with some elementary remarks. First, $0'$ being regular
implies that any bisep that has at least 1 point on $Y$
actually has both points on $Y$ .
Moreover the various biseps occurring on $Y$ are disjoint,
hence are contained in disjoint maximal polyseparators.
Next, for any nonmaximal bisep $\vtheta$ ocurring on $Y$
(and having $Y$ to its left), $\vtheta$ is automatically
non-right-hyperelliptic, hence imposes no azimuthal condition
on $\lf{\vtheta}$.\par
Now to check necessity of the conditions defining
$|\omega|^\sep$, we can work as in the proof of Proposition 6.10 of \cite{canonodal}, smoothing out all but one sep or bisep on $Y$, in which case the condition at the remaining one
becomes obvious. This proves necessity. The proof of sufficiency is identical to the corresponding argument
in the proof of Proposition 6.10 of \cite{canonodal}.
\end{proof}
\newsubsection{Azimuthal Hilb}
\lbl{azimuthal-hilb}
Here we give the construction of an
azimuthal modification of the Hilbert scheme by
an appropriate stratified blowup. This is the construction we will really use.
It is related
to that of \S \ref{azi-blowup-sec}, but differs from it in being defined already over  $B$
itself. We shall work here with \emph{oriented} biseps and their left sides only,
but since each plain bisep will occurs twice with opposite orientations, the eventual
construction will be symmetric.
\subsubsection{The definition}
We fix the versal family $X/B$ as before.
Here we put together the per-binode
modifications of the Hilbert scheme described in \S
\ref{mod-hilb-sec}. Given an \emph{oriented} bisep
$\vtheta$, usually defined only locally over $B$,  let
\[ S_\vtheta= 
\lf{X}(\vtheta)\sbr 2.\]
which is a subset of $\sing(T)$, where $T$ is the pullback
on $X\sbr 2._B$ of the divisor $\delta_0(B)$ of generically
irreducible nodal curves. For a good family $X/B$ (cf. \S \ref{hilb-review}),
$T$ has normal crossings. Then let
\[ S=\bigcup\limits_\vtheta S_\vtheta ,\]
the union being over all oriented biseps $\vtheta$
(each unoriented bisep will appear twice, so both its sides
will appear). Then
$S$ is globally defined over $B$, and we
 define the \emph{azimuthal Hilbert scheme} as the stratified (or equivalently,  normal, see Proposition \ref{stratified-is-normal-prop}) blowup:
\eqspl{}{
X\sbc 2._B=B^\sigma_{S}(X\sbr 2._B)\stackrel{b}{\to} X\sbr 2._B.
}
Elements of $X\sbc 2._B$ will be referred to as 'azimuthal schemes.
By construction, $X\sbc 2._B$ dominates every $X\sbc 2._B(\vtheta)$ over the open set of $B$ where the latter
is defined, and also dominates the analogous blowup corresponding to any collection of binodes where
the collection is defined. Therefore an  azimuthal scheme, whose
 underlying 'plain' scheme is disjoint from all biseps, may be viewed as a scheme
with a collection of 
middle azimuths $\md{\zeta}(\vtheta)$, one for each
oriented bisep $\vtheta$ having $z$ on its left.
 These azimuths are not independent but
are subject to relations  of compatibility 'around maximal polyseparators'
(see the proof of Proposition \ref{stratified-is-normal-prop}).
\par
\subsubsection{Exceptional divisors} We aim to describe the exceptional locus of this blowup, first
locally over $B$. By the general description above, the components
of the exceptional divisor correspond to components of the various
strata (of different dimensions) of the stratification corresponding to $S$. Thus, consider a singular fibre $X_0$ and a maximal oriented polyseparator $\Theta_{\max}$ on $X_0$
and adjacent bisep in it \[\vtheta=(\theta_1, \theta_2)\subset\Theta_{\max}.\] Note that
$X_0$ and $\vtheta$ determine $\Theta_{\max}$. Then set
\[\quad Y_0=\lf{X_0}(\vtheta).\]

There is a codimension-2 boundary
locus $\del_\vtheta\subset B$ in a
neighborhood of $0$ over which $\vtheta$ is defined,
and $Y_0$ extends to a family
$Y=Y_\vtheta/\del_\vtheta=\lf{X_\vtheta}(\vtheta)$, in whose generic
member the opposite side $\rt{X_{\vtheta}}(\vtheta)$
is irreducible, as is $Y_\vtheta$. Note $Y_0$ extends as well as to smaller subfamilies
\[Y_{\vtheta,\Theta}=Y_\vtheta\times_{\del_\vtheta}\del_\Theta,\]
one for each $\Theta$ which is a
polyseparator on $X_0$ containing $\vtheta$ and contained in
$\Theta_{\max}$. The family $Y_{\vtheta,\Theta}$ does not depend on
$\Theta_{\max}$.
 It is locally defined on $X$ by the equations defining
$\lf{X}(\theta_1)$ and $\lf{X}(\theta_2)$ plus base equations
defining $\del_\theta, \forall \theta\in\Theta\setminus\vtheta$.
For a general fibre over $\del_\Theta$, $\Theta$ is a maximal
polyseparator and $\vtheta$ is adjacent on it.
 Each such family $Y_{\vtheta,\Theta}$ gives rise to
a divisor $\Xi_\vtheta(\Theta)$, equal to the proper transform of
$Y_{\vtheta, \Theta}\sbr 2.$ on $X\sbc 2._B$. By Lemma \ref{nbl} and Lemma \ref{azim-bundle-lem}, the restricted
blowdown map $\Xi_\vtheta(\Theta)\to Y_{\vtheta, \Theta}\sbr 2.$ is
a $W_{n-1}$-bundle, of the form \[\Xi_\vtheta(\Theta)=W_{Y_{\vtheta, \Theta}\sbr 2.}[\theta_1^\dag, \theta_2^\dag,
-\del_{\theta_3},...,-\del_{\theta_n}],\] where $\Theta=(\theta_1,...,\theta_n)$ is a cyclic arrangement and
\[\theta_i^\dag=\O(-\del_{\theta_i}+[2]_*( \lf{\theta_i})),\ 
i=1,2.\] 
In particular, there is a projection
\eqspl{Xi-projection}{p:\Xi_\vtheta(\Theta)\to \P_{Y_{\vtheta, \Theta}\sbr 2.}(\theta_1^\dag\oplus \theta_2^\dag\oplus
\O(-\del_{\theta_3})\oplus...\oplus\O(-\del_{\theta_n})).}
Let
\eqspl{Xi-bundle}{L_\vtheta(\Theta)=p^*(\O(1)).}
Also, if $\Theta'$ is any polyseparator between $\vtheta$ and $\Theta$, i.e.
 with $\vtheta\subset\Theta'\subset \Theta$, we have a map
 \[p_{\Theta'}:\Xi_\vtheta(\Theta)\overset{q_{\Theta'}}{\to} W_{Y_{\vtheta, \Theta}\sbr 2.}[\theta_1^\dag, \theta_2^\dag,
- \del_{\theta_i}:\theta_i\in\Theta']\to \P(\theta_1^\dag\oplus\theta_2^\dag\oplus
  -\del_{\theta_i}:\theta_i\in\Theta' )\]
 whence a line bundle
 \eqspl{Xi-bundle-prime}{
 L_{\vtheta}(\Theta, \Theta')=p_{\Theta'}^*(\O(1)).
 }
 Thus,
 \[L_{\vtheta}(\Theta, \Theta')=q_{\Theta'}^*(L_\vtheta(\Theta').\]

Also, we have locally
near $0$, \eqspl{Xi-vtheta}{\Xi_\vtheta:=b\inv(Y\sbr
2._{\del_\vtheta})=\bigcup\limits_{\vtheta\subset\Theta\subset\Theta_{\max}}
\Xi_\vtheta(\Theta).}
Because $\Xi_\vtheta$ is the pullback of $\lf{D}^\dag(\vtheta)$ from $X\sbc 2._B(\vtheta)$, we have by \eqref{D-vtheta-eq}
and \eqref{conormal-D-vtheta-eq} that
\eqspl{Xi-vtheta-conormal}{
\O_{\Xi_\vtheta}(-\Xi_\vtheta)=\O_{\O(\theta_1^\dag)\oplus\O(\theta_2^\dag)}(1).
}
Then consider, for any $n\geq 2$: \[\Xi(n)=\sum\limits_{\vtheta}\sum\limits
_{\substack{\vtheta\subset\Theta\\ |\Theta|=n}}\Xi_\vtheta(\Theta).\]
This is independent of any branch choices and extends
to a global divisor on $X\sbc 2._B$.\par
In terms of local equations, if $X/B$ is given
locally near $\theta_i$ by $x_iy_i=t_i$,
 where $y_1, y_2$ are local coordinates on
 $Y_0$ near $\rt{\theta_1}, \rt{\theta_2}$
 respectively, then  $Y_{\vtheta,\Theta}$ is defined
 near $(\theta_1, \theta_2)$ by $(x_1, x_2, t_i,i\in\Theta\setminus\vtheta)$. Therefore the
divisor $\sum\limits_{\Theta_1\supset\Theta}\Xi_\vtheta(\Theta_1)$ is locally defined
in various covering opens upstairs
in $X\sbc 2._B$ by the minimum of these coordinates,
in the sense of Remark \ref{min-coordinate} (the minimum varies from
one open to the other). This minimum also corresponds to a local generator of $L_\vtheta(\Theta)$, and
consequently the conormal bundle
\eqspl{conormal-Xi}{\O_{\Xi_\vtheta(\Theta)}(-\Xi_\vtheta(\Theta))=L_\vtheta(\Theta)\otimes\O_{\Xi_\vtheta(\Theta)}
(\sum\limits_{\substack{\Theta_1\\ {\Theta_1\supsetneqq\Theta}}}\Xi_\vtheta(\Theta_1)).}
\par
We summarize the above discussion as follows.
\begin{thm}\label{azi-hilb-thm}\begin{enumerate}\item
The map $X\sbc 2._B\to X\sbr 2._B$ is birational, with
exceptional divisor
$\Xi=\sum\limits_{n\geq 2}\Xi(n)$ which
decomposes locally as
\[\Xi=\sum\limits_{\Theta\ \mathrm{polyseparator\ }}\sum\limits_{\vtheta\subset\Theta\ \mathrm{adjacent}}
\Xi_\vtheta(\Theta).\]
\item $X\sbc 2._B$ is 'virtually smooth', i.e. smooth when $X/B$ is a good family.
    \item The boundary of $X\sbc 2._B/B$ is a divisor with local normal crossings.
\item Each $\Xi_\vtheta(\Theta), \Theta=(\theta_1,...,\theta_n)$, is a $W_{n-1}$-bundle fibration
induced by a birational base-change from a bundle of the form 
\[W[\theta_1^\dag,\theta_2^\dag,
-\delta_3,...,-\delta_n ]\to Y\sbr 2._{B_\Theta}\]

where $Y$ is locally a 2-inseparable component subfamily of $X_{B_\Theta}$
which is one side of the separating binode $\vtheta$, $\vtheta\subset
\Theta$, 
$-\delta_i$ is the ideal of the branch of the boundary 
corresponding to $\theta_i$ and $\theta_i^\dag=-\delta_i+\theta_i$.
\begin{flushleft}

\end{flushleft}\item The self-intersection of $\Xi_\vtheta(\Theta)$ is
 \eqspl{Xi-normal}{-c_1(L_\vtheta(\Theta))- (\sum\limits_{\substack{\Theta_1\supsetneqq\Theta\\ \vtheta\ \mathrm{adjacent\  on\ }  \Theta_1}}\Xi_\vtheta(\Theta_1)).\Xi_\vtheta(\Theta)
 }
 
 (see \eqref{Xi-bundle}).

\end{enumerate}
\end{thm}
\begin{rem}\label{azi-hilb-rem}
To be precise, $\Xi_\vtheta(\Theta)$ is a fibration of the form $W[\vtheta; \Theta]:=W[\theta_1^\dag, \theta_2^\dag, -\delta_3,
...,-\delta_n]$
over the 'induced azimuthal blowup' of $Y\sbr 2._{B_\Theta}$ which is the normal blowup corresponding to
$\Theta_{\max}\setminus \Theta$. 
The exceptional divisor in this blowup is the sum of the divisors
$\Xi_\vtheta(\Theta_1).\Xi_\vtheta(\Theta)$ appearing in \ref{Xi-normal}. The divisor
$\Xi_\vtheta(\Theta_1).\Xi_\vtheta(\Theta)$ (on $\Xi_\vtheta(\Theta)$) is in itself a birational pullback of
 $W$-bundle of the form $W[\vtheta;\Theta]$ over a birational pullback of $Y\sbr 2._{B_{\Theta_1}}$.
 Note that two such divisors $\Xi_\vtheta(\Theta_1).\Xi_\vtheta(\Theta), \Xi_\vtheta(\Theta_2).\Xi_\vtheta(\Theta)$
 are disjoint unless $\Theta_1\subset\Theta_2$ or vice versa, because they are both proper transforms
 in a blowup of a smaller stratum corresponding to $\Theta_1\cup\Theta_2 $.
If $\Theta_1\subsetneqq\Theta_2$ then the intersection $\Xi_\vtheta(\Theta_1).\Xi_\vtheta(\Theta_2).\Xi_\vtheta(\Theta)$
is a fibre product $W[\vtheta, \Theta_1]\times W[\Theta_2\setminus\Theta_1]$ over a birational pullback of 
$Y\sbr 2._{B_{\Theta_2}}$.\par 
More generally an intersection $\bigcap\limits_{i=1}^m(\Xi_\vtheta(\Theta_i).\Xi_\vtheta(\Theta))$
is empty unless the polyseparator collection $\{\Theta_1, ...,\Theta_m\}$ is totally ordered under inclusion, and if
we label them so that $\Theta_1\subsetneqq...\subsetneqq\Theta_m$ then the intersection is a fibre product
\[W[\vtheta;\Theta_1]\times W[\Theta_2\setminus\Theta_1]\times...W[\Theta_m\setminus\Theta_{m-1}]\]
over a birational pullback
of $Y\sbr 2._{B_{\Theta\cup\Theta_1...\cup\Theta_m}}$.


\end{rem}
\begin{example}
For the universal family of curves over $\mgbar$, the choice of
$\vtheta$ and $\Theta$ corresponds to writing $g=g_1+...+g_n+1$ and looking at the boundary locus corresponding to a
cyclical arrangement of 2-pointed curves glued
cyclically along the marked points, so that
$(\lf{X}(\theta_i,\theta_{i+1}), \lf{\theta}_i, \lf{\theta}_{i+1})$ is the $i$-th marked curve.
In particular, $Y(\vtheta), \vtheta=(\theta_1, \theta_2)$
is then
the universal 2-pointed curve over
$\M_{g_1,2}$, pulled back over $\M_{g_1,2}\times
\M_{g_2, 2}\times...\times\M_{g_n,2}$ and $Bl_\vtheta(Y(\vtheta)\sbc 2._{B(\Theta)})$ is the appropriate
modification of its Hilbert scheme. We may identify $\Xi_\vtheta(\Theta)$ with the $W_{n-1}$-bundle
$W[\psi_1\otimes\O([2]_*\theta_1)\oplus\psi_2\otimes\O([2]_*\theta_2)\oplus\bigoplus\limits_{i=3}^n\psi_i]$ where $\psi_i=\lf{\psi}(\theta_i)\otimes
\rt{\psi}(\theta_i)$.
See Section \ref{w-bundles-sec} for more information on this.
\end{example}
\begin{example}
Consider the extremal case of  an $n$-dimensional family with
an isolated 'cyclic' fibre of the form $X_0=Y_1\cup...\cup Y_n$.
We get a polyseparator $\Theta_{\max}=(\theta_1,... ,\theta_n)$
and adjacent biseps $\vtheta_i=(\theta_i,\theta_{i+1})$
and the line bundles $ \del_{\theta_i}$ are all trivial. Hence
$\Xi_{\vtheta_i}(\Theta)$ is a $W_{n-1}$-bundle of the form
$W[L_i, L_{i+1}, \O,...,\O]$ where $L_i=[2]_*(\O(-\lf{\theta_i}))$.
For example, $W_2$ is the blowup of $\P^2$ in 3 points and the above
$W_2$-bundle is the blowup of $\P(L_i\oplus L_{i+1}\oplus\O)$ in
the 3 distinguished sections. Because $\Theta$ is maximal, the conormal bundle of $\Xi_{\vtheta_i}(\Theta)$
is just $L_{\vtheta_i}(\Theta)$.

\end{example}
\newsection{Azimuthal Brill Noether and sepcanonical system}
\label{azi-bn-sec}
In this section we will
state and prove our definitive result on
the modified \bn map and its relation to the sepcanonical system.\newsubsection{Azimuthal Hodge bundle}
To this end we will construct an echelon modification of the Hodge bundle
over the full azimuthal Hilbert scheme $X\sbc 2._B$.
First working locally near a singular fibre $X_0$,
assumed oriented, consider
the following collection of echelon data on $X\sbr 2._B$:
\begin{itemize}
\item $\chi(\theta)$, for all separating (relative) nodes $\theta$
that meet (i.e. occur on) $X_0$ (see \eqref{echelon-data-sep-node});
\item $\chi(\vtheta):=(\be^{0,-1}(\vtheta), \Xi_\vtheta)$,
where $\be^{0,-1}$ is as in \eqref{echelon-data-sep-binode}
and $\Xi_\vtheta$ is as in \eqref{Xi-vtheta} for all
oriented biseps $\vtheta$ meeting $X_0$ (each oriented bisep will occur twice in this list, with its two orientations).
\end{itemize}
By Theorem \ref{azi-hilb-thm},  $\sum\limits_\vtheta \Xi_\vtheta$ is a divisor with normal crossings. Therefore
these echelon data  are mutually transverse.
Moreover, the entire collection- though not its individual
members - is canonically and globally defined.
Therefore this collection is a collection of \emph{polyechelon data} in the sense of \cite{echelon}, \S 3.
Therefore, there is globally defined over
the azimuthal Hilbert scheme $X\sbc 2._B$
an associated modification which we call the \emph{azimuthal
Hodge bundle} $\ae$. It comes together with a map called the
\emph{azimuthal
Brill-Noether map} associated to $X/B$ :
\eqspl{azimuthal-bn}
{\aphi:\leftidx{_\a}\bE\to\Lambda_2(\omega).
}
\newsubsection{Comparison}
The first order of business with $\aphi$ is to derive for it
a comparison diagram analogous to \eqref{interp-1node} and \eqref{interp-binode}. This diagram is local, and
depends on a choice of fibre of $X/B$ as well as a
2-component of it.\par
Thus, let $X_0$ be a fibre and $Y$ a 2-component of $X_0$, i.e. a connected component of the blowup of $X_0$ in
the collection $\Theta$ of all seps
and biseps (maximal or not, on or off $Y$). We assume
 $X_0$ is oriented so that $Y$ is to the left of each sep $\theta$ and bisep $\vtheta$.
Set
\[X\sbc 2.(Y)=(X\sbc 2._B\times_B\prod\limits_{\mathrm {bisep\ }\vtheta} (X_{L})\sbr 2._{B( \vtheta)})'\]
where as usual ' refers to the unique component
 of the fibre product over $B$ which dominates $B$.
 This is an analogue of the space denoted $Y$ in
 \eqref{extended-azimuthal-eq}.
 Then $X\sbc 2.(Y)$ admits a natural map to $B^\perp$,
 the normal blowup of $B$ associated to the family of $B(\vtheta)$; as well as to
the various factors, via which we may pull back various objects defined on these factors.
Set
\eqspl{}{ \lf{D}^\dag_1(Y)=\sum\limits_{
 \mathrm{seps\ } \theta}\lf{D}^\dag(\theta),\ 
 \lf{D}^\dag_2(Y)=\sum\limits_{
 \mathrm{\ biseps\ } \vtheta}\lf{D}^\dag(\vtheta),\\
 \lf{X}_1(Y)=
\sum\limits_{\mathrm{seps\ } \theta}\lf{X}(\theta), \ 
 \lf{X}_2(Y)=\sum\limits_{\mathrm{biseps\ }\vtheta}\lf{X}_P(\vtheta).
}
We let $Y\sbc 2.\subset X\sbc 2._B$
 denote the inverse image of $Y\sbr 2.\subset X\sbr 2._B$.
Thus, an interior element $z$ of $Y\sbc 2.$ consists
of an interior subscheme of $Y$ together with
a collection of middle azimuths at all biseps on $X_0$
(on or off  $Y$).
In particular, $z$ induces on $X_0$ a structure of azimuthal curve.
We have, for each sep $\theta$,
\[Y\sbc 2.\subset \lf{D}^\dag(\theta).\]
Note that the map $X\sbc 2.(Y)\to X\sbc 2._B$ is an isomorphism over the locus of 'interior' schemes, i.e. those disjoint from all biseps.
Now set
\eqspl{}{ \be^Y=(\pi_*(\omega(-2\lf{X}_1(Y)-\lf{X}_2(Y))))_{X\sbc 2.(Y)}.}
Then we get a comparison diagram analogous to
\eqref{interp-binode}:
\eqspl{comparison-general-eq}{
\mycd{\be^{Y}(2\lf{D}^\dag_1(Y)+\lf{D}^\dag_2(Y))}{\ae}
{\Lambda_2(\omega(-2\lf{X}_1(Y)-\lf{X}_2(Y)))(2\lf{D}^\dag_1(Y)+
\lf{D}^\dag_2(Y))}
{\Lambda_2(\omega).}  }
Again the left column is up to an isomorphism (namely
a twist) the \bn map associated to the bundle
$\omega(-2\lf{X}_1(Y)-\lf{X}_2(Y))$.
We will be using this diagram mainly over $Y\sbc 2.$.
A few remarks are in order.
\begin{enumerate}
\item If $\theta$ is a *-sep (i.e. sep or bisep) disjoint from $Y$, with
associated divisor $\lf{D}^\dag(\theta)$, which sits
over a divisor $\del_\theta\subset B^\perp$, then
\[\lf{D}^\dag(\theta)|_{Y\sbc 2.}=\pi^*(\del_\theta)|_{Y\sbc 2.}\]
and similarly for $\lf{X}_P(\theta)$ etc.
It follows that the left column of \eqref{comparison-general-eq} is equivalent to a map
where $\lf{D}_i^\dag(Y), \lf{X}_i(Y)$ are replaced by
the sum over those *-seps that meet $Y$.
\item The bottom arrow of \eqref{comparison-general-eq}
is an isomorphism over the interior of $Y\sbc 2.$.
\item Given a *-sep $\theta$ on $Y$, $\theta$ is right
hyperelliptic on $X_0$ if and only if
$(\rt{X}(\theta), \rt{\theta})$ is hyperelliptic as a pair; this has the consequence,
nontrivial only if $\theta$ is a bisep,
that $\rt{\theta}$ is contained in
a single 2-component of $\rt{X}(\theta)$, so
that any 2-component of $\rt{X}(\theta)$ is also
a 2-component of $X_0$.
\item Given an interior scheme $z$ on $Y$,
it induces
azimuthal data $\md{\zeta_\vtheta}$ on all the
 biseps $\vtheta\in \Theta$ and hence endows
$Y$ with the restriction of the
sepcanonical system $|\omega|_\Theta^\sep$ from $X_0$, and the image of the associated \bn map coincides with that
of the azimuthal \bn map $\aphi$ over $z$.
\end{enumerate}
\newsubsection{Main results}
We recall that a 2-inseparable stable curve is hyperelliptic if it is a 2-1 cover of $\P^1$
and that such curves are classified, e.g. in \cite{canonodal}, \S 2. A stable hyperelliptic curve
is a tree-like arrangement of 2-inseparable hyperelliptics, or equivalently, an admissible 2-1
cover of a rational tree. A semistable curve is hyperelliptic provided its
stable contraction is. We now proceed to state our main results.
First
some notation. Denote by $\he\sbc 2._B$ the closure of the
inverse image in $X\sbc 2._B$ of the hyperelliptic pair
locus $\he^2_{B^o}$ over the set of smooth curves in $B$.
Denote by $\rR^\sep(X/B)$ the inverse image in $X\sbc 2._B$ of
the locus of schemes supported on separating nodes
of $X/B$. It has codimension 2 in $X\sbc 2._B$.
The following result extends Proposition \ref{mod-bn-sepcan-disjoint} and Corollary \ref{semicompact-type-cor} from the semicompact to the
general case.
\begin{thm}
\lbl{main-thm}(i)\quad
We have a  schematic equality
\eqspl{degeneracy-aphi}{\bD_2(\aphi)=\he\sbc 2._B \cup \rR^\sep(X/B).}\par
Moreover the fibre of $\he\sbc 2._B$ corresponding to a boundary fibre $X_0$ of $X/B$ is empty if $X_0$ is not
hyperelliptic; if $X_0$ is hyperelliptic, the fibre consists
of the locus of hyperelliptic schemes contained in some
2-component, each endowed with the collection of hyperelliptic
middle azimuths on $X_0$.\par
(ii)\quad For each fibre $X_0$ and each azimuthal subscheme $z$
contained in a 2-component of $X_0$ but not supported at a node, the
image of $\aphi(z)$ coincides with that of the Brill-Noether map
associated to the sepcanonical system on $X_0$
compatible with $z$.
\end{thm}
\begin{proof}
We begin with the observation that both assertions are local on
$X\sbc 2._B$; the second can be checked fibre by fibre, though the
second does involve scheme structure over $B$. We follow the broad
outline of the proof of Theorem 7.5 of \cite{canonodal}. We will prove both
assertions  using simultaneous induction on the number of
components of a fibre $X_0$.  \par
Next, note that if the fibre $X_0$ is hyperelliptic, it is of
semicompact type and then our assertions already follow from
the results of Part 1, specifically Proposition \ref{mod-bn-sepcan-disjoint} and Corollary \ref{semicompact-type-cor}. Therefore, we may henceforth
assume that $X_0$ is non-hyperelliptic.\par
 We will focus next on proving that
if $X_0$ is non-hyperelliptic as 'plain' curve, ignoring 
any azimuthal data, then, over a neighborhood of
the corresponding point $0\in B$, the degeneracy scheme
$\bD_2(\aphi)$ is contained in
$\rR^\sep(X/B)$. The opposite containment
is obvious by generalizing to a nearby fibre of
compact type with the same set of seps as $X_0$ (and
no other singularities).
 By the results of \S 1, we may assume $X_0$ is 2-separable,
 i.e. contains a sep or bisep.
Suppose we are given an azimuthal scheme $z$ of $X_0$.
 We will show that the azimuthal \bn image $\im(\aphi(z))$
 is 2-dimensional. To this end we will need to distinguish
 a hierarchy of cases and subcases. We suppose next that there exists a sep or bisep $\theta$, so that $z$ is
left of $\theta$ and, for now, also that $z$
 (more precisely, the underlying scheme) is disjoint from $\theta$.
Let $ \lf{D^\dag}=\lf{X}\sbc
2._{\del_\theta}$ be the associated divisor, containing all
azimuthal schemes left of $\theta$.
Let $\be'\subset\aE$ be
the subsheaf corresponding to all the echelon modifications except
at $\theta$. Let ${\be'}^{\ .,0}$ denote the bundle obtained by
applying similar modifications on $\be^{.,0}$. Then we get a
comparison diagram analogous to \ref{interp-1node}:

\eqspl{interp-sep-azimuthalBN}{
\begin{matrix}
{\be'}^{-n,0}(n\lf{D^\dag})&\to& \aE\\
\downarrow&&\downarrow\\
 \Lambda_2(\omega(-n\lx))(n\lf{D^\dag})&\to&\Lambda_2(\omega)
\end{matrix}
} where $n=2$ (sep case) or $n=1$ (bisep case). The left column
restricts, up to a twist, to the azimuthal Brill-Noether mapping of
$\lf{X}.$\par
Because $X_0$
is non-hyperelliptic,
at least one of the pairs $(\lf{X}(\theta), \lf{\theta}), (\rt{X}(\theta), \rt{\theta})$ is non- hyperelliptic. We first
analyze the former case (the two are distinct in that
$\lf{X}(\theta)$ is the side containing $z$). Within this case,
the worst-case scenario
is that $\lf{X}(\theta)$ (as plain curve)
and $(\rt{X}(\theta), \rt{\theta})$ (as pair)
are hyperelliptic. In particular,
$\lf{X}(\theta)$ and $\rt{X}(\theta)$ are of
 semicompact type. Then, because the pair
$(\lf{X}(\theta), \lf{\theta})$ is \nhep, either
$\lf{\theta}$ (bisep case) or $2\lf{\theta}$
(sep case)  is
a (degree-2) nonhyperelliptic divisor. If $z$ is not contained
in any single 2-component of $\lf{X}(\theta)$, it is easy
to see that the canonical system of $\lf{X}(\theta)$ already
separates $z$, hence $\aphi$ is surjective over $z$ (this is
independent of any azimuthal data). Thus
 we may assume $z$ is contained in a 2-component $Y$ of $\lf{X}(\theta)$. If $Y$ is the 'right-extremal' 2-component,
 i.e. the one that contains
$\lf{\theta}$, then by the above comparison diagram,
the azimuthal \bn image $\im(\aphi(z))$
contains the \bn image of $\omega_Y(\lf{\theta})$
(bisep case) or $\omega_Y(\lf{2\theta})$ over $z$,
which is already onto by Lemma 3.8 of \cite{canonodal}. If $z$ is not
contained in the
extremal 2-component, then a 2-component $Y$
 containing $z$  contains $\lf{\theta'}$
for some sep or bisep of $\lf{X}(\theta)$
(necessarily $\neq \theta$),
so that $\rt{X}(\theta')$ is non-
hyperelliptic and then, again by
 using the comparison diagram, the \bn image over $z$ contains that
of $\omega_Y(3\lf{\theta'})$ (sep case) or $\omega_Y(2\lf{\theta'})$ which is onto by Lemma 3.1 of \cite{canonodal}. This concludes the case where $z$ is contained in a \nhep side
of $\theta$.\par
At this point, we may assume that for any sep or maximal bisep
$\theta$ having $z$ on one side, say left,
the pair $(\lf{X}(\theta), \lf{\theta})$ is hyperelliptic. Suppose such $\theta$ actually
exists. Then if $z$ is not contained in any 2-component
of $\lf{X}(\theta)$, we can conclude as above. Else, our
assumption implies that $z$ must be
contained in a left-extremal 2-component  $Y$, of $\lf{X}(\theta)$, i.e. on that has the form
$Y=\lf{X}(\theta')$ for some sep or bisep $\theta'$
(possibly $=\theta$) (otherwise, $z\subset Y\subset \rt{X}(\theta")$
for some $\theta"$ and $\rt{X}(\theta")$ cannot be hyperelliptic). Therefore, we may as well assume
$Y=\lf{X}(\theta)$. Thus $(Y, \theta)$ is
2-inseparable hyperelliptic, i.e.  either $\theta$ (bisep case)
or $2\theta$ (sep case) is a member of the $g^1_2$.
Then by  Lemma 3.1 of \cite{canonodal}, the system
$|\omega_Y(3\lf{\theta})|$ (sep case)
 or $|\omega_Y(2\lf{\theta})|$
 (bisep case) is very ample, therefore the
left column of \eqref{interp-sep-azimuthalBN} is surjective over $\lf{D}^\dag$. Therefore $\aphi$ is surjective over the
interior of $\lf{D^\dag}$ and hence over $z$.
\par It remains to consider the case where $z$ is not on one side of any sep or maximal bisep. This implies $z=(z_1, z_2),
z_i\in Y_i$ where $Y_1, Y_2$ are extremal vertices in the
2-separation tree $G_2(X_0)$. But then because $g(Y_i)>0$,
$z_1, z_2$ are already separated by the canonical system of $X_0$, and in particular $\aphi(z)$ is surjective. \par

This completes that the proof that if $X_0$ is
non-hyperelliptic then assertion (ii) holds for $X_0$ and the
degeneracy locus of $\aphi$ on $(X_0)\sbc 2.$ is supported on,
hence coincides with,
$\rR^\sep(X_0)$. Then a computation as in \S \ref{sep-local-forms-sec} (which is valid locally over $B$) shows that for $X_0$ non-hyperelliptic,  the locus $\bD_2(\aphi)$ coincides schematically with
$\rR^\sep(X/B)$ over a neighborhood of $0\in B$.
This completes the proof of Theorem \ref{main-thm}.

\end{proof}
Now a dimension count similar to the proof of Proposition
\ref{single-sep-prop} shows the following result,
which also follows from the Harris-Mumford theory of admissibe covers:
\begin{cor}
A stable curve is hyperelliptic if and only if it is the limit of smooth
hyperelliptic curves.
\end{cor}
\begin{cor}
Notations as above, assume $X/B$ admits a section $s$ 
disjoint from the singular locus of $X$, and let
$Y_s\subset X\sbc 2._B$ be the inverse image of the locus of
schemes meeting $s(B)$. Then
\eqspl{}{
[\bD_2(\aphi)]\cap [Y_s]=[\he\sbc 2._B]\cap [Y_s]
}
\end{cor}
\begin{proof}
Immediate from the fact that $s(B)$ is disjoint from the
singular locus of $X/B$, i.e. the locus of singular points of fibres,
hence $Y_s$ is disjoint from $R^\sep(X/B)$.
\end{proof}
Various sheaves associated with degeneracy loci admit resolutions by Eagon-Northcott-type complexes (see \cite{rkl-positivity}, I.B.2 or \cite{eisenbud}). The simplest of these yields
\begin{cor}
We have, modulo $R^\sep(X/B)$:
\eqspl{}{
[\O_{\he\sbc 2._B}]=c(\ae\otimes\Lambda_2(\omega)^*)/
c(\sym^2(\ae)\otimes\det(\Lambda_2(\omega))^*).
}
\end{cor}

\newsection{Excess degeneracy and fundamental class}\lbl{excess-porteous-sec}
\newsubsection{Excess Porteous}
Motivated by Theorem \ref{main-thm}, our purpose here is to
analyze intersection-theoretically a simple excess-degeneracy
situation, aiming for a Porteous-type formula.\par
More explcitly, let
\[\phi:E\to F\]
be a map of vector bundles of respective ranks $g, 2$ on a
variety $X$, and denote by $\bD_2(\phi)$ its degeneracy
scheme, locally defined by the $2\times 2$ minors.
Let $Z$ be a sum of irreducible components of $\bD_2(\phi)$,
 purely of (possibly excessive)
 dimension
  \mbox{$\l\geq g-1$.} The 'expected' or virtual class
of the degeneracy locus is computed by Porteous's formula
(\cite{ful}, Thm. 14.4 and Example 14.4.9):
\eqspl{porteous}{[\bD_2(\phi)]_\vir=(-1)^{g-1}[c(E)/c(F)]_{g-1}.}
We are interested in the contribution of $Z$
to the virtual degeneracy locus of $\phi$, i.e. the
virtual fundamental class of $Z$, $[Z]_\vir$. Were the
 rank of $F$ equal to 1 rather than 2, i.e.
were the degeneracy locus a zero locus, this class would be computed directly by the 'local' intersection theory of  Fulton-MacPherson
\cite{ful}. Here we will spell out a simple extension,
 implicit in \cite{ful}, to the case where $F$ has rank 2. \par Assume\begin{enumerate}
\item The rank of $\phi$ is at least 1 everywhere.
\item Off a subset $Y\subset Z$ of codimension $g$ in $X$,
$Z$ is regular, i.e.
a local complete intersection of codimension $\l$ in $X$.
\end{enumerate}
Since we are interested in codimension $g-1$ or less, $Y$ can be ignored, i.e. we can replace $X$ by $X\setminus Y$ and assume $Y=\emptyset$.
Set
\[ K=\ker (E_Z\to F_Z), C=\coker(E_Z\to F_Z), N_Z=\Hom(I_Z, \O_Z)\]
where $E_Z=E|_Z, F_Z=F|_Z$.
As $Y=\emptyset$, these are $Z$-locally free of ranks $g-1, 1, \l$,
respectively. Let
\eqspl{z-vir}{
z=[Z]_\vir=[\frac{c(K^*\otimes C)}{c(N_Z)}]_{g-1-\l}.
}
\begin{prop}\lbl{excess-porteous-prop}
Assumptions and notations as above, the contribution of $Z$ to the class of the
degeneracy locus of $\phi$ is $z$. In particular, if $\bD_2(\phi)=U\cup Z$
with $U$ of codimension $g-1$ and $Z\cap U$ purely
of codimension $\geq g$, then $[U]+z$ is the virtual degeneracy class given
by Porteous's formula \eqref{porteous}.
\end{prop}
\begin{proof}
In the case of a zero-locus in place of degeneracy locus, the analogous formula
is given in \cite{ful}. To extend to the degeneracy situation, we pass
as usual to the projective bundle $\P(F)$ (of rank-1 quotients (!), unlike in \cite{ful})
with its Grothendieck quotient  \mbox{$\pi^*F\to\O(1)$.}
There we have a composite map \[\pi^*E\to\O(1)\] whose zero-locus contains a component $\tilde Z$ projecting isomorphically to $Z$. By the excess zero-set
formula of Fulton-MacPherson, the contribution of $\tilde Z$ to the zero-set is
\[[\frac{c(E^*(1))}{c(N_{\tilde Z})}]_{g-1-\l}.\] Now
by  a reult of Golubitsky and Guillemin  (computation of normal bundles to degeneracy loci,
\cite{gol-gui}, p. 145, recounted in \cite{har-tu}),
the latter
corresponds to \eqref{z-vir} under the projection $\tilde Z\to Z$.
Indeed the difference between $E^*(1)$ and $K^*\otimes C$ is precisely
the 'vertical' subsheaf, that is also the difference between
$N_{\tilde Z}$ and $N_Z$, and these cancel out.

\end{proof}
\newsubsection{Fundamental class of the hyperelliptics}
\lbl{class-hypell-sec}
We aim to apply Proposition \ref{excess-porteous-prop} with
$Z=\rR^\sep(X/B)=\sum\limits_{\theta}R_\theta$. In light of
Theorem \ref{main-thm}, this will compute the class of the
extended hyperelliptic locus in $X\sbc 2._B$.\par
We will evaluate \eqref{z-vir} to obtain the contribution of $R_\theta$ for a fixed $\theta$
to the degeneracy locus of the modified Brill-Noether $\phi_\theta$. Because the $R_\theta$ are disjoint for different $\theta$, the total contribution of $\rR^\sep(X/B)$
will be the sum of those. The ultimate result is \eqref{[Rsep]}. The reader only interested
in the hyperelliptic class on the base $B$ is in for some good news: (s)he may skip ahead to Corollary
\ref{fund-class-base-cor} which shows that the contribution of $R^\sep$ to the latter vanishes for trivial reasons.

To begin with, it is clear, e.g. from \eqref{Lambda-theta}, that the image
of $\phi_\theta$ at each point $z\in R_\theta$,
i.e. a length-2 scheme with ideal $I_z$ co-supported on $\theta$,
 coincides
with the 1-dimensional space of sections of $\omega\otimes\O_z$ vanishing on the length-1 subscheme $\theta\subset z$,
i.e. the quotient
\[\omega\otimes I_\theta/\omega\otimes I_z=\omega\otimes(I_\theta/I^2_\theta)/\omega\otimes(I_z/I^2_\theta).\]
Clearly, further echelon modifications of $\be_\theta$
corresponding to other sep/biseps don't change that image, so the same is true for $\aphi$ in place of $\phi_\theta$.
Because $\omega|_\theta=\O_\theta$ by residues and $I_\theta/I^2_\theta=\lf{\psi}\oplus\rt{\psi}$,
this
image sheaf coincides with
$\O_{R_\theta}(1)$ under the identification \[R_\theta=\P(\lf{\psi}\oplus\rt{\psi}).\]
Because $\O_{R_\theta}(1)=\det(\Lambda_2(\omega))\otimes\O_{R_\theta}$,
it follows that the cokernel of $\aphi|_{R_\theta}$ is the trivial bundle. Thus,
setting $h_\theta=c_1(\O_{R_\theta}(1))$,
we have, with the notations of \eqref{z-vir},
\[ C=\O, K\sim [\ae]-h_\theta,\] therefore
the numerator of \eqref{z-vir} evaluates to
\[c(K^*\otimes C)= \frac{c(\ae^*_\theta)}{1-h_\theta}.\] As for the denominator, note first
that we can write
\[X\sbc 2._{\del_\theta}=\lf{X}\sbc 2._{\del_\theta}\cup \rt{X}\sbc 2._{\del_\theta}\cup
X_{12}\] where $X_{12}$ is the inverse image of
$\lf{X}\times_{\partial_\theta}\rt{X}$ and can be identified with the blowup of the latter in $(\lf{\theta}, \rt{\theta})$
with exceptional divisor $R_\theta$.   We have
\[ N_{R_\theta/X_{12}}=\O_{R_\theta}(-1).\] On the other hand,
writing $_*\theta\sbr 2.$ for the locus of subschemes of $_*X\sbr 2.$ supported on $_*\theta$, \mbox{$*=\lf{}, \rt{}$}, we have
\eqsp{
N_{X_{12}/X\sbc 2._B}|_{R_\theta}=\O_{R_\theta}(\partial_\theta-\lf{X}\sbc 2._{\del_\theta}- \rt{X}\sbc 2._{\del_\theta})
=\O_{R_\theta}(\partial_\theta-\lf{\theta}\sbr 2.-\rt{\theta}\sbr 2.)\\
=\O_{R_\theta}(-\psi_1-\psi_1-\lf{\theta}\sbr 2.-\rt{\theta}\sbr 2.) } where the latter
results from a formula we have used
 before, due to Faber \cite{fab-alg}, \S2, which says, in
our notations, that
\eqspl{faber-formula}{\del_\theta|_{\del_\theta}=-\psi_1-\psi_2.}
Note that $_*\theta.R_\theta$ can be identified with the
section $\P(_*\psi_i)$ on $\P(\lf{\psi}\oplus\rt{\psi})=R_\theta$.
Now on $R_\theta$, we have
\eqspl{h-theta}{\lf{\psi}+\rt{\theta}\sbr 2. \sim\rt{\psi}+\lf{\theta}\sbr 2.\sim h_\theta.}
Therefore we have \[c_1(N_{X_{12}/X\sbc
2._B}|_{R_\theta})=-2h_\theta.\] Therefore
\eqspl{normal-Rtheta}{c(N_{R_\theta/X\sbc 2._B})=(1-h_\theta)(1-2h_\theta).}
All in all, we get for the contribution of $R_\theta$:
\eqspl{}{
[R_\theta]_\vir=[\frac{c(\aE^*)|_{R_\theta}}{(1-h_\theta)^2(1-2h_\theta)}]
_{g-3}.
} It remains to compute the numerator. To this end we now work
over $R_\theta$.
Note that the restriction of $\bE$ on the boundary $\partial_\theta$ (a fortiori, on $R_\theta$) splits as
$\lf\bE\oplus\rt\bE$,\  
where each $_*\bE$ is the Hodge bundle of $_*X/\delta_\theta$, and $_*\be$
admits an exact sequence induced by restriction
of differentials on $_*\theta$:
\[\exseq{_*\bE\inv}{_*\bE}{_*\psi},\  _*=\lf{}, \ \rt{}.\] Thus $_*\bE\inv$ is the
rank-
$(g(_*X)-1)$ bundle of relative differentials of
$_*X/\delta_\theta$ vanishing on $_*\theta$ (it is the
 bundle denoted $\be^{0,-1}$ or $\be^{-1,0}$ in \S \ref{mod-bn-sep-sec}) and $_*\psi$
is the usual line bundle on $\delta_\theta$ (and also by pullback, on $R_\theta$).
The saturated image $_{\a*}\bE$ of
$_*\bE$ in $\aE|_{\delta_\theta}$ fits in
an exact sequence
\[\exseq{_{\a*}\bE\inv(2_\dag D)}{_{\a*}\bE}{_*\psi\otimes\O(_\dag D)}\]
where $\dag$ is the opposite
 of $*$, i.e $\dag=\lf{}$ if and only if $*=\rt{}$ and $_{\a*}\be\inv$ is the appropriate subbundle of $\ae(_*X/\del_{\theta})$, i.e. the
azimuthal modification of the Hodge bundle of the family
$_*X/\del_\theta$, namely the subbundle defined by vanishing on
$_*\theta$ analogously to $_{*}\be\inv$ above; thus as virtual bundles,
we have \[[_{\a*}\be\inv]=[\ae(_*X/\del_\theta)]- _*\psi\] in which
$_*\psi$ is viewed via pullback from the base as a line bundle on $(X_i)\sbc
2._{\delta_\theta}$. Moreover, clearly \[_*D.R_\theta=_*\theta\sbr 2..\]

Putting it
all together, we obtain \eqspl{}{
c(\aE^*|_{R_\theta})=
\frac{c(_{\a}\lf{\bE}^*(-2\lf{\theta}\sbr 2.)).(1-\lf{\psi}-\rt{\theta}\sbr 2.)}
{1-\lf{\psi}-2\rt{\theta}\sbr 2.}
\frac{c(_\a\rt{\bE}^*(-2\rt{\theta}\sbr 2.)).(1-\rt{\psi}-\lf{\theta}\sbr 2.)}
{1-\rt{\psi}-2\lf{\theta}\sbr 2.} \cap [R_\theta].} Now in view of \eqref{h-theta} and \eqref{normal-Rtheta}, we obtain
\eqspl{}{
[R_\theta]_\vir=[\frac{c(_{\a}\lf{\bE}^*(-2\lf{\theta}\sbr 2.))
c(_\a\rt{\bE}^*(-2\rt{\theta}\sbr 2.))}
{(1-2h_\theta)(1-2h_\theta+\lf{\psi})(1-2h_\theta+\rt{\psi})}]_{g-3}
\cap[{R_\theta}]. } Since for different $\theta$, $R_\theta$ are
disjoint, it follows that \eqspl{[Rsep]}{
[\rR^\sep]_\vir=\sum\limits_{\theta}
[\frac{c({{\aE(\lf{X}(\theta))}}^*(-2\lf{\theta}))
c({{\aE(\rt{X}(\theta))}}^*(-2\rt{\theta}))}
{(1-2h_\theta)(1-2h_\theta+\lf{\psi}_{{\theta}})
(1-2h_\theta+\rt{\psi}_{{\theta}})}]_{g-3} \cap[{R_\theta}] }
where we recall that on $R_\theta$, we have
\[h_\theta=\lf{\psi}_{{\theta}}+\rt{\theta}\sbr 2.=\rt{\psi}_{{\theta}}+\lf{\theta}\sbr 2..\]
Thus we have finally obtained our main result on the hyperelliptic class:
\begin{thm}\label{fund-class-thm}
The fundamental class on $X\sbc 2._B$ of the closure of the locus of hyperelliptic divisors on smooth hyperelliptic curves is give by
\eqspl{he-fund-class-eq}{
[\he^2]=(-1)^{g-1}[c(\ae)/c(\Lambda_2(\omega))]_{g-1}
-[\rR^\sep]_\vir
} where $[\rR^\sep]_\vir$ is given by \eqref{[Rsep]}.\qed
\end{thm}
As mentioned in the Introduction, the image of this class on $B$ is fully computed in genus 3, e.g. in
\cite{internodal}, \S 4.5. For the computation of the boundary in genus 4, see
\S \ref{genus-4-example}.
 The computation of this image can be approached via the following
 \begin{cor}\label{fund-class-base-cor}
 If the image on  $B$ of the locus of smooth hyperelliptic fibres of $X/B$ is of codimension $g-2$, then
 the fundamental class of its closure is, where $\omega:=\omega_{X/B}$:
 \eqspl{}{
 [\he]_B=(\pi_B)_*(\frac{1}{2g-2}(-1)^{g-1}[c(\ae)/c(\Lambda_2(\omega))]_{g-1}(\omega)\sbr 2.).
 } 
 \end{cor}
 \begin{proof} Each smooth hyperelliptic curve contributes a linear pencil to
 the locus $\he$ above, and $2g-2$ many members of this pencil will have a point in
 common with a fixed canonical divisor. Therefore
the formula  follows from \eqref{he-fund-class-eq} by multiplying by $\omega\sbr 2.$, 
except for the vanishing of the term
 coming from $\mathrm R^\sep$. This term vanishes because, as is well known from residues,
 $\omega|_\theta$ is trivial for any node $\theta$, therefore $\omega\sbr 2.$ is trivial on
 $\mathrm R^\theta$.
 \end{proof}
\newsection{Azimuthal intersection theory}
\label{azi-int-sec}
\newsubsection{Chern classes of azimuthal bundles}\label{chern-of-azi}
We will assume WLOSG\footnote{without loss of significant generality}
that $g>2$.
The azimuthal bundle $\ae$ is obtained by a
 polyechelon modification (\cite{echelon}, \S 3),
 i.e. sequence of mutually transverse echelon modifications corresponding to seps and biseps, and to compute its
Chern classes it suffices to compute how these classes change
on passing from one bundle $E$, a partial azimuthal modification, to 
a further partial azimuthal modification, viz. $E_\theta$ or $E_\vtheta$. Note that
$E$ inherits from the initial $\be$ subsheaves $E^{.,.}$
and quotients $\lf{E}, \rt{E}$, which are corresponding
modifications of the analogous sheaves associated to
$\be$. \par
\subsubsection{Case of sep}
We begin with the case of a sep. Thus,
we fix a sep $\theta$ and compute the Chern classes of the echelon modification $E_\theta$ (see \S \ref{mod-bn-sep-sec}). We assume to start with that $\lf{g}, \rt{g}>1$. Then the echelon modification
is obtained in 2 steps
along each of the divisors $\lf{D}, \rt{D}$, which are mutually
\emph{ disjoint}. Therefore it will suffice to work out the left modification. This takes the form
\[E\subset E_1\subset E_2.\]
As virtual bundle, $E_1$ has the form
\eqspl{}{ E_1\sim E+\rt{E}\otimes\O_{\lf{D}}(\lf{D})
}
(where $\rt{E}$ is, in general, the appropriate azimuthal modification of $\pi_*(\omega_{\rt{X}/\del_\theta})$
with respect to other seps/biseps; 
in case $E$ is the (unmodified) Hodge bundle itself, $\rt{E}=\pi_*(\omega_{\rt{X}/\del_\theta})$.
The difference is formal difference as virtual bundle).
Similarly,
\eqspl{}{
E_2\sim E_1+\rt{E}^0\otimes\O_{\lf{D}}(2\lf{D}), 
}
where $\rt{E}^0$ is an appropriate azimuthal modification of $\pi_*(\omega_{\rt{X}/\del_\theta}(-\rt{\theta}))$
and equal to the latter push-forward when $E$ is the Hodge bundle. Thus, when $E$ is the Hodge bundle, 
then as virtual bundle
\eqspl{}{
\rt{E}^0=\rt{E}-\rt{\psi}.
}
In particular, $\rt{E}=0$ if $\rt{g}=1$.

%
%

Therefore, since $\lf{D}$ and $\rt{D}$ are
disjoint, we have in all, 
\eqspl{}{
E_\theta\sim E+&(\rt{E}\otimes\O_{\lf{D}}(\lf{D})+(\rt{E}-\rt{\psi})\otimes\O_{\lf{D}}(2\lf{D}))\\
+& \mathrm{mirror}(.),
} 
where by convention the 'mirror' of an expression is the corresponding one
with each 'L' replaced by 'R'' and vice versa. 
We recall from Lemma \ref{conormal-sep-lem} that $\O_{\lf{D}}(\lf{D})=\psi\inv\otimes\O_{\lf{D}}(-\lf{\theta}\sbr 2.)$.
We will use \S \ref{asd} and its notations, e.g. \eqref{Q-dual-eq}. Thus
set
\eqspl{q-def}{\lf{q}=&Q(\lf{D}, \rt{E})c(\rt{E\dual}), \rt{q}=Q(\rt{D}, \lf{E})c(\lf{E\dual})\\
\lf{q}=&
e+\binom{e}{2}D-c_1+\binom{e}{3}D^2-(e-1)Dc_1-(e-1)c_1^2+2(e-1)c_2\\
&+\binom{e}{4}D^3-\binom{e-1}{2}D^2c_1-\binom{e-1}{2}Dc_1^2
+c_1c_2+(e-4)c_3
}where $e=\rt{g}, c_i=c_i(\rt{E})$ and likewise for the mirror; of course, $\lf{q}, \rt{q}$ depend
on $E, \theta$ as  well.

Then when $E$ is the unmodified Hodge bundle, we have
$c(i_{\lf{D}*}(\rt{E}(\lf{D})))=1+i_{\lf{D}*}(\lf{q})$. \par
The Chern class of $i_{\lf{D}*}(\rt{E}^0(2\lf{D})$ 
can be computed using
the general Riemann-Roch without denominators directly, but, with a view
to an application in genus 4, we shall do so explicitly only for
$\rt{g}=2,3$. We shall omit all terms which have codimension $>2$ in the boundary
or  whose base image has codimension
$>1$ in the boundary . Because the class in question will ultimately be multiplied
(see \eqref{genus4-onBase}) by $\omega_1$,
which has trivial restriction on nodes, we may also omit classes supported on nodes.
Now note that $(\lf{\theta}\spr 2.)^2$ is the sum of a node-supported class (locus
of schemes contained in $\lf{\theta}$) and $-\lf{\psi}\lf{\theta}\spr 2.$.
Therefore, modulo negligible classes, we have
\[ (\lf{\theta}\spr 2.)^2\sim -\lf{\psi}\lf{\theta}\spr 2..\]
 Set $\rt{F}=\rt{E}^0\otimes\O_{\lf{D}}\sim (\rt{E}-\rt{\psi})_{\lf{D}}, D=\lf{D}$. Then
\eqspl{}{
\frac{1}{c(\rt{F}(D))}\sim&1-(\rt{\lambda}_-\rt{\psi}-(\rt{g}-1)(\psi+\lf{\theta}\spr 2.)\\
&+(\rt{\lambda}-\rt{\psi}-(\rt{g}-1)(\psi+\lf{\theta}\spr 2.))^2+(\rt{g}-1)(\rt{\lambda}-\rt{\psi})(\psi+\lf{\theta}\spr 2.)
-\binom{\rt{g}-1}{2}(\psi+\lf{\theta}\spr 2.)^2\\
\sim_{(\rt{g}=2)}&1+(-\rt{\lambda}+\rt{\psi}+\psi+\theta)-2\lf{\theta}\spr 2.(\rt{\lambda}-\rt{\psi}-\psi)\\
\sim_{(\rt{g}=3)}&1-(\rt{\lambda}-\rt{\psi}+2(\psi+\lf{\theta}\spr 2.)
+5\lf{\theta}\spr 2.(-\rt{\lambda}+\rt{\psi}+2\psi).
}
Set
\eqspl{q0-def}{
\lf{q}^0&=\frac{Q(\lf{D}, \rt{F})}{c(\rt{F}(D))}=_{(\rt{g}=2)}\frac{1}{c(\rt{F}(D))},\\
1+\lf{m}(\theta, E)&=(1+i_{\lf{D}*}(\lf{q}))(1+i_{\lf{D}+}(\lf{q}^0)).
}


Note that, because $\lf{D}\cap\rt{D}=\emptyset$, we have $i_{\lf{D}}(*)i_{\rt{D}}(*)=0$.
Also, note the elementary formula, for any divisor $D$,
\eqspl{}{ 
i_{D*}(x)i_{D*}(y)=i_{D*}(xyD)
}(clear if $D$ moves, general case follows by moving lemma; or, it follows from the
universal formula $(\Delta_X)._{X\times X}(D\times D)\sim \Delta_D.\O_D(D)$).
Consequently,
\eqspl{}{
\lf{m}(\theta, E)=i_{\lf{D}*}(\lf{q}+\lf{q}^0-(\psi+\theta)\lf{q}\lf{q}^0).
}
Then the above becomes
\eqspl{}{
c(E_\theta)=c(E)(1+ \lf{m}(\theta, E))(1+\rt{m}(\theta, E))=c(E)(1+\lf{m}(\theta, E)+\rt{m}(\theta, E)).
}

Here $E$ will in general be a partial azimuthal modification of the Hodge bundle $\be$
with respect to some biseps and  seps different from $\theta$.
For $\theta$ of type (2,2,), we compute, where $E$ is the unmodified Hodge bundle:
\eqspl{}{
\lf{m}(\theta, E)=i_{\lf{D}*}(3-2(\rt{\lambda}+\psi+\lf{\theta}\spr 2.)+\rt{\psi}
+2\lf{\theta}\spr 2.\rt{\lambda}).
}
Consequently, we get for the (2,2) sep:
\eqspl{(2,2)-sep}{
c(E_\theta)=1+&c_1(E)\\+&3(\lf{D}+\rt{D}))+3c_1(E)(\lf{D}+\rt{D})\\
+&(i_{\lf{D}*}(-2(\rt{\lambda}+\psi+\lf{\theta}\spr 2.)+\rt{\psi})
+2i_{\lf{D}*}(\lf{\theta}\spr 2.\rt{\lambda}))\\
+&\mathrm{mirror}(.).
}
\par
\subsubsection{Contribution of node scroll}
Next, we study the restriction of $\be_\theta$ on various scrolls $R_{\theta'}\simeq\P(\lf{\psi}(\theta')\oplus\rt{\psi}(\theta'))$, where
$\theta'$ is some sep of $X/B$, possibly equal to $\theta$. By Corollary \ref{fund-class-base-cor}.
this does not affect the class of the hyperelliptic locus in the base.
First consider the case $\theta'=\theta$.
Note that
 \[_*{D}.R_\theta=\P(_*\psi):=_*S_\theta, \ *=\lf{}, \rt{} \]
(transverse intersection). Thus $_*S_\theta$ is a section
of the $\P^1$-bundle $R_\theta$. Therefore restricting ${\be_\theta}$
on $R_\theta$ is elementary. If $\theta'\neq\theta$, then over each
point of $\del_{\theta'}\cap\del_\theta$, $R_{\theta'}$ is contained
in the interior of  at most one of $\lf{D}, \rt{D}$. In the case of a versal family $X/B$, assuming $\lf{X}(\theta), \lf{X}(\theta')$ have different genera,
$R_{\theta'}$ is contained
in the interior of  precisely one of $\lf{D}, \rt{D}$
and is disjoint from the other.
Therefore, if we arrange notations so that each (unoriented) sep
$\theta$ appears twice, once with each orientation, we have
\[(\lf{D}+\rt{D}).R_{\theta'}=\del_\theta .R_{\theta'}.\]
The case of a bisep $\vtheta$ is similar: here we can always
assume $\theta'\in\lf{X}(\vtheta)$ and then $\vtheta$ contributes a modification along $\Xi(\vtheta)$.
Putting all together, we can write
\eqspl{}{\ae|_{R_{\theta'}}=(\be(\lf{S}_{\theta'}+
\rt{S}_{\theta'}+
\sum\limits_{\theta\neq \theta'}2\del_\theta+\sum\limits_{\theta'\in\lf{X}(\vtheta)}\Xi(\vtheta))
+A+B+C)|_{R_{\theta'}}
} where
\eqspl{}{
A=\sum\limits_{\substack{\theta\neq \theta'\\\lf{D}(\theta)\cap R_{\theta'}\neq\emptyset}}-2\lf{\bE(\theta)}\otimes\O_{\del_\theta}(2\del_\theta)
-\lf{\bE(\theta)}\otimes\O_{\del_\theta}(\del_\theta)-\lf{\psi}_{\theta'}\inv\otimes
\O_{\del_\theta}(2\del_\theta) }
\eqspl{}{
B=-&2\lf{\bE(\theta')}\otimes\O_{\lf{S}_{\theta'}}(2\lf{S}_{\theta'})-
\lf{\bE(\theta')}\otimes\O_{\lf{S}_{\theta'}}(\lf{S}_{\theta'})-
\lf{\psi}(\theta')\inv\otimes\O_{\lf{S}_{\theta'}}(2\lf{S}_{\theta'})\\
-&2\rt{\bE(\theta')}\otimes\O_{\rt{S}_{\theta'}}(2\rt{S}_{\theta'})-
\rt{\bE(\theta')}\otimes\O_{\rt{S}_{\theta'}}(\rt{S}_{\theta'})-
\rt{\psi}(\theta')\inv\otimes\O_{\rt{S}_{\theta'}}(2\rt{S}_{\theta'})
}
\eqspl{}{
C=-\sum\limits_{\theta'\in\lf{X}(\vtheta)}\lf{E}(\vtheta')(\Xi(\vtheta)).
}
Note that the divisors involved, i.e. $\del_\theta$ and $\Xi(\vtheta)$ are transverse to $R(\theta')$.\par
\subsubsection{Case of bisep}
The case of the modification corresponding to an
  oriented proper
   bisep $\vtheta$ is similar and simpler in the sense that the modification has just one step, i.e. equals $E_1$,
  and there is no residual scheme like $R(\theta)$. Here $D$
 is the divisor $\Xi_\vtheta$ (see \eqref{Xi-vtheta}) and we have
 \eqspl{azi-bundle-bisep}{
 E_\vtheta\sim E(D)-\lf{E}\otimes\O_D(D)\sim E+\rt{E}^0\otimes\O_D(D)
 } where $\O_D(D)$ is given by (the dual of) \eqref{Xi-vtheta-conormal} and of course $\lf{E}=\lf{E}(\vtheta)$, the
 quotient associated to the
 oriented  bisep $\vtheta$, which is the appropriate echelon modification of $\pi_*(\omega_{\lf{X}(\vtheta)}(\vtheta))$, and is defined over $\Xi_\vtheta$.
 Of course, at the end of the day when both $\vtheta$ and its opposite
 are accounted for, \eqref{azi-bundle-bisep} will be replaced by
 \eqspl{azi-bundle-bisep-rightleft}{
 E(\lf{D}+\rt{D})-\lf{E}\otimes\O_{\lf{D}}(\lf{D})-\rt{E}\otimes\O_{\rt{D}}(\rt{D})\\
 \sim E+\rt{E}^0\otimes\O_{\lf{D}}(\lf{D})+\lf{E}^0\otimes\O_{\rt{D}}(\rt{D}).
}
 See \S \ref{w-bundles-sec} on how to compute the terms appearing in \eqref{azi-bundle-bisep}.\par
 \subsection{Hyperelliptic class in genus 4}\label{genus-4-example}
 In \cite{faber-chow2}, Faber gives a set of 14 generators for the group $R^2$ of tautological classes in degree 2
 and genus 4, and shows that these generators satisfy a unique relation, so that $R^2$ is of rank 13.
 In \cite{faber-pandh}, Faber and Pandharipande state a
  formula for the hyperelliptic class in genus 4 in terms
 of Faber's generators; their approach is via 'undetermined coefficients', determining the coefficients
involved in expressing the hyperelliptic class in terms of Faber's generators
  by pairing with suitable test classes. 
To compare their results with ours, we 
note at the outset that our azimuthal modifications involve only reducible curves
 and separating (collections of) 
 nodes and consequently, any explicit version of our formula \eqref{fund-class-base-cor} will have coefficient zero
 for the class $\delta_{00}$ of irreducible binodal curves. Therefore, for comparison, we
 modify the formula of \cite{faber-pandh}, Prop. 5, by adding a suitable multiple (viz. -4/9) 
 of Faber's relation, so as to yield a (uniquely determined) formula with coefficient 0 for $\delta_{00} $,
 viz.
 \eqspl{fab-pan-form}{
 [\overline{\he_4}]_Q=&\frac{1}{6}(\kappa_2+63\lambda_1^2-16\lambda_1\delta_0-
 66\lambda_1\delta_1+\delta_0^2+8\delta_0\delta_1+13\delta_1^2+27\delta_2^2\\
 &-2\gamma_1-\frac{1}{3}\delta_{01a}+4\delta_{11}+18\lambda_1\delta_2+18\delta_1\delta_2).
 }
 Here $\gamma_1$ is the class of the locus of curves with a proper bisep of type
 $(2,1)$, i.e. unions of curves of genera 2,1, $\delta_{01a}$ is the class of the 
 (codimension-2) locus of curves containing
 a 1-nodal rational tail, and $\delta_{11}$ is the locus of 
 chains of genera 1,2,1, i.e. the curves admitting
  an improper bisep of type (1,2,1). Note that,
 in our notation,
 \[\delta_{01a}=12\lf{\psi}_{\delta_1}=12\lf{\lambda}_{\delta_1}.\]
 Also, $\delta_{11}$ is the (normal crossing)  double locus of the $\Delta_1$ boundary divisor.\par
 The formula \eqref{fab-pan-form} implies an analogous one (same coefficients) for the hyperelliptic class
 $[\he]_B$ on the base of any given family of stable curves.
 Our purpose here is to recover \eqref{fab-pan-form} via azimuthal modifications.
 In our approach, the generating classes are not presumed but arise out of the construction.
 \subsubsection{The uncorrected formula}
  It will be convenient to work in an ordered version of $X\sbc 2._B$, namely\nl
 $X\sbc 2._B\times_{X\spr 2._B}W^2(X/B)$ where $W^*(X/B)$ is the flag
 Hilbert scheme (cf. \cite{geonodal}), which dominates $X^2_B$. We will
 denote by $\omega_i$ the $i$-th pullback $p_i^*(\omega)$ etc. 
 Then the analogue of Corollary \ref{fund-class-base-cor} reads
  \eqspl{genus4-onBase}{
 [\he]_B=p(_\a\be):=\frac{-1}{6}\pi_{B*}( [\frac{c(_\a\be)}{1+\omega_1+\omega_2
   -\Gamma\spr 2.+\omega_1\omega_2
   -\omega_1\Gamma\spr 2.}]_3\omega_1).}
   We note that
    \[1/c(\Lambda_2(\omega))=1-(\omega_1+\omega_2-\Gamma)+
     \omega_1^2+\omega_2^2+\omega_1\omega_2-4\omega\Gamma+(\mathrm{higher-order\  terms}).\]
   We can naturally decompose $p(E)=\sum\limits_{i=0}^3p_i(E)$ where 
   \[p_i(E)= c_i(E)q_{3-i},\  q_{3-i}:=\omega_1(\frac{-1}{6c(\Lambda_2(\omega))})_{3-i}\]
    ($p_0(E)=q_3$ is of course independent of $E$).
   We note that the analogue of the RHS of \eqref{genus4-onBase} for the unmodified
   Hodge bundle  was essentially computed in \cite{curvilinear} and reproduced
   in \cite{faber-chow2} and elsewhere, and equals
   \eqspl{genus4-umodified}{p(\be)=
   \frac{-1}{6}\pi_{B*}( [\frac{c(\be)}{1+\omega_1+\omega_2
      -\Gamma\spr 2.+\omega_1\omega_2
      -\omega_1\Gamma\spr 2.}]_3\omega_1)=\frac{1}{6}(\kappa_2+63\lambda_1^2-16\lambda_1\delta+\delta^2)}
      where $\delta$ is the entire boundary and we have used Mumford's formula $\kappa_1=12\lambda_1-\delta$.
       This expression computes the hyperelliptic class
       correctly in the interior, i.e. on the open set of smooth
fibres in $B$ but to be correct everywhere must be augmented by the appropriate boundary contributions, 
that we proceed to compute via replacing $\be$ by $_\a\be$.
The strategy will be, more or less, to represent
\[_\a\be=((\be_{\vtheta_{2,1}})_{\theta_{2,2}})_{\theta_{1,3}}.\]
This strategy will require some modification due to the double locus of the (1,3) boundary, denoted $\Delta_{1,1}$
by Faber.
Note that, because the boundary locus $\del_{\vtheta_{2,1}}$ has codimension 2, its azimuthal modification does not
'interact' with the others; i.e.  we have
\[p(_\a\be)-p((E_{\theta_{2,2}})_{\theta_{1,3}})=p(\bE_{\vtheta_{2,1}})-p(\bE).\]
Similarly, 
\[p((E_{\theta_{2,2}})_{\theta_{1,3}}-p((E_{\theta_{2,2}})-
(p(E_{\theta_{1,3}})-p(E))\]
is a multiple of $[B(\theta_{2,2}, \theta_{1,3})]$.
\par
We note that, by another formula of Mumford \cite{Mu}, 
we have $\delta_0.\delta_2=(10\lambda_1-2\delta_1).\delta_2$.
Also, $\delta_2^2=-\psi_{\Delta_2}$ (Faber) and 
\eqspl{delta1-squared}
{\delta_1^2=-\psi_{\Delta_1}+2\delta_{1,1}}
by Faber's result together with the fact that $\Delta_1$ has normal crossings and multiplicity 2
along the interior of $\Delta_{11}$.


 \subsubsection{The (2,1) bisep}
 Here we will compute 
 the unique term of noncompact type: viz. the contribution of the locus 
 $B(\vtheta)\subset \Delta_0, \vtheta=\vtheta_{2,1}$
 corresponding to curves $\lf{X}, \rt{X} $ of genera 2, 1 meeting in 2 points (note the latter
 specification determines an orientation on $\vtheta$). This corresponds to 
 an excessive locus in $X\sbc 2._B$.\par
 Note that $\lf{D}\rt{D}$ collapses in the unblown-up Hilbert scheme $X\sbr 2._B$
 while the normal bundles to $\lf{D}, \rt{D}$ are pulled back from there: hence
 $\lf{D}^2\rt{D}=\lf{D}\rt{D}^2=0$. Also, terms containing Chern classes
 of $E, \lf{E}, \rt{E}$ cannot contribute to $[B(\vtheta)]$. 
 From this it is easy to see that the only term
 contributing to the fundamental class $[B(\vtheta)]$ on $B$
 is 
 \[-[\rt{D}^2c_1(\Lambda_1(\omega))\omega_1]=-(\omega_1
 +\omega_2-\Gamma)\omega_1|\O_{\rt{D}}(1)\] 
($\lf{D}^2$ fails to occur because $\rt{E}^0$ has rank 1). The latter
 pushes down to $-2[B(\vtheta)]$ as $\omega$ has degree 2 on $\rt{X}$. 
 For the contribution to the hyperelliptic locus in $B$ one must divide by $6=\deg(\omega)$,
 yielding $(-1/3)[B(\vtheta)]$. In Faber's notation, $[B(\vtheta)]$ is denoted $\gamma_1$. Thus,
 \eqspl{bisep-final}{p(\be_{\vtheta})-p(\be)=\frac{-1}{6}2[B(\vtheta)]=\frac{-1}{6}2\gamma_1.}
 As noted above, because $\del_{\vtheta}$ has codimension 2, the same formula holds with $\be$
 replaced by its azimuthal modification with respect to any collection of seps.


  \subsubsection{The (2,2)   sep}
  Here we will use \eqref{(2,2)-sep}
  \par To compute  $p_1(\be_\theta)-p_1(\be)$ for $\theta=\theta_{2,2}$, 
  note that $\lf{D}\omega_1^2\omega_2=(\lf{\omega}+\lf{\theta})_1^2
  (\lf{\omega}+\lf{\theta})_2$ has base image $(2\lf{g}-1)i_{\del_\theta *}(\lf{\kappa_1}+\lf{\psi})$
  (here $\lf{g}=2$, but it's best to plug this in only at the end). Ditto $\lf{D}\omega_1\omega_2^2$. 
  Similarly, $\lf{D}\omega^2\Gamma$
  has base image $\lf{\kappa}_1+\lf{\psi}$.
  Therefore in all, 
  \eqspl{}{
 p_1(\be_\theta)-p_1(E)=(-1/6)6i_{\del_\theta *}(\lf{\kappa}+\rt{\kappa}+\psi).
 }

For $p_2$, note that base classes like $\lambda_1$ or $\psi$ on $\lf{D}$ ,
when multiplied by $\omega_1\omega_2-\omega\Gamma$, yield a coefficient of
$(2\lf{g}-1)(2\lf{g}-2)=2\binom{2\lf{g}-1}{2}$, which is the fibre degree of $\omega_1\omega_2-\omega\Gamma$ (here $\omega_1^2$ has
fibre degree 0). On the other hand a term like $\lf{\theta}$, which becomes $\lf{\theta}_1+\lf{\theta}_2$
on the ordered version, yields $\lf{\theta}_2(\lf{\omega}+\lf{\theta})_1^2$ 
when multiplied by $\omega_1^2$ and zero when multiplied by 
$\omega_1\omega_2$ or $\omega\Gamma$
(recall that $\omega\theta=0$), therefore it pushes down as above to 
$\lf{\kappa}+\lf{\psi}$. So in all we get:
\eqspl{}{
p_2(\be_\theta)-p_2(\be)=(-1/6)i_{\del_\theta*}(-24\lambda+18\psi+2(\lf{\kappa}+\rt{\kappa}+\psi)).
}


Finally, for $p_3$, it comes from the part involving
$\lf{\theta}\sbr 2.$ times a base class, etc., which yields in total
\eqspl{}{
p_3(\be_\theta)-p_3(\be)=(-1/6)6i_{\del_\theta*}(\lambda).
}

All in all, we get
\eqspl{}{
p(\be_\theta)-p(\be)=(-1/6)i_{\del_\theta*}(26\psi+8(\lf{\kappa}+\rt{\kappa})-30\lambda).
}
Using Faber's notation and Mumford's formula in genus two:
$\kappa_1=2\lambda_1+\delta_1$, which applies to $\lf{\kappa}$
and $\rt{\kappa}$,  this can be written as
\eqspl{(2,2)-final}{
(-1/6)(-14\lambda.\delta_2+8\delta_{2}\delta_1-26\delta_2^2).
}
Because $\delta_0.\delta_2=(10\lambda-2\delta_1).\delta_2$, note that this, combined
with \eqref{genus4-umodified},  already yields $18/6, 27/6$ for the coefficients of $\lambda.
\delta_2, \delta_2^2$ in the hyperelliptic class on $B$, in accord with 
Faber-Pandharipande.
\subsubsection{The (1,3) sep}

Set $\theta(2)=\theta_{2,2}, E=\be_{\theta_{2}}, \theta:=\theta(1):=\theta_{1,3},
\lf{D}=\lf{D}(\theta)$, etc.
Here we will compute $p(E_{\theta_{1,3}})-p(E)$. We will assume temporarily that
the the locus $\del_\theta$ has no double points, i.e. that $\theta$ has at most one point in any fibre.
Over $\del_{\theta_{2,2}, \theta_{1,3}}$ the curve splits as
\[X=X_{1,1}\cup _{\theta_1}X_{1,2}\cup_{\theta_2}X_2\]
with components of genera $1,1,2$ respectively. Let's denote the corresponding
codimension-2 loci in the Hilbert scheme by $D_{1,1}\subset\lf{D}$ and $ D_{1,2},
 D_2\subset\rt{D}$. Note that each of these pushes down to $\delta_2.\delta_1$ and \emph{not} $\delta_{1,1}$.
Using \eqref{(2,2)-sep}, we can write the restriction of $E$ on the (1,3) boundary
as $\rt{E}+\lf{E}$ where, neglecting negligible terms,
\eqspl{}{
c(\rt{E}|_{\lf{D}})&\sim c(\rt{\be})+3[D_{1,1}],\\
c(\lf{E}|_{\rt{D}})&\sim c(\lf{\be})+[D_2].
} Here $c_1(\lf{\be})=\lf{\lambda}$ is just the pullback of the Hodge class on 
$\overline{\mathfrak M}_{1,1}$. 
We have used the obvious fact $\theta(2).D_{1,1}=0$.
Now we compute, setting $\rt{\lambda}=c_1(\rt{\be}),
\psi=\psi(\theta_{1,3})$ :
\eqspl{}{
Q(\lf{D}, \rt{E})&=3-3(\psi+\theta_1\sbr 2.)+2\rt{\lambda}+6D_{1,1}+\theta_1\sbr 2.(2\psi-\rt{\lambda}-3
D_{1,1}),\\
\frac{1}{c(\rt{E})}&=1-\rt{\lambda}-3D_{1,1},\\
\lf{q}:=\frac{Q(\lf{D}, \rt{E})}{c(\rt{E})}&=3-3\psi-3\theta_1-\rt{\lambda}-3D_{1,1}
+\theta_1\sbr 2.(2\psi-\lf{\psi}+2\rt{\lambda}+6D_{1,1}).
}
\eqspl{}{
\rt{E}^0&\sim\rt{E}-\rt{\psi}-D_{1,1},\\
Q(D, \rt{E}^0(\lf{D}))&\sim 2-3(\psi+\lf{\theta}_1\sbr 2.)+\rt{\lambda}-\rt{\psi}+3D_{1,1},\\
\frac{1}{c(\rt{E}^0(\lf{D}))}&\sim 1-\rt{\lambda}+\rt{\psi}-2D_{1,1}+(\psi+\lf{\theta}_1\sbr 2.)
(1+\rt{\lambda}-\rt{\psi}+2D_{1,1}-2\psi),\\
\lf{q}^0:= \frac{Q(D, \rt{E}^0(\lf{D})}{c(\rt{E}^0(\lf{D})}&\sim 2+(\psi+\lf{\theta}_1\sbr 2.)
-2D_{1,1}-\rt{\lambda}+\rt{\psi}+(\psi+\lf{\theta}_1\sbr 2.)(-\rt{\lambda}+\rt{\psi}-2
D_{1,1}).
}
\eqspl{}{
\lf{m}(\theta_1, E)&=i_{\lf{D}*}(\lf{q}+\lf{q}^0-(\psi+\lf{\theta}_1\sbr 2.)\lf{q}\lf{q}^0)\\
&=i_{\lf{D}*}(5-8(\psi+\lf{\theta}_1\sbr 2.)-5D_{1,1}-2\rt{\lambda}+\rt{\psi}
+\lf{\theta}_1\sbr 2.(8\psi+6\rt{\lambda}+16D_{1,1}-2\rt{\psi}-4\lf{\psi})),\\
\rt{m}(\theta_1, E)&=i_{\rt{D}*}(1-\lf{\lambda}-2D_2-2\rt{\theta_2}\sbr 2. D_2).
}
Then
\eqspl{}{
c(E_{\theta_1})=c(E)(1+\lf{m}(\theta_1, E))(1+\rt{m}(\theta_1, E)).
}
Now we can compute as in the previous subsection:
\eqspl{}{
p_1(E_\theta)-p_1(E)=\frac{-1}{6}(5\lf{D}+\rt{D})(\omega_2^2\omega_1+\omega_1^2\omega_2-4\omega^2\Gamma)=
\frac{-1}{6}(-10(\lf{\kappa}+\lf{\psi})+6(\rt{\kappa}+\rt{\psi})).
}
For $p_2$, the portion involving $c_1(E)m_1$ yields $-(24\delta_2+20\lambda)\delta_1$; the portion involving
$m_2$ yields $+8(\lf{\kappa}+\lf{\psi})+(12\delta_2+20\lf{\lambda})\delta_1$ so, considering
that $\lf{\kappa}=0$ and $\lf{\lambda}=\lf{\psi}$  in genus 1, we have in total,
\eqspl{}{
p_2(E_\theta)-p_2(E)=\frac{-1}{6}(-12\delta_2-20\lambda_1+28\lf{\psi})\delta_1.
} 
For $p_3$, the part coming from $c_2(E)m_1$ yields $-10\delta_1\delta_2$ (arising from
$-2\theta(2).(\lf{D}(2)+\rt{D}(2))$).
The part coming from $c_1(E)m_2$, yields: $-8\lambda$ from $\lambda(-8\theta(1).\lf{D}(1))$, 
$-24\delta_2\delta_1$ arising from $3(\lf{D}+\rt{D}(2)).(-8\theta(1).\lf{D}(1)$,
and finally $18\delta_1\delta_2$ arising from $3\lf{D}(2).2D_2$; total: $(-8\lambda-6\delta_2).\delta_1$. Finally, the part coming from $m_3$ yields:$(8\psi+6\rt{\lambda}-2\rt{\psi}+16\delta_2).\delta_1$ (arising from multiples of $\theta(1).\lf{D}$ and $-6\delta_2\delta_1$ (arising from $-2D_2\theta_2$). Thus in total:
\eqspl{}{
p_3(E_\theta)-p_3(E)=\frac{-1}{6}\delta_1.(-6\delta_2-8\lambda+8\psi+6\rt{\lambda}-2\rt{\psi}).
}
Summing up and using \[\rt{\kappa}=12\rt{\lambda}-\rt{\delta_0}-\rt{\delta_1}=
12\lambda.\delta_1-12\lf{\lambda}-\rt{\delta_0}-\delta_2.\delta_1-\delta_{1,1}=(12\lambda-\delta_0-\delta_2).\delta_1-\delta_{1,1},\]
we get
\eqspl{(1,3)-simple}{
p(E_\theta)-p(E)=\frac{-1}{6}(4\lf{\psi}+12\psi+50\lambda-6\delta_0-24\delta_2-6\delta_{1,1})_{\delta_1}.
}

\subsubsection{The (1,2,1) improper bisep} To complete the calculation we must compute the
contribution from the locus $\del_{1,2,1}$ corresponding to the improper bisep of type (1,2,1). This is denoted by 
$\Delta_{1,1}$ by Faber and equals the double locus of the boundary divisor $\Delta_1$. 
Note that in \eqref{fab-pan-form}, rewritten via \eqref{delta1-squared}, 
the total coefficient of this cycle is $(1/6)30$ where 30 comes from $26=2.13$ 
from the $\delta_1^2$ plus 4.
In our azimuthal approach,
this locus makes
an additional contribution stemming from the fact that the azimuthal bundle is modified
there twice, once from each of the seps of type (1,3). To compute this contribution, we may assume
by a suitable base change that the global $(1,3)$ sep splits in two components, 
say $\theta', \theta"$ so that the (1,2,1) bisep locus is just
$\del_{\theta', \theta"}=\del_{\theta'}\cap\del_{\theta"}$ and the curve over this locus splits as
\[X_{\theta', \theta"}=X'(1)\cup_{\theta'}X(2)\cup_{\theta"}X"(1)\]
(respective genera 1,2,1). Denote the corredponding
Hilbert loci of schemes contained in the respective curve 
by $D'(1), D(2), D"(1)$.  We then calculate the $\del_{1,2,1}$ term in
\eqspl{}{(&1+\lf{m}(\theta', E)+\rt{m}(\theta', E))(1+\lf{m}(\theta", E_{\theta'})+\rt{m}(\theta", E_{\theta'})).}

The first factor is analogous to the above for $\theta_1$; the second is likewise analogous, except for the additional
term of $+6\theta".D(1)$. Then, neglecting terms that do not affect
the coefficient of $\delta_{1,1}$, the above product becomes
\eqspl{}{
&(1+5\lf{D}(\theta')-8\theta'\lf{D}(\theta')+\rt{D}(\theta')).(1+5\lf{D}(\theta")-8\theta"\lf{D}(\theta")+\rt{D}(\theta"))\sim\\
&-8\theta'.D'(1)-2\theta"D"(1)+D(2).
} This contributes a term of 
\eqspl{(1,3)-double}{(-1/6)(-22)\del_{1,1}} to $p(_\a E)$.
Adding this to $(-1/6)(-6\delta_{1,1})$ from \eqref{(1,3)-simple} and $(-1/6)(-2)\delta_{1,1}$ from
\eqref{genus4-umodified}, we get a total $\delta_{1,1}$ term of $(-1/6)(-30)\delta_{1,1}=(1/6)(13.2+4)\delta_{1,1}$
matching \eqref{fab-pan-form}.
\subsubsection{Finale} Now adding together \eqref{genus4-umodified},  \eqref{bisep-final}, \eqref{(2,2)-final}, \eqref{(1,3)-simple} and
\eqref{(1,3)-double} we recover \eqref{fab-pan-form}.\qed

\newsubsection{W-bundles}\lbl{w-bundles-sec}

As discussed in \S\ref{normal-blowup-sec}, the exceptional divisors in an
$\mathcal S$-stratified blowup have a $W_n$-bundle structure
and the quantitative, enumerative aspect of
this structure is involved in computing azimuthal modifications
corresponding to biseps (see \eqref{azi-bundle-bisep}). Our purpose here
is to study $W_n$-bundle structures and their intersection theory generally.\par
Let $\L=(L_0, ...,L_n)$ be a collection of line bundles
on a variety $Z$, and set, for $I\subset\{0,...,n\}, |I|>1,$
\eqsp{L_+:=\bigoplus \limits_{i=0}^n L_i
\supset L_{I}=\bigoplus\limits_{i\in I}L_i,\\
P_+=\P(L_+),\  P_{I}=\P(L_I^*), \ P=P_+\times\prod\limits_{2\leq |I|\leq n}P_{I}\\
} (all products relative over $Z$)
and let $Q_I$ be the tautological quotient bundle of rank $|I|-1$ of $L_I^* $, i.e.
the dual of the tautological subbundle on $P_I$ (pulled back to $P$). 
Then on $P$, we have  natural composite maps
\[\O_{P_+}(-1)\to L_+^*\to L^*_{I}\to Q_{I}\]
which together induce
\eqspl{W'-zero-locus}{\O_{P_+}(-1)\to\bigoplus_{2\leq |I|\leq n}Q_{I}.}
The zero locus of the latter map
 will be temporarily  denoted by $W'$.
 It consists of collections $(h; (h_{I}))$ where
 $h$ is a hyperplane in $L_+$ and $h_{I}$ is a hyperplane
 in $L_{I}$ contained in $h$.
The image of $W'$ in $\prod\limits_{i<j}P_{i,j}$
is the locus denoted earlier (see \S\ref{normal-blowup-sec}, \eqref{W[L]-def-eq})
by
 $W[\L]=W[L_0,...,L_n]$.
Note that the closure of the graph of a linear projection $P_+\dashrightarrow P_{I}$ can be identified with the blowup of
$\P(L/L_{I})\subset P_+$. Applying this
to the components of the projection of $W'$ to $\prod\limits_{|I|=2} P_{I}$,
it follows that
the image of the projection
can be identified with $W[\L]$, i.e. the normal blowup, hence
 by Proposition \ref{stratified-is-normal-prop}, also
 with the $S$-stratified blowup of $P_+$ corresponding to the
stratification by coordinate planes, which is smooth.

\begin{lem}
(i) $ W'$ projects isomorphically to its image  $W[\L]\subset\prod\limits_{|I|=2}P_{I}$;\par
(ii) this image is equal
 to the degeneracy ( rank $\leq n$) locus of the natural map
 \[\bigoplus\limits_{|I|=2}Q^*_{I}\to L_+\]
 and this map has rank at least $n$ everywhere.\par
\end{lem}
\begin{proof} By smoothness of $W[\L]$,
it suffices to prove (ii) plus the bijectiveness part of (i).
This in turn is
a consequence of following elementary fact: given a collection of points $A=(A_{i,j}\in P_{i,j})$, one on each line of the 1-skeleton of the coordinate simplex in $\P^n$, they span \emph{at least} a hyperplane. The proof is by induction on $n$, the case $n=2$ being obvious. Assuming the case for $n$, consider the case of $n+1$.
If those among the  $A_{i,j}$ that lie on $\P^n$, i.e. those
with $i,j\leq n$,
span $\P^n$, there is nothing to prove. Else, by induction those points lie
on a unique hyperplane $H\subset\P^n$. Then there exists $k$
such that the $k$-th coordinate point $e_k\in\P^n\setminus H$.
But then $P_{k,n+1}\cap H=\emptyset$ so $A_{k,n+1}\not\in H$,
therefore $A$ spans at least a hyperplane in $\P^{n+1}$.\end{proof}

Now set $q_+=c_1(\O_{P_+(1)}), q_{I}=c_1(\O_{P_I}(1))$. Then 
because $W'$ is the zero locus of the map \eqref{W'-zero-locus}, the fundamental class of $W'$ on $P$ is
\eqspl{}{ [W']=\prod\limits_{\substack{I\\ 2\leq |I|\leq n}}c_{|I|-1}(Q_I(q_+))
=\prod\limits_{\substack{I\\ 2\leq |I|\leq n}}[\frac{\prod\limits_{i\in I}(1-[L_i]+q_+)}{1-q_I+q_+}]_{|I|-1}.
}
This formula easily allows us to compute intersections of standard classes on $W$, as they
are pulled back via $W'\hookrightarrow P$:
\begin{prop}We have
\eqspl{product-on-W}{
\pi_{(W[\L]\to Z)*}(q_+^m\prod q_{I}^{m_{I}})
=\pi_{(P\to Z)*}(q_+^m\prod\limits_{\substack{I\\ 2\leq |I|\leq n}}q_I^{m_I}[\frac{\prod\limits_{i\in I}(1-[L_i]+q_+)}{1-q_I+q_+}]_{|I|-1}).
}
\end{prop}
\begin{rem}
The argument of $\pi_{(P\to Z)*}$ in \eqref{product-on-W} is a polynomial in $q_+$ and the
$q_I$.
Note that (where $s.$ denotes Segre class)
\eqsp{\pi_{(P\to Z)*}(q_+^m \prod\limits_I q_I^{m_I})=s_{m-n}(L_+)\prod\limits_Is_{m_I-|I|+1}(L_I)\\
=[\prod\limits_{i=0}^n\frac{1}{1+[L_i]}]_{m-n}
\prod\limits_I[\prod\limits_{i\in I}\frac{1}{1+[L_i]}]_{m_I-|I|+1}
.}
Consequently, the LHS of \eqref{product-on-W} can be computed as
a polynomial in the $c_1(L_i)$.
\end{rem}
By Theorem \ref{azi-hilb-thm} and Remark \ref{azi-hilb-rem}, the classes appearing in \eqref{product-on-W} are precisely what is needed to
compute powers of the exceptional divisors on the azimuthal Hilbert scheme. Indeed  in the notation
of that Theorem, we have, identifying $\Theta$ with $\{0,...,n\}$:
\[c_1(L_\vtheta(\Theta))=q_+, \  c_1(L_\vtheta(\Theta, \Theta'))=q_{\Theta'}.\] 
The self-intersection of $\Xi_\vtheta(\Theta)$ can be computed using \eqref{conormal-Xi} and Remark \ref{azi-hilb-rem}.

\pagestyle{empty}
\bibliographystyle{amsplain}
\bibliography{mybib}
\end{document}